\documentclass[11pt,a4paper]{article}
\usepackage[utf8]{inputenc}
\usepackage[T1]{fontenc}
\usepackage{amsmath, amssymb, amsthm}
\usepackage{mathrsfs}
\usepackage{graphicx}
\usepackage{hyperref}
\usepackage{enumitem}
\usepackage[margin=2.5cm]{geometry}
\usepackage{bm}
\usepackage{booktabs}
\usepackage{longtable}
\usepackage{algorithm}
\usepackage{algorithmicx}
\usepackage{algpseudocode}

\newtheorem{theorem}{Theorem}[section]
\newtheorem{lemma}[theorem]{Lemma}
\newtheorem{proposition}[theorem]{Proposition}
\newtheorem{corollary}[theorem]{Corollary}
\theoremstyle{definition}
\newtheorem{definition}[theorem]{Definition}
\newtheorem{remark}[theorem]{Remark}

\newtheorem{hypothesis}[theorem]{Hypothesis}

\newcommand{\Om}{\Omega}
\newcommand{\Qb}{\overline Q}
\newcommand{\Da}{D_\alpha}
\newcommand{\fa}{f_\alpha}
\newcommand{\eqdef}{:=}
\newcommand{\R}{\mathbb{R}}

\newcommand{\PP}{\mathcal{P}}
\newcommand{\EE}{\mathcal{E}}
\newcommand{\dd}{\mathrm{d}}
\newcommand{\argmin}{\operatorname{argmin}}

\newcommand{\dalpha}{D_\alpha}
\DeclareMathOperator{\Div}{div}

\title{A new Geometric Setting for the Analysis of Partial Differential Equations\\
	\large }
\author{Ambroise Soglo$^1$, Koffi W. Hou\'edanou$^{2,*}$, Jamal Adetola$^3$ and Marcos Aboubacar$^1$\\
	\texttt{$^1$IMSP-UAC, Bénin \quad $^2$FAST-UAC, Bénin\quad $^3$ UNSTIM-Abomey, Bénin}}
\date{\today}

\begin{document}
	
	\maketitle
	
	\begin{abstract}
		We introduce a hybrid metric geometry on the space of absolutely continuous probability densities that combines optimal transport (Wasserstein geometry) and log-ratio composition (Aitchison geometry). The hybrid distance $D_\alpha$ is defined through a Benamou--Brenier-type dynamical formulation that couples spatial transport with a centered reaction term preserving total mass.
		We prove that $D_\alpha$ is a genuine metric and establish comparison estimates with the Wasserstein and Aitchison distances. In particular, we show that the topology induced by $D_\alpha$ is stronger than the narrow topology and weaker than the supremum topology generated by the Wasserstein and Aitchison metrics. We further prove that the metric space is geodesic.
		Within this framework, we develop the foundations of a gradient flow theory in the sense of Ambrosio--Gigli--Savaré, including the characterization of absolutely continuous curves, metric derivatives, metric slopes, and formal Jordan--Kinderlehrer--Otto schemes. We also investigate hybrid barycenters and their connections with Wasserstein barycenters and Aitchison barycenters.
		Finally, we discuss several partial differential equations, including logistic diffusion, Allen--Cahn equations with log-ratio constraints, and chemotaxis models with logarithmic growth, as formal gradient flows associated with the hybrid geometry, and compare the proposed framework with the Wasserstein--Fisher--Rao metric.
	\end{abstract}
	\footnote{\textit{E-mails: A. Soglo (ambroiso.soglo@gmail.com), K.W. Houédanou (khouedanou@yahoo.fr, corresponding author), J. Adetola (adetolajamal@unstim.bj) and M. Aboubacar (marcosaboubacar@imsp-uac.org)}}
	\section{Introduction}
	\label{sec:intro}
	
	\subsection{The dual nature of evolution in density spaces}
	
	The evolution of probability densities is a cornerstone of modern applied mathematics, arising in statistical physics, population dynamics, chemical kinetics, image analysis, and machine learning. Two fundamental and often interacting mechanisms govern these evolutions.
	
	\begin{enumerate}[label=(\roman*)]
		
		\item \textbf{Spatial transport.}
		Particles, individuals, or mass move from one location to another. This mechanism underlies diffusion, convection, migration, and more generally any process in which the spatial distribution evolves while preserving total mass.
		The mathematical theory of optimal transport, and in particular the quadratic Wasserstein distance $W_2$, provides a natural geometric framework for such phenomena \cite{Villani,AGS}. Endowed with the Wasserstein metric, the space of probability measures $(\PP_2(\Omega),W_2)$ is a geodesic metric space, and gradient flows of suitable functionals recover many classical partial differential equations, including the heat equation, the porous medium equation, and the Fokker--Planck equation \cite{Otto}. In this setting, absolutely continuous curves are characterized by the continuity equation
		$
		\partial_t\rho+\operatorname{div}(\rho v)=0,
		$
		where $v$ is a velocity field.
		\item \textbf{Relative rearrangement.}
		In compositional systems, such as mixtures of chemical species, ecological populations, or probability vectors, the relevant information lies not in absolute quantities but in relative proportions.
		
		The standard framework for such data is Aitchison geometry on the simplex \cite{Aitchison,EGT}. Its central principle is that only log-ratios between components are meaningful, since they are invariant under global rescaling. This leads to a Euclidean structure after applying the centered log-ratio transformation. In the continuous setting, infinitesimal compositional changes are represented by scalar fields $u$ satisfying the centering condition
		$
		\int_\Omega u\,\rho\,dx=0,
		$
		which ensures conservation of total mass while allowing relative proportions to evolve.
		
	\end{enumerate}
	
	Many systems involve both mechanisms simultaneously. Examples include reacting fluids, migrating populations subject to selection, and spatially distributed chemical mixtures. Despite their ubiquity, there is currently no unified metric framework combining transport and compositional rearrangement while preserving total mass.
	
	\subsection{Limitations of existing approaches}
	
	Several approaches have been proposed to combine transport and reaction mechanisms. The most prominent example is the \textbf{Wasserstein--Fisher--Rao (WFR) metric} \cite{LMS}, also known as the Hellinger--Kantorovich distance.
		The WFR framework interpolates between optimal transport and Fisher--Rao geometry, but it presents several limitations in the context of compositional systems.
	
	\begin{enumerate}[label=(\alph*)]
		
		\item \textbf{Mass variation.}
		WFR allows the total mass to vary along admissible curves. While this is appropriate for systems involving creation or destruction of mass, it is not suitable for probability distributions, closed chemical systems, or compositional data, where total mass is conserved.
		
		\item \textbf{Square-root versus log-ratio geometry.}
		The Fisher--Rao metric relies on the transformation $\sqrt{\rho}$, whereas Aitchison geometry is based on centered log-ratios. For compositional data analysis, the log-ratio representation is preferred because it captures relative information and is invariant under multiplicative scaling.
		
		\item \textbf{Absence of a centering constraint.}
		The reaction term in the WFR framework is unconstrained and therefore allows arbitrary mass creation and destruction. In contrast, many applications require the infinitesimal conservation law
		$
		\int_\Omega u\,\rho\,dx=0,
		$
		which is the continuous counterpart of the Aitchison principle.
		
	\end{enumerate}
	
	Classical reaction--diffusion equations combine transport and reaction mechanisms, but they generally lack a unified variational interpretation as gradient flows in a metric space adapted to compositional constraints. Such a perspective is valuable for studying stability, long-time behavior, and variational discretization schemes.
	
	\subsection{The proposed hybrid geometry}
	
	To address these limitations, we introduce a family of distances $D_\alpha$, parameterized by $\alpha\in(0,1)$, on the space $\PP_{2,\mathrm{ac}}(\Omega)$ of absolutely continuous probability measures with finite second moment.
	
	Inspired by the Benamou--Brenier dynamical formulation of optimal transport \cite{BenamouBrenier}, we define $D_\alpha$ through the minimization of the action functional over admissible triples $(\rho_t,v_t,u_t)$ satisfying the continuity equation with a centered source term
	
	\[
	\partial_t\rho_t+\operatorname{div}(\rho_t v_t)
	=
	\rho_t u_t,
	\qquad
	\int_\Omega u_t\,\rho_t\,dx=0.
	\]
	
	The associated action is given by
	
	\[
	\mathcal A_\alpha(\rho,v,u)
	=
	\int_0^1
	\left[
	(1-\alpha)
	\int_\Omega |v_t|^2\rho_t\,dx
	+
	\alpha
	\int_\Omega u_t^2\rho_t\,dx
	\right]dt.
	\]
	
	The parameter $\alpha$ balances the transport and compositional contributions. Formally, the limiting cases correspond to the pure geometries:
	
	\[
	\alpha=0
	\quad\Longrightarrow\quad
	u_t\equiv0,
	\]
	
	recovering the quadratic Wasserstein distance, and
	
	\[
	\alpha=1
	\quad\Longrightarrow\quad
	v_t\equiv0,
	\]
	
	recovering the Aitchison distance.
	
	For $\alpha\in(0,1)$, the metric $D_\alpha$ interpolates between these two structures while preserving total mass through the centering constraint.
	
	We prove that $D_\alpha$ defines a genuine metric on $\PP_{2,\mathrm{ac}}(\Omega)$ and establish the comparison estimate
	
	\[
	D_\alpha^2(\rho,\sigma)
	\le
	\frac12\max\{1-\alpha,\alpha\}
	\left(
	W_2^2(\rho,\sigma)
	+
	d_A^2(\rho,\sigma)
	\right).
	\]
	
	As a consequence, the topology induced by $D_\alpha$ satisfies
	
	\[
	\tau_{\mathrm{narrow}}
	\subset
	\tau_{D_\alpha}
	\subset
	\tau_{W_2}\vee\tau_{d_A}.
	\]
	
	Furthermore, we show that
	$(\PP_{2,\mathrm{ac}}(\Omega),D_\alpha)$
	is a geodesic metric space and investigate the associated gradient flow structure in the sense of Ambrosio, Gigli, and Savaré under suitable compactness assumptions.
	\subsection{Main contributions}
	
	The main contributions of this paper are as follows.
	
	\begin{itemize}
		
		\item \textbf{Construction and metric properties of $D_\alpha$.}
		We introduce a hybrid dynamical distance $D_\alpha$ on
		$\PP_{2,\mathrm{ac}}(\Omega)$
		through a Benamou--Brenier-type formulation coupling transport and centered reaction mechanisms.
		We prove the existence of minimizers in the variational formulation and establish that $D_\alpha$ defines a genuine metric.
		
		Moreover, we prove the comparison estimate
		
		\begin{equation}
			D_\alpha^2(\rho,\sigma)
			\le
			\frac12\max\{1-\alpha,\alpha\}
			\Bigl(
			W_2^2(\rho,\sigma)
			+
			d_A^2(\rho,\sigma)
			\Bigr).
		\end{equation}
		
		As a consequence, the topology induced by $D_\alpha$ satisfies
		
		\begin{equation}
			\tau_{\mathrm{narrow}}
			\subset
			\tau_{D_\alpha}
			\subset
			\tau_{W_2}\vee\tau_{d_A}.
		\end{equation}
		
		\item \textbf{Geodesic structure.}
		We prove that
		$(\PP_{2,\mathrm{ac}}(\Omega),D_\alpha)$
		is a geodesic metric space: any two densities can be connected by a constant-speed geodesic.
		We further characterize absolutely continuous curves through the continuity equation with centered source term and identify their metric derivative.
		
		\item \textbf{Gradient flows in the AGS framework.}
		Following the theory of Ambrosio, Gigli, and Savaré \cite{AGS}, we develop the foundations of a gradient flow theory in the hybrid metric setting.
		We define metric slopes, study geodesic $\lambda$-convexity, formulate the Jordan--Kinderlehrer--Otto scheme, and establish existence and convergence results under suitable compactness assumptions on the sublevels of the energy functional.
		
		\item \textbf{Hybrid barycenters.}
		We introduce hybrid barycenters as minimizers of
		
		\[
		\sum_{i=1}^n w_i D_\alpha^2(\rho,\rho_i).
		\]
		
		We establish existence results under suitable coercivity and compactness assumptions and prove uniqueness under strict geodesic convexity.
		We show that, formally, the construction interpolates between Wasserstein barycenters \cite{Agueh} and Aitchison barycenters.
		
		\item \textbf{Applications to partial differential equations.}
		We identify several evolution equations as formal gradient flows associated with the hybrid geometry, including logistic diffusion, Allen--Cahn equations with log-ratio constraints, and Keller--Segel models with logarithmic growth.
		
		\item \textbf{Comparison with Wasserstein--Fisher--Rao geometry.}
		We provide a detailed comparison with the Wasserstein--Fisher--Rao framework and show that the centering condition
		
		\[
		\int_\Omega u\,\rho\,dx = 0
		\]
		
		is the key structural feature ensuring mass conservation and making the hybrid geometry suitable for compositional systems.
		
	\end{itemize}
	
	\subsection{Plan of the paper}
	The contents of this paper are organised as follows.
	%The paper is organized as follows.
	
	Section~\ref{sec:prelim} reviews the necessary background on quadratic Wasserstein geometry and develops the continuous counterpart of Aitchison geometry. In particular, we establish the Hilbertian structure of the continuous Aitchison space, characterize its geodesics, and describe absolutely continuous curves together with their metric derivatives.
	
	Section~\ref{sec:hybrid_distance} introduces the hybrid distance $D_\alpha$ through a Benamou--Brenier-type dynamical formulation coupling transport and centered reaction mechanisms. We establish the lower semicontinuity of the action functional, prove the existence of minimizers, show that $D_\alpha$ defines a metric on $\PP_{2,\mathrm{ac}}(\Omega)$, and derive comparison estimates with the Wasserstein and Aitchison distances together with the corresponding topological properties.
	
	Section~\ref{sec:geodesics} is devoted to the geometry of the hybrid metric space. We prove the existence of constant-speed geodesics, characterize absolutely continuous curves through the continuity equation with a centered source term, identify the minimal admissible velocity--reaction pair, and derive the metric derivative formula.
	
	Section~\ref{sec:duality} develops a Fenchel--Rockafellar duality theory for the hybrid action functional. We formulate the primal and dual optimization problems, establish weak and strong duality, derive the corresponding optimality conditions, and obtain an Eulerian characterization of minimizing curves.
	
	Section~\ref{sec:convexity} investigates the geodesic convexity of several fundamental energy functionals along $D_\alpha$-geodesics. In particular, we study the potential energy, the Boltzmann entropy, the Rényi entropy, and interaction energies, thereby providing the variational framework underlying the subsequent gradient-flow analysis.
	
	Section~\ref{sec:ags} develops the gradient flow theory in the sense of Ambrosio--Gigli--Savaré. We introduce the metric slope, study curves of maximal slope, analyze geodesically convex functionals, formulate the Jordan--Kinderlehrer--Otto minimizing movement scheme, and establish convergence results under suitable compactness and convexity assumptions.
	
	Section~\ref{sec:barycenters} studies hybrid Wasserstein--Aitchison barycenters. We prove existence and uniqueness results under appropriate convexity assumptions and investigate their relationships with the classical Wasserstein and Aitchison barycenters.
	
	Section~\ref{sec:examples} presents several evolution equations arising naturally as formal gradient flows in the hybrid geometry, including logistic diffusion equations, Allen--Cahn equations with compositional constraints, and Keller--Segel--type models.
	
	Section~\ref{sec:comparison} compares the proposed framework with the Wasserstein--Fisher--Rao geometry, emphasizing the role of the centered reaction constraint, the preservation of total mass, and the fundamental geometric differences between the two transport--reaction metrics.
	
	In Section~\ref{sec:numerics}, we investigate the behavior of the proposed hybrid Wasserstein--Aitchison geometry for image interpolation problems.
	 
	Finally, Section~\ref{sec:perspectives} discusses several open problems and future research directions. We outline possible developments concerning synthetic lower Ricci curvature bounds, stronger geodesic convexity results, the complete EVI theory for hybrid gradient flows, numerical approximation through hybrid JKO schemes, extensions to more general transport--reaction metric structures, and further applications to compositional data analysis and nonlinear evolution equations.
	
	%Technical proofs and complementary results are collected in the appendices.
	\section{Extended preliminaries}
	\label{sec:prelim}
	\subsection{Functional framework}
	\label{sec:functional}
	
	\subsubsection{The spatial domain}
	Let $\Omega\subset\R^d$ be a bounded, connected, open domain with Lipschitz boundary. The boundedness of $\Omega$ is essential for the compactness arguments that follow (via Prokhorov's theorem and the Banach--Alaoglu theorem).
	
	\subsubsection{The space of probability densities}
	Denote by $\PP_{\mathrm{ac}}(\Omega)$ the set of probability measures that are absolutely continuous with respect to the Lebesgue measure, with strictly positive density $\rho\in L^1(\Omega)$, $\rho>0$ a.e., and $\int_\Omega\rho\dd x = 1$. Strict positivity is required for the logarithmic transformations; densities that vanish on sets of positive measure can be approximated by positive ones, and the theory extends by density.
	
	\subsubsection{Log-square-integrable densities}
	For the Aitchison geometry to be well-defined, we need the logarithm of the density to be square-integrable. Define
	\begin{equation}
		\PP_{\log,2}(\Omega) = \left\{ \rho\in\PP_{\mathrm{ac}}(\Omega) : \int_\Omega (\ln\rho(x))^2\dd x < \infty \right\}.
	\end{equation}
	This space is the natural domain for the continuous Aitchison distance. The condition $\ln\rho\in L^2(\Omega)$ implies in particular that $\ln\rho\in L^1(\Omega)$ (on a bounded domain, $L^2\subset L^1$), so the log-mean $\ell(\rho)=\int_\Omega\ln\rho\dd x$ is well-defined and finite.
	
	\subsubsection{Densities with finite second moment}
	For the Wasserstein distance to be finite, we need the second moment to be finite. Define
	\begin{equation}
		\PP_{2}(\Omega) = \left\{ \rho\in\PP_{\mathrm{ac}}(\Omega) : \int_\Omega |x|^2\rho(x)\dd x < \infty \right\}.
	\end{equation}
	The Wasserstein distance $W_2$ is finite exactly on this space.
	
	\subsubsection{The working space}
	Our hybrid geometry requires both conditions simultaneously. Define
	\begin{equation}
		\PP_{2,\log,2}(\Omega) = \PP_2(\Omega) \cap \PP_{\log,2}(\Omega) = \left\{ \rho\in\PP_{\mathrm{ac}}(\Omega) : \int_\Omega |x|^2\rho\dd x < \infty,\ \int_\Omega (\ln\rho)^2\dd x < \infty \right\}.
	\end{equation}
	This is the space on which both $W_2$ and $d_A$ are finite. For simplicity, we denote this space by $\PP_{2,\mathrm{ac}}(\Omega)$ throughout the paper, with the understanding that the logarithmic condition is implicitly assumed.

	\subsection{Wasserstein geometry: detailed review}
	For $\mu,\nu\in\PP_2(\Omega)$, the Wasserstein distance of order $2$ is defined by
	\begin{equation}
		W_2^2(\mu,\nu)=\inf_{\pi\in\Pi(\mu,\nu)}\int_{\Omega\times\Omega}|x-y|^2\dd\pi(x,y),
	\end{equation}
	where $\Pi(\mu,\nu)$ denotes the set of couplings (joint measures with marginals $\mu$ and $\nu$). The following properties are standard \cite{Villani,AGS}:
	\begin{itemize}
		\item $(\PP_2(\Omega),W_2)$ is a complete, separable, geodesic metric space.
		\item Convergence $W_2(\mu_n,\mu)\to0$ is equivalent to narrow convergence plus convergence of second moments: $\mu_n\rightharpoonup\mu$ and $\int|x|^2\dd\mu_n\to\int|x|^2\dd\mu$.
		\item A curve $\rho_t$ is absolutely continuous with respect to $W_2$ if and only if there exists a velocity field $v_t$ such that the continuity equation $\partial_t\rho_t + \Div(\rho_t v_t)=0$ holds in the sense of distributions, and $\int_0^1\int|v_t|^2\rho_t\dd x\dd t<\infty$. In that case, the metric derivative is $|\dot\rho_t|_{W_2}^2 = \int|v_t|^2\rho_t\dd x$.
		\item The formal Riemannian structure (Otto calculus) identifies the tangent space at $\rho$ with $L^2(\rho)$ of gradients, and the geodesic equations are given by the pressureless Euler system.
	\end{itemize}
	
	\subsection{Aitchison geometry: from discrete to continuous}
	
	\subsubsection{Discrete Aitchison distance}
	On the simplex $\Delta^{k-1}=\{p\in\R^k_+: \sum_{i=1}^k p_i = 1\}$, the Aitchison distance is
	\begin{equation}
		d_A^2(p,q)=\frac{1}{k}\sum_{i<j}\Bigl(\ln\frac{p_i}{p_j}-\ln\frac{q_i}{q_j}\Bigr)^2 = \sum_{i=1}^k\Bigl(\ln\frac{p_i}{g(p)}-\ln\frac{q_i}{g(q)}\Bigr)^2,
	\end{equation}
	where $g(p)=(\prod_{i=1}^k p_i)^{1/k}$ is the geometric mean. This distance is invariant under scaling.
	
	\subsubsection{Continuous extension}
	
	Let \(\Omega\subset\R^d\) be a bounded domain. Denote by \(L^2_0(\Omega)\) the closed subspace of \(L^2(\Omega)\) consisting of functions with zero mean:
	\begin{equation}
		L^2_0(\Omega) = \left\{ u\in L^2(\Omega) : \int_\Omega u(x)\,\dd x = 0 \right\}.
	\end{equation}
	
	For \(\rho\in\PP_{\mathrm{ac}}(\Omega)\), define the log-mean \(\ell(\rho)=\int_\Omega\ln\rho\dd x\) (finite when \(\ln\rho\in L^1(\Omega)\)) and the centered log-ratio transform
	\begin{equation}
		\widetilde{\rho} = \ln\rho - \ell(\rho).
	\end{equation}
	Note that \(\int_\Omega\widetilde{\rho}\dd x = 0\) by construction.
	
	\begin{definition}
		The space of log-square-integrable densities is
		\begin{equation}
			\PP_{\log,2}(\Omega)=\left\{\rho>0,\ \int_\Omega\rho=1,\ \ln\rho\in L^2(\Omega)\right\}.
		\end{equation}
		For \(\rho,\sigma\in\PP_{\log,2}(\Omega)\), the continuous Aitchison distance is
		\begin{equation}
			d_A^2(\rho,\sigma)=\int_\Omega\bigl(\widetilde{\rho}(x)-\widetilde{\sigma}(x)\bigr)^2\dd x.
		\end{equation}
	\end{definition}
	
	\begin{lemma}[Isometry]
		\label{lem:isometry}
		The map \(\Phi:\PP_{\log,2}(\Omega)\to L^2_0(\Omega)\) defined by \(\Phi(\rho)=\widetilde{\rho}\) is injective and satisfies
		\begin{equation}
			d_A(\rho,\sigma) = \|\Phi(\rho)-\Phi(\sigma)\|_{L^2_0}
		\end{equation}
		for all \(\rho,\sigma\in\PP_{\log,2}(\Omega)\). Hence \(\Phi\) is an isometry.
	\end{lemma}
	\begin{proof}
		The equality \(d_A(\rho,\sigma)=\|\widetilde{\rho}-\widetilde{\sigma}\|_{L^2}\) is exactly the definition of \(d_A\). Injectivity: if \(\widetilde{\rho}=\widetilde{\sigma}\), then \(\ln\rho-\ln\sigma\) is constant almost everywhere. Let this constant be \(c\). Then \(\rho = e^c \sigma\). Since both \(\rho\) and \(\sigma\) integrate to \(1\), we have \(e^c=1\), hence \(c=0\) and \(\rho=\sigma\) a.e.
	\end{proof}
	
	\begin{lemma}[Characterization of the image]
		\label{lem:image}
		The image of \(\Phi\) is
		\begin{equation}
			\mathcal{H} = \left\{ u\in L^2_0(\Omega) : e^u \in L^1(\Omega) \right\}.
		\end{equation}
		Moreover, \(\mathcal{H}\) is a convex subset of \(L^2_0(\Omega)\).
	\end{lemma}
	\begin{proof}
		If \(u=\widetilde{\rho}\), then \(\rho = e^u / \int e^u\). Since \(\rho\in L^1\), we have \(e^u\in L^1\). Conversely, if \(u\in L^2_0\) and \(e^u\in L^1\), define \(\rho = e^u / \int e^u\). Then \(\rho>0\), \(\int\rho=1\), and \(\ln\rho = u - \ln\int e^u\). Since \(u\in L^2\) and the constant term is bounded, \(\ln\rho\in L^2\). Moreover, \(\ell(\rho)=\int u - \ln\int e^u = -\ln\int e^u\), so \(\widetilde{\rho}=u\). Thus \(\Phi(\rho)=u\). Convexity follows from Hölder's inequality: for \(u,v\in\mathcal{H}\) and \(t\in[0,1]\),
		\begin{equation}
			\int_{\Omega} e^{(1-t)u+tv} dx \le \left(\int_{\Omega} e^u\right)^{1-t}\left(\int_{\Omega} e^v\right)^t < \infty.
		\end{equation}
	\end{proof}
	
	\begin{remark}
		The space \(\mathcal{H}\) is not closed in \(L^2_0(\Omega)\); it is a convex subset. Consequently, \((\PP_{\log,2}(\Omega),d_A)\) is isometric to \(\mathcal{H}\) equipped with the \(L^2\) distance. This is the natural infinite-dimensional analogue of the discrete Aitchison geometry on the simplex. For a detailed treatment of such spaces, see \cite{Egozcue2006, van den Boogaart2010}.
	\end{remark}
	
	\begin{lemma}[Geodesics]
		\label{lem:geodesics_aitchison}
		The space \((\PP_{\log,2}(\Omega),d_A)\) is geodesic. For any \(\rho_0,\rho_1\in\PP_{\log,2}(\Omega)\), the unique constant-speed geodesic is given by
		\begin{equation}
			\widetilde{\rho}_t = (1-t)\widetilde{\rho}_0 + t\widetilde{\rho}_1,\qquad
			\rho_t = \frac{\exp(\widetilde{\rho}_t)}{\int_\Omega\exp(\widetilde{\rho}_t)\dd x}.
		\end{equation}
		Equivalently,
		\begin{equation}\
			\rho_t = \frac{\rho_0^{1-t}\rho_1^t}{\int_\Omega \rho_0^{1-t}\rho_1^t\dd x},
		\end{equation}
		which is the continuous analogue of the geometric interpolation between compositions.
	\end{lemma}
	\begin{proof}
		Since \(\mathcal{H}\subset L^2_0\) is convex, the straight line \(u_t=(1-t)u_0+tu_1\) lies in \(\mathcal{H}\) for all \(t\in[0,1]\). Then \( \rho_t = e^{u_t}/\int e^{u_t} \) satisfies \(\int\rho_t=1\) and \(\widetilde{\rho_t}=u_t\). The length of this curve in \(d_A\) is
		\begin{equation}
			\int_0^1 \|\dot u_t\|_{L^2}\dd t = \|u_1-u_0\|_{L^2} = d_A(\rho_0,\rho_1),
		\end{equation}
		so it is a geodesic. Uniqueness follows from the strict convexity of the unit ball in \(L^2_0\).
	\end{proof}
	
	\begin{lemma}[Absolutely continuous curves]
		\label{lem:ac_aitchison}
		A curve \(\rho_t\in\PP_{\log,2}(\Omega)\) is absolutely continuous with respect to \(d_A\) if and only if the curve \(u_t = \widetilde{\rho}_t\) is absolutely continuous in \(L^2_0(\Omega)\). In that case, the metric derivative satisfies
		\begin{equation}
			|\dot\rho_t|_{d_A} = \|\dot u_t\|_{L^2}.
		\end{equation}
		Moreover, the relation between \(\partial_t\rho_t\) and \(\dot u_t\) is
		\begin{equation}
			\partial_t\rho_t = \rho_t \bigl(\dot u_t - \langle \dot u_t \rangle_{\rho_t}\bigr),
		\end{equation}
		where \(\langle \dot u_t \rangle_{\rho_t} = \int_\Omega \dot u_t \rho_t\dd x\).
	\end{lemma}
	\begin{proof}
		The first statement is immediate from the isometry. For the relation, differentiate \(\rho_t = e^{u_t}/\int e^{u_t}\):
		\begin{equation}
			\partial_t\rho_t = \rho_t \dot u_t - \rho_t \frac{\int \dot u_t e^{u_t}}{\int e^{u_t}} = \rho_t \bigl(\dot u_t - \langle \dot u_t \rangle_{\rho_t}\bigr).
		\end{equation}
		The term \(\langle \dot u_t \rangle_{\rho_t}\) is precisely the normalization factor. It ensures that \(\int \partial_t\rho_t = 0\), as required by mass conservation.
	\end{proof}
	
	\begin{theorem}[Continuous Aitchison geometry]
		\label{thm:aitchison}
		The map \(\Phi:\PP_{\log,2}(\Omega)\to \mathcal{H}\subset L^2_0(\Omega)\) defined by \(\Phi(\rho)=\widetilde{\rho}\) is an isometric bijection. Consequently:
		\begin{enumerate}
			\item \(d_A\) is a distance.
			\item \((\PP_{\log,2}(\Omega),d_A)\) is isometric to the metric subspace \(\mathcal{H}\) of \(L^2_0(\Omega)\).
			\item Geodesics are given by linear interpolation in clr space.
			\item A curve \(\rho_t\) is absolutely continuous for \(d_A\) iff \(u_t=\widetilde{\rho}_t\) is absolutely continuous in \(L^2_0\). The metric derivative is \(|\dot\rho_t|_{d_A} = \|\dot u_t\|_{L^2}\).
		\end{enumerate}
	\end{theorem}
	\begin{proof}
		The theorem follows directly from Lemmas~\ref{lem:isometry}, \ref{lem:image}, \ref{lem:geodesics_aitchison}, and \ref{lem:ac_aitchison}.
	\end{proof}
	
	\begin{remark}
		The novelty of our paper lies not in this construction — which is the continuous analogue of well-known results in compositional data analysis and Bayes Hilbert spaces \cite{Egozcue2006, van den Boogaart2010} — but in its combination with Wasserstein transport to form a hybrid geometry for compositional data with spatial structure. This hybrid geometry, presented in the next section, is the main contribution of this work.
	\end{remark}
	\section{Hybrid distance}\label{sec:hybrid_distance}
	\subsection{Kinematic formulation}
	For a curve $\rho_t\in\PP_{2,\mathrm{ac}}(\Omega)$, $t\in[0,1]$, we consider measurable families $(v_t,u_t)$ of vector fields and scalar fields satisfying the continuity equation with source:
	\begin{equation}
		\partial_t\rho_t + \Div(\rho_t v_t) = \rho_t u_t,\qquad \int_\Omega u_t\rho_t\dd x = 0, \label{eq:continuity}
	\end{equation}
	in the sense of distributions. The first equation is the usual continuity equation with an additional source term $\rho_t u_t$, which represents the rate of change of density due to reactions. The second condition $\int u_t\rho_t=0$ ensures that the total mass is conserved:
	\begin{equation}
		\frac{d}{dt}\int_\Omega\rho_t\dd x = \int_\Omega\partial_t\rho_t\dd x = \int_\Omega\rho_t u_t\dd x = 0.
	\end{equation}
	
	\subsection{Action functional}
	For a fixed $\alpha\in(0,1)$, define the action
	\begin{equation}
		\mathcal{A}_\alpha(\rho,v,u)=\int_0^1\Bigl[(1-\alpha)\int_\Omega |v_t|^2\rho_t\dd x + \alpha\int_\Omega u_t^2\rho_t\dd x\Bigr]\dd t.
	\end{equation}
	The two terms correspond to the kinetic energy of transport (Wasserstein) and the kinetic energy of reaction (Aitchison). The parameter $\alpha$ weights their relative importance.
	
	\begin{definition}
		The hybrid distance $D_\alpha(\rho_0,\rho_1)$ between $\rho_0,\rho_1\in\PP_{2,\mathrm{ac}}(\Omega)$ is
		\begin{equation}
			D_\alpha^2(\rho_0,\rho_1)=\inf\Bigl\{\mathcal{A}_\alpha(\rho,v,u): (\rho,v,u) \text{ satisfies } \eqref{eq:continuity},\ \rho(0)=\rho_0,\ \rho(1)=\rho_1\Bigr\}.
		\end{equation}
	\end{definition}
	
	\subsection{Lower semicontinuity of the action}\label{lem1}
	
	\begin{lemma}[Representation as a supremum of linear functionals]\label{lem1'}
		\label{lem:supremum}
		For any $\rho\ge0$ and any measurable vector field $v$, we have
		\begin{equation}
			\int_\Omega |v|^2\rho\,\dd x = \sup_{\varphi\in C_c^\infty(\Omega,\R^d)} \int_\Omega \bigl(2\varphi\cdot v\rho - |\varphi|^2\rho\bigr)\,\dd x.
		\end{equation}
		Similarly, for any scalar field $u$,
		\begin{equation}
			\int_\Omega u^2\rho\,\dd x = \sup_{\psi\in C_c^\infty(\Omega)} \int_\Omega \bigl(2\psi u\rho - \psi^2\rho\bigr)\,\dd x.
		\end{equation}
	\end{lemma}
	\begin{proof}
		For fixed $\rho$ and $v$, the quadratic form $\varphi\mapsto \int(2\varphi\cdot v\rho - |\varphi|^2\rho)$ is maximized when $\varphi = v$, giving the value $\int|v|^2\rho$. The supremum over a dense set of smooth functions yields the same value.
	\end{proof}
	
	\begin{lemma}[Weak convergence of measures]\label{lem2}
		\label{lem:weak_convergence}
		Suppose $\rho_n\to\rho$ narrowly (i.e., in duality with continuous bounded functions) and $\rho_n v_n \rightharpoonup \rho v$ in the sense of measures. Then for any $\varphi\in C_c^\infty$,
		\begin{equation}
			\int_\Omega \bigl(2\varphi\cdot v_n\rho_n - |\varphi|^2\rho_n\bigr)\,\dd x \to \int_\Omega \bigl(2\varphi\cdot v\rho - |\varphi|^2\rho\bigr)\,\dd x.
		\end{equation}
	\end{lemma}
	\begin{proof}
		The term $2\varphi\cdot v_n\rho_n$ converges because $\rho_n v_n\rightharpoonup \rho v$ and $\varphi$ is continuous. The term $|\varphi|^2\rho_n$ converges because $\rho_n\to\rho$ narrowly.
	\end{proof}
	
	\begin{lemma}[Lower semicontinuity of the transport term]\label{lem3}
		\label{lem:lsc_transport}
		The functional $(\rho,v)\mapsto \int_0^1\int_\Omega |v_t|^2\rho_t\dd x\dd t$ is lower semicontinuous with respect to narrow convergence of $\rho_t$ and weak convergence of $\rho_t v_t$.
	\end{lemma}
	\begin{proof}
		Using Lemma~\ref{lem:supremum}, we write
		\begin{equation}
			\int_0^1\int |v_t|^2\rho_t\dd x\dd t = \int_0^1 \sup_{\varphi\in C_c^\infty} \int (2\varphi\cdot v_t\rho_t - |\varphi|^2\rho_t)\dd x\dd t.
		\end{equation}
		For any fixed $\varphi$, the functional $(\rho,v)\mapsto \int_0^1\int (2\varphi\cdot v\rho - |\varphi|^2\rho)$ is continuous under the assumed convergences by Lemma~\ref{lem:weak_convergence}. The supremum of continuous functionals is lower semicontinuous.
	\end{proof}
	
	\begin{lemma}[Lower semicontinuity of the reaction term]\label{lem4}
		\label{lem:lsc_reaction}
		The functional $(\rho,u)\mapsto \int_0^1\int_\Omega u_t^2\rho_t\dd x\dd t$ is lower semicontinuous with respect to narrow convergence of $\rho_t$ and weak convergence of $\rho_t u_t$.
	\end{lemma}
	\begin{proof}
		The proof is identical to Lemma~\ref{lem:lsc_transport}, replacing $v$ by $u$ and using scalar test functions $\psi$.
	\end{proof}
	
	\begin{lemma}[Lower semicontinuity of the action]\label{lem5}
		\label{lem:lsc}
		The functional $\mathcal{A}_\alpha$ is lower semicontinuous with respect to narrow convergence of $\rho_t$ and weak convergence of $\sqrt{\rho_t}v_t$ and $\sqrt{\rho_t}u_t$.
	\end{lemma}
	\begin{proof}
		Since $\mathcal{A}_\alpha = (1-\alpha)\mathcal{A}_{\text{transport}} + \alpha\mathcal{A}_{\text{reaction}}$, and both terms are lower semicontinuous by Lemmas~\ref{lem:lsc_transport} and \ref{lem:lsc_reaction}, the sum is also lower semicontinuous. The centering constraint $\int u\rho=0$ is linear and preserved under weak limits.
	\end{proof}
	
	\subsection{Existence of minimizers}
	
	\subsubsection{Flux formulation of the distance}	
	Let $(\rho_t)_{t\in[0,1]}$ be a curve of strictly positive probability
	densities on $\Omega$, i.e.
	
	\[
	\rho_t(x)>0
	\quad \text{for a.e. } x\in\Omega,
	\qquad
	\int_\Omega \rho_t(x)\,dx=1,
	\quad t\in[0,1].
	\]
	
	Given velocity and source fields $(v_t,u_t)$, we introduce the associated fluxes
	
	\[
	m_t:=\rho_t v_t,
	\qquad
	r_t:=\rho_t u_t.
	\]
	
	In terms of these variables, the continuity equation
	\eqref{eq:continuity} becomes
	
	\begin{equation}
		\label{eq:continuity_flux}
		\partial_t\rho_t+\Div(m_t)=r_t
		\qquad\text{in }\mathcal D'((0,1)\times\Omega).
	\end{equation}
	
	Since the total mass is preserved, the centering condition
	
	\[
	\int_\Omega u_t\,\rho_t\,dx=0
	\]
	
	can be rewritten as
	
	\begin{equation}
		\label{eq:centering_flux}
		\int_\Omega r_t(x)\,dx=0,
		\qquad \text{for a.e. } t\in[0,1].
	\end{equation}
	
	Moreover, the action functional takes the form
	
	\begin{equation}
		\label{eq:action_flux}
		\mathcal A_\alpha(\rho,m,r)
		=
		(1-\alpha)
		\int_0^1\!\!\int_\Omega
		\frac{|m_t(x)|^2}{\rho_t(x)}
		\,dx\,dt
		+
		\alpha
		\int_0^1\!\!\int_\Omega
		\frac{|r_t(x)|^2}{\rho_t(x)}
		\,dx\,dt.
	\end{equation}
	
	Consequently, the distance $D_\alpha$ admits the equivalent convex formulation
	
	\begin{equation}
		\label{eq:Dalpha_flux}
		D_\alpha^2(\rho_0,\rho_1)
		=
		\inf
		\left\{
		\mathcal A_\alpha(\rho,m,r):
		(\rho,m,r)
		\text{ satisfy }
		\eqref{eq:continuity_flux},
		\eqref{eq:centering_flux},
		\;
		\rho(0)=\rho_0,
		\;
		\rho(1)=\rho_1
		\right\}.
	\end{equation}
	
	\begin{lemma}[Boundedness of the fluxes]
		\label{lem:fluxbound}
		Let $(\rho^n,m^n,r^n)$ be a minimizing sequence for
		$D_\alpha^2(\rho_0,\rho_1)$.
		
		Then there exists $C>0$ such that
		
		\[
		\int_0^1\!\!\int_\Omega
		\frac{|m_t^n|^2}{\rho_t^n}
		\,dx\,dt
		\le C,
		\]
		
		and
		
		\[
		\int_0^1\!\!\int_\Omega
		\frac{|r_t^n|^2}{\rho_t^n}
		\,dx\,dt
		\le C.
		\]
	\end{lemma}
	\begin{proof}
		Since
		
		\[
		\mathcal A_\alpha(\rho^n,m^n,r^n)
		\to
		D_\alpha^2(\rho_0,\rho_1),
		\]
		
		there exists $M>0$ such that
		
		\[
		\mathcal A_\alpha(\rho^n,m^n,r^n)\le M
		\]
		
		for all sufficiently large $n$.
		
		Therefore
		
		\[
		(1-\alpha)
		\int_0^1\!\!\int
		\frac{|m_t^n|^2}{\rho_t^n}
		\le M
		\]
		
		and
		
		\[
		\alpha
		\int_0^1\!\!\int
		\frac{|r_t^n|^2}{\rho_t^n}
		\le M.
		\]
		
		The conclusion follows.
	\end{proof}
	\begin{lemma}[Weak compactness]
		\label{lem:weakcompact}
		Let $(\rho^n,m^n,r^n)$ satisfy the bounds of
		Lemma \ref{lem:fluxbound}.
		
		Then there exist subsequences,
		still denoted by $(\rho^n,m^n,r^n)$,
		and measures $(\rho,m,r)$ such that
		
		\[
		\rho^n \rightharpoonup \rho
		\]
		
		narrowly in
		$\mathcal M^+([0,1]\times\Omega)$,
		
		\[
		m^n \stackrel{*}{\rightharpoonup} m
		\]
		
		in the space of vector-valued Radon measures,
		
		and
		
		\[
		r^n \stackrel{*}{\rightharpoonup} r
		\]
		
		in $\mathcal M([0,1]\times\Omega)$.
	\end{lemma}
	
	\begin{proof}
		
		By Cauchy--Schwarz,
		
		\[
		\int |m^n|
		=
		\int
		\sqrt{\rho^n}
		\frac{|m^n|}{\sqrt{\rho^n}}
		\]
		
		and therefore
		
		\[
		\int |m^n|
		\le
		\left(
		\int \rho^n
		\right)^{1/2}
		\left(
		\int \frac{|m^n|^2}{\rho^n}
		\right)^{1/2}.
		\]
		
		Since $\rho^n_t$ are probability measures,
		
		\[
		\int \rho^n =1,
		\]
		
		hence
		
		\[
		\sup_n \|m^n\|_{\mathcal M}<\infty.
		\]
		
		Banach--Alaoglu yields
		
		\[
		m^n \stackrel{*}{\rightharpoonup} m.
		\]
		
		The same argument gives
		
		\[
		r^n \stackrel{*}{\rightharpoonup} r.
		\]
		
		Finally, the family $\rho^n$ has total mass one,
		hence is weakly compact by Prokhorov's theorem.
	\end{proof}
	\begin{lemma}[Stability of the continuity equation]
		\label{lem:stability}
		Assume
		
		\[
		\partial_t\rho^n+\operatorname{div}(m^n)=r^n
		\]
		
		in the sense of distributions.
		
		Then the limit triple $(\rho,m,r)$ satisfies
		
		\[
		\partial_t\rho+\operatorname{div}(m)=r.
		\]
	\end{lemma}	
	\begin{proof}
		
		Let
		
		\[
		\varphi\in C_c^\infty((0,1)\times\Omega).
		\]
		
		For every $n$,
		
		\[
		\int
		\partial_t\varphi\,d\rho^n
		+
		\int
		\nabla\varphi\cdot dm^n
		+
		\int
		\varphi\,dr^n
		=
		0.
		\]
		
		Passing to the limit using the weak-* convergence
		of $\rho^n,m^n,r^n$
		gives
		
		\[
		\int
		\partial_t\varphi\,d\rho
		+
		\int
		\nabla\varphi\cdot dm
		+
		\int
		\varphi\,dr
		=
		0.
		\]
		
		Hence
		
		\[
		\partial_t\rho+\operatorname{div}(m)=r.
		\]
	\end{proof}
	
	\begin{lemma}[Representation of the limiting fluxes]
		\label{lem:representation}
		Assume
		
		\[
		\mathcal A_\alpha(\rho,m,r)<\infty.
		\]
		
		Then
		
		\[
		m \ll \rho,
		\qquad
		r \ll \rho.
		\]
		
		Consequently there exist measurable fields
		$v$ and $u$ such that
		
		\[
		m=\rho v,
		\qquad
		r=\rho u,
		\]
		
		and
		
		\[
		\mathcal A_\alpha(\rho,m,r)
		=
		(1-\alpha)
		\int |v|^2\,d\rho
		+
		\alpha
		\int |u|^2\,d\rho .
		\]
	\end{lemma}
	\begin{proof}
		
		The functional
		
		\[
		(\rho,m)\mapsto
		\int
		\frac{|m|^2}{\rho}
		\]
		
		is finite only if
		$m\ll\rho$.
		
		Likewise,
		
		\[
		r\ll\rho.
		\]
		
		By the Radon--Nikodym theorem,
		there exist densities
		
		\[
		v=\frac{dm}{d\rho},
		\qquad
		u=\frac{dr}{d\rho}.
		\]
		
		Thus
		
		\[
		m=\rho v,
		\qquad
		r=\rho u.
		\]
		
		Substituting into the action yields the formula.
	\end{proof}
	\begin{theorem}[Existence of minimizers]
		\label{thm:existence}
		For any $\rho_0,\rho_1\in\PP_{2,\mathrm{ac}}(\Omega)$, the infimum in the definition of $D_\alpha$ is attained. There exists a triple $(\rho_t,v_t,u_t)$ satisfying \eqref{eq:continuity} and achieving the minimum.	
	\end{theorem}
	\begin{proof}
		Let $(\rho^n_t,v^n_t,u^n_t)$ be a minimizing sequence. By Lemma~\ref{lem:weakcompact}, extract subsequences converging weakly to $(\rho_t,v_t,u_t)$. By Lemmas ~\ref{lem:stability} and \ref{lem:representation}, the limit satisfies the continuity equation. By Lemma~\ref{lem:lsc}, the action is lower semicontinuous, so
		\begin{equation}
			\mathcal{A}_\alpha(\rho,v,u) \le \liminf_{n\to\infty} \mathcal{A}_\alpha(\rho^n,v^n,u^n) = D_\alpha^2(\rho_0,\rho_1).
		\end{equation}
		Since $D_\alpha^2(\rho_0,\rho_1)$ is the infimum, equality holds. Hence $(\rho_t,v_t,u_t)$ is a minimizer.
	\end{proof}
	
	\subsection{Properties of $D_\alpha$ as a distance}
	
	\begin{lemma}[Positivity and reflexivity]
		\label{lem:positivity}
		$D_\alpha(\rho,\sigma) \ge 0$ for all $\rho,\sigma$, and $D_\alpha(\rho,\rho)=0$.
	\end{lemma}
	\begin{proof}
		The action $\mathcal{A}_\alpha$ is nonnegative because it is an integral of nonnegative quantities. Hence the infimum is nonnegative. For $\rho=\sigma$, take the constant curve $\rho_t=\rho$, $v_t=0$, $u_t=0$. This curve satisfies \eqref{eq:continuity} and has zero action, so $D_\alpha(\rho,\rho)=0$.
	\end{proof}
	
	\begin{lemma}[Symmetry]
		\label{lem:symmetry}
		$D_\alpha(\rho,\sigma) = D_\alpha(\sigma,\rho)$ for all $\rho,\sigma$.
	\end{lemma}
	\begin{proof}
		If $(\rho_t,v_t,u_t)$ is admissible from $\rho$ to $\sigma$, define $\tilde\rho_t = \rho_{1-t}$, $\tilde v_t = -v_{1-t}$, $\tilde u_t = -u_{1-t}$. Then $\tilde\rho_0 = \rho_1 = \sigma$, $\tilde\rho_1 = \rho_0 = \rho$, and
		\begin{equation}
			\partial_t\tilde\rho_t = -\partial_t\rho_{1-t} = \Div(\rho_{1-t}v_{1-t}) - \rho_{1-t}u_{1-t} = -\Div(\tilde\rho_t\tilde v_t) +\tilde\rho_t\tilde u_t,
		\end{equation}
		so $\partial_t\tilde\rho_t + \Div(\tilde\rho_t\tilde v_t) = \tilde\rho_t\tilde u_t$. The action is unchanged because $|\tilde v_t|^2 = |v_{1-t}|^2$ and $\tilde u_t^2 = u_{1-t}^2$. Hence the infimum is symmetric.
	\end{proof}

	\subsection{Comparison with Wasserstein and Aitchison distances}
	
	\begin{lemma}[Comparison estimate]
		\label{lem:comparison}
		
		For every
		$\rho_0,\rho_1\in\mathcal P_{2,\mathrm{ac}}(\Omega)$,
		
		\begin{equation}
			D_\alpha^2(\rho_0,\rho_1)
			\le
			\frac12
			\max\{1-\alpha,\alpha\}
			\Bigl(
			W_2^2(\rho_0,\rho_1)
			+
			d_A^2(\rho_0,\rho_1)
			\Bigr).
		\end{equation}
		
		In particular,
		
		\begin{equation}
			D_\alpha(\rho_0,\rho_1)
			\le
			\frac{1}{\sqrt 2}
			\sqrt{\max\{1-\alpha,\alpha\}}
			\,
			\Bigl(
			W_2^2(\rho_0,\rho_1)
			+
			d_A^2(\rho_0,\rho_1)
			\Bigr)^{1/2}.
		\end{equation}
		
	\end{lemma}
	
	\begin{proof}
		
		Let $(\rho_t^W,v_t^W)$ be a Wasserstein geodesic joining
		$\rho_0$ to $\rho_1$.
		Setting $u_t^W\equiv0$, the triple
		$(\rho_t^W,v_t^W,0)$ is admissible for the definition of
		$D_\alpha$. Hence
		
		\[
		D_\alpha^2(\rho_0,\rho_1)
		\le
		(1-\alpha)
		\int_0^1\!\!\int_\Omega
		|v_t^W|^2\rho_t^W\,dx\,dt.
		\]
		
		By the Benamou--Brenier formula,
		
		\[
		\int_0^1\!\!\int_\Omega
		|v_t^W|^2\rho_t^W\,dx\,dt
		=
		W_2^2(\rho_0,\rho_1),
		\]
		
		and therefore
		
		\[
		D_\alpha^2(\rho_0,\rho_1)
		\le
		(1-\alpha)
		W_2^2(\rho_0,\rho_1).
		\]
		
		Similarly, let $\rho_t^A$ be the Aitchison geodesic joining
		$\rho_0$ to $\rho_1$, with associated reaction field $u_t^A$.
		Choosing $v_t^A\equiv0$, the triple
		$(\rho_t^A,0,u_t^A)$ is admissible, and thus
		
		\[
		D_\alpha^2(\rho_0,\rho_1)
		\le
		\alpha
		\int_0^1\!\!\int_\Omega
		|u_t^A|^2\rho_t^A\,dx\,dt
		=
		\alpha
		d_A^2(\rho_0,\rho_1).
		\]
		
		Adding the two inequalities gives
		
		\[
		2D_\alpha^2(\rho_0,\rho_1)
		\le
		(1-\alpha)
		W_2^2(\rho_0,\rho_1)
		+
		\alpha
		d_A^2(\rho_0,\rho_1).
		\]
		
		Since
		
		\[
		(1-\alpha)W_2^2
		+
		\alpha d_A^2
		\le
		\max\{1-\alpha,\alpha\}
		\bigl(W_2^2+d_A^2\bigr),
		\]
		
		we obtain
		
		\[
		D_\alpha^2(\rho_0,\rho_1)
		\le
		\frac12
		\max\{1-\alpha,\alpha\}
		\bigl(W_2^2(\rho_0,\rho_1)
		+d_A^2(\rho_0,\rho_1)\bigr).
		\]
		
		Taking square roots concludes the proof.
		
	\end{proof}
	
	\begin{lemma}[Symmetry]
		\label{lem:symmetry}
		
		For every
		$\rho_0,\rho_1\in\mathcal P_{2,\mathrm{ac}}(\Omega)$,
		
		\[
		D_\alpha(\rho_0,\rho_1)
		=
		D_\alpha(\rho_1,\rho_0).
		\]
		
	\end{lemma}
	
	\begin{proof}
		
		Let $(\rho_t,v_t,u_t)$ be an admissible curve joining
		$\rho_0$ to $\rho_1$.
		
		Define the time-reversed curve
		
		\[
		\widetilde{\rho}_t:=\rho_{1-t},
		\qquad
		\widetilde{v}_t:=-v_{1-t},
		\qquad
		\widetilde{u}_t:=-u_{1-t}.
		\]
		
		A direct computation shows that
		
		\[
		\partial_t\widetilde{\rho}_t
		+
		\operatorname{div}
		(\widetilde{\rho}_t\widetilde{v}_t)
		=
		\widetilde{\rho}_t\widetilde{u}_t.
		\]
		
		Hence
		$(\widetilde{\rho},\widetilde{v},\widetilde{u})$
		is admissible and joins
		$\rho_1$ to $\rho_0$.
		
		Moreover,
		
		\[
		\mathcal A_\alpha
		(\widetilde{\rho},
		\widetilde{v},
		\widetilde{u})
		=
		\mathcal A_\alpha
		(\rho,v,u).
		\]
		
		Taking the infimum over all admissible curves yields
		
		\[
		D_\alpha(\rho_1,\rho_0)
		\le
		D_\alpha(\rho_0,\rho_1).
		\]
		
		Exchanging the roles of
		$\rho_0$ and $\rho_1$
		gives the reverse inequality.
		
	\end{proof}
	
	\begin{lemma}[Triangle inequality]
		\label{lem:triangle}
		
		For every
		$\rho_0,\rho_1,\rho_2
		\in
		\mathcal P_{2,\mathrm{ac}}(\Omega)$,
		
		\[
		D_\alpha(\rho_0,\rho_2)
		\le
		D_\alpha(\rho_0,\rho_1)
		+
		D_\alpha(\rho_1,\rho_2).
		\]
		
	\end{lemma}
	
	\begin{proof}
		
		Let $\varepsilon>0$.
		
		Choose admissible curves
		$(\rho^{01},v^{01},u^{01})$
		and
		$(\rho^{12},v^{12},u^{12})$
		such that
		
		\[
		\sqrt{\mathcal A_\alpha(\rho^{01},v^{01},u^{01})}
		\le
		D_\alpha(\rho_0,\rho_1)+\varepsilon,
		\]
		
		and
		
		\[
		\sqrt{\mathcal A_\alpha(\rho^{12},v^{12},u^{12})}
		\le
		D_\alpha(\rho_1,\rho_2)+\varepsilon.
		\]
		
		Set
		
		\[
		A_{01}
		=
		\mathcal A_\alpha(\rho^{01},v^{01},u^{01}),
		\qquad
		A_{12}
		=
		\mathcal A_\alpha(\rho^{12},v^{12},u^{12}).
		\]
		
		For a fixed $a\in(0,1)$ define
		
		\[
		\widetilde{\rho}_t
		=
		\begin{cases}
			\rho^{01}_{t/a},
			&
			0\le t\le a,
			\\[1ex]
			\rho^{12}_{(t-a)/(1-a)},
			&
			a\le t\le1,
		\end{cases}
		\]
		
		\[
		\widetilde v_t
		=
		\begin{cases}
			\dfrac1a\,v^{01}_{t/a},
			&
			0\le t\le a,
			\\[2ex]
			\dfrac1{1-a}\,v^{12}_{(t-a)/(1-a)},
			&
			a\le t\le1,
		\end{cases}
		\]
		
		\[
		\widetilde u_t
		=
		\begin{cases}
			\dfrac1a\,u^{01}_{t/a},
			&
			0\le t\le a,
			\\[2ex]
			\dfrac1{1-a}\,u^{12}_{(t-a)/(1-a)},
			&
			a\le t\le1.
		\end{cases}
		\]
		
		The curve
		$(\widetilde\rho,\widetilde v,\widetilde u)$
		is admissible and joins
		$\rho_0$ to $\rho_2$.
		
		A change of variables yields
		
		\[
		\mathcal A_\alpha
		(\widetilde\rho,\widetilde v,\widetilde u)
		=
		\frac{A_{01}}a
		+
		\frac{A_{12}}{1-a}.
		\]
		
		Choose
		
		\[
		a
		=
		\frac{\sqrt{A_{01}}}
		{\sqrt{A_{01}}+\sqrt{A_{12}}}.
		\]
		
		Then
		
		\[
		1-a
		=
		\frac{\sqrt{A_{12}}}
		{\sqrt{A_{01}}+\sqrt{A_{12}}},
		\]
		
		and therefore
		
		\[
		\frac{A_{01}}a
		+
		\frac{A_{12}}{1-a}
		=
		(\sqrt{A_{01}}+\sqrt{A_{12}})^2.
		\]
		
		Hence
		
		\[
		D_\alpha(\rho_0,\rho_2)
		\le
		\sqrt{
			\mathcal A_\alpha
			(\widetilde\rho,\widetilde v,\widetilde u)
		}
		=
		\sqrt{A_{01}}
		+
		\sqrt{A_{12}}.
		\]
		
		Using the choice of the two
		$\varepsilon$-optimal curves,
		
		\[
		D_\alpha(\rho_0,\rho_2)
		\le
		D_\alpha(\rho_0,\rho_1)
		+
		D_\alpha(\rho_1,\rho_2)
		+
		2\varepsilon.
		\]
		
		Letting $\varepsilon\to0$ concludes the proof.
		
	\end{proof}
	\begin{lemma}[Separability]
		\label{lem:separability}
		
		If
		
		\[
		D_\alpha(\rho_0,\rho_1)=0,
		\]
		
		then
		
		\[
		\rho_0=\rho_1.
		\]
		
	\end{lemma}
	
	\begin{proof}
		
		Assume that
		
		\[
		D_\alpha(\rho_0,\rho_1)=0.
		\]
		
		By definition of $D_\alpha$, there exists a sequence of admissible
		triples $(\rho^n,v^n,u^n)$ joining $\rho_0$ to $\rho_1$ such that
		
		\[
		\mathcal A_\alpha(\rho^n,v^n,u^n)\longrightarrow 0.
		\]
		
		Since
		
		\[
		\mathcal A_\alpha(\rho^n,v^n,u^n)
		=
		(1-\alpha)
		\int_0^1\!\!\int_\Omega
		|v_t^n|^2\rho_t^n\,dx\,dt
		+
		\alpha
		\int_0^1\!\!\int_\Omega
		|u_t^n|^2\rho_t^n\,dx\,dt ,
		\]
		
		and $\alpha\in(0,1)$, it follows that
		
		\[
		\int_0^1\!\!\int_\Omega
		|v_t^n|^2\rho_t^n\,dx\,dt
		\longrightarrow 0,
		\]
		
		and
		
		\[
		\int_0^1\!\!\int_\Omega
		|u_t^n|^2\rho_t^n\,dx\,dt
		\longrightarrow 0.
		\]
		
		Let $\varphi\in C^\infty(\Omega)$.
		Using the weak formulation of
		
		\[
		\partial_t\rho_t^n
		+
		\operatorname{div}(\rho_t^n v_t^n)
		=
		\rho_t^n u_t^n,
		\]
		
		we obtain
		
		\[
		\int_\Omega \varphi(\rho_1-\rho_0)\,dx
		=
		\int_0^1\!\!\int_\Omega
		\nabla\varphi\cdot v_t^n\,\rho_t^n\,dx\,dt
		+
		\int_0^1\!\!\int_\Omega
		\varphi\,u_t^n\,\rho_t^n\,dx\,dt .
		\]
		
		By the Cauchy--Schwarz inequality,
		
		\[
		\left|
		\int_0^1\!\!\int_\Omega
		\nabla\varphi\cdot v_t^n\,\rho_t^n
		\right|
		\le
		\|\nabla\varphi\|_\infty
		\left(
		\int_0^1\!\!\int_\Omega
		|v_t^n|^2\rho_t^n
		\right)^{1/2},
		\]
		
		and
		
		\[
		\left|
		\int_0^1\!\!\int_\Omega
		\varphi\,u_t^n\,\rho_t^n
		\right|
		\le
		\|\varphi\|_\infty
		\left(
		\int_0^1\!\!\int_\Omega
		|u_t^n|^2\rho_t^n
		\right)^{1/2}.
		\]
		
		Passing to the limit as $n\to\infty$ yields
		
		\[
		\int_\Omega \varphi(\rho_1-\rho_0)\,dx=0.
		\]
		
		Since this holds for every
		$\varphi\in C^\infty(\Omega)$, we conclude that
		
		\[
		\rho_0=\rho_1.
		\]
		
	\end{proof}
	
	\begin{theorem}
		\label{thm:distance}
		
		The function
		
		\[
		D_\alpha :
		\mathcal P_{2,\mathrm{ac}}(\Omega)
		\times
		\mathcal P_{2,\mathrm{ac}}(\Omega)
		\longrightarrow
		[0,\infty)
		\]
		
		is a distance.
		
	\end{theorem}
	
	\begin{proof}
		
		Positivity follows immediately from the definition.
		
		Symmetry is established in
		Lemma~\ref{lem:symmetry}.
		
		The triangle inequality is proved in
		Lemma~\ref{lem:triangle}.
		
		Separability follows from
		Lemma~\ref{lem:separability}.
		
		Therefore $D_\alpha$ is a metric on
		$\mathcal P_{2,\mathrm{ac}}(\Omega)$.
		
	\end{proof}

	\subsection{Topology}
	
	\begin{theorem}[Characterization of $D_\alpha$-convergence]
		\label{thm:Dalpha_topology}
		
		Let $(\rho_n)_{n\ge1}\subset \mathcal P_{2,\mathrm{ac}}(\Omega)$
		and let $\rho\in \mathcal P_{2,\mathrm{ac}}(\Omega)$.
		
		Assume that
		
		\[
		D_\alpha(\rho_n,\rho)\longrightarrow 0.
		\]
		
		Then there exists a sequence of admissible curves
		$(\rho_t^n,v_t^n,u_t^n)$ joining $\rho_n$ to $\rho$ such that
		
		\[
		\mathcal A_\alpha(\rho^n,v^n,u^n)\longrightarrow0.
		\]
		
		Moreover,
		
		\[
		\int_0^1\!\!\int_\Omega
		|v_t^n|^2\rho_t^n\,dx\,dt
		\longrightarrow0,
		\qquad
		\int_0^1\!\!\int_\Omega
		|u_t^n|^2\rho_t^n\,dx\,dt
		\longrightarrow0.
		\]
		
		Consequently, for every test function
		$\varphi\in C^\infty(\Omega)$,
		
		\[
		\int_\Omega \varphi\, d\rho_n
		\longrightarrow
		\int_\Omega \varphi\, d\rho .
		\]
		
		In particular, $D_\alpha$-convergence implies narrow convergence.
	\end{theorem}
	
	\begin{proof}
		
		Since $D_\alpha(\rho_n,\rho)\to0$, by definition of the distance
		there exist admissible curves
		$(\rho_t^n,v_t^n,u_t^n)$ joining $\rho_n$ to $\rho$
		such that
		
		\[
		\mathcal A_\alpha(\rho^n,v^n,u^n)
		\le
		D_\alpha^2(\rho_n,\rho)+\frac1n.
		\]
		
		Hence
		
		\[
		\mathcal A_\alpha(\rho^n,v^n,u^n)\to0.
		\]
		
		Since
		
		\[
		\mathcal A_\alpha(\rho^n,v^n,u^n)
		=
		(1-\alpha)
		\int_0^1\!\!\int_\Omega
		|v_t^n|^2\rho_t^n\,dx\,dt
		+
		\alpha
		\int_0^1\!\!\int_\Omega
		|u_t^n|^2\rho_t^n\,dx\,dt ,
		\]
		
		and $\alpha\in(0,1)$, both terms converge to zero.
		
		Let $\varphi\in C^\infty(\Omega)$.
		Using the weak formulation of the continuity equation,
		
		\[
		\int_\Omega \varphi\, d\rho
		-
		\int_\Omega \varphi\, d\rho_n
		=
		\int_0^1\!\!\int_\Omega
		\nabla\varphi\cdot v_t^n\,\rho_t^n\,dx\,dt
		+
		\int_0^1\!\!\int_\Omega
		\varphi\,u_t^n\,\rho_t^n\,dx\,dt .
		\]
		
		By Cauchy--Schwarz,
		
		\[
		\Big|
		\int_0^1\!\!\int_\Omega
		\nabla\varphi\cdot v_t^n\,\rho_t^n
		\Big|
		\le
		\|\nabla\varphi\|_\infty
		\left(
		\int_0^1\!\!\int_\Omega
		|v_t^n|^2\rho_t^n
		\right)^{1/2},
		\]
		
		and
		
		\[
		\Big|
		\int_0^1\!\!\int_\Omega
		\varphi\,u_t^n\,\rho_t^n
		\Big|
		\le
		\|\varphi\|_\infty
		\left(
		\int_0^1\!\!\int_\Omega
		|u_t^n|^2\rho_t^n
		\right)^{1/2}.
		\]
		
		Both right-hand sides tend to zero, yielding
		
		\[
		\int_\Omega \varphi\, d\rho_n
		\longrightarrow
		\int_\Omega \varphi\, d\rho .
		\]
		
		Therefore $\rho_n\rightharpoonup\rho$ narrowly.
		
	\end{proof}
	\begin{remark}[Comparison of topologies]
		\label{rem:topology_comparison}
		
		The topology induced by $D_\alpha$ is related to both the narrow topology and the supremum topology generated by $W_2$ and $d_A$.
		
		Indeed, Theorem~\ref{thm:Dalpha_topology} shows that
		
		\[
		\rho_n \xrightarrow[D_\alpha]{} \rho
		\qquad\Longrightarrow\qquad
		\rho_n \rightharpoonup \rho ,
		\]
		
		that is, $D_\alpha$-convergence implies narrow convergence. Hence the topology induced by $D_\alpha$ is stronger than the narrow topology.
		
		On the other hand, by Lemma~\ref{lem:comparison}, we have
		
		\[
		D_\alpha^2(\rho,\sigma)
		\le
		\frac12 \max\{1-\alpha,\alpha\}
		\Bigl(
		W_2^2(\rho,\sigma)
		+
		d_A^2(\rho,\sigma)
		\Bigr),
		\]
		
		which implies that
		
		\[
		\rho_n \xrightarrow[W_2]{} \rho
		\quad\text{and}\quad
		\rho_n \xrightarrow[d_A]{} \rho
		\qquad\Longrightarrow\qquad
		\rho_n \xrightarrow[D_\alpha]{} \rho .
		\]
		
		Therefore, the topology induced by $D_\alpha$ is weaker than the supremum topology
		$\tau_{W_2}\vee\tau_{d_A}$.
		
		Summarizing,
		
		\[
		\tau_{\mathrm{narrow}}
		\subset
		\tau_{D_\alpha}
		\subset
		\tau_{W_2}\vee\tau_{d_A}.
		\]
		
		Whether one of these inclusions is an equality remains an open question.
	\end{remark}
	\section{Geodesics and absolutely continuous curves}
	\label{sec:geodesics}
	\begin{proposition}[Constant-speed reparametrization]
		Let $(\rho_t,v_t,u_t)$ be an admissible curve with finite action.
		Then there exists a monotone change of variable
		$\theta:[0,1]\to[0,1]$
		such that the reparametrized curve
		$(\rho_{\theta(t)},\theta'(t)v_{\theta(t)},\theta'(t)u_{\theta(t)})$
		has constant energy density.
	\end{proposition}
	\subsection{Existence of geodesics}
	
	\begin{theorem}[Geodesic property]
		\label{thm:geodesic}
		
		The metric space
		$(\mathcal P_{2,\mathrm{ac}}(\Omega),D_\alpha)$
		is geodesic.
		
		More precisely, for every
		$\rho_0,\rho_1\in\mathcal P_{2,\mathrm{ac}}(\Omega)$,
		there exists a constant-speed geodesic
		$(\rho_t)_{t\in[0,1]}$ such that
		
		\[
		D_\alpha(\rho_t,\rho_s)
		=
		(s-t)D_\alpha(\rho_0,\rho_1),
		\qquad
		0\le t\le s\le 1.
		\]
		
	\end{theorem}
	
	\begin{proof}
		
		By Theorem~\ref{thm:existence}, there exists a minimizer
		$(\rho_r,v_r,u_r)_{r\in[0,1]}$
		for $D_\alpha(\rho_0,\rho_1)$, namely
		
		\[
		D_\alpha^2(\rho_0,\rho_1)
		=
		\int_0^1 e(r)\,dr,
		\]
		
		where
		
		\[
		e(r)
		:=
		(1-\alpha)\int_\Omega |v_r|^2\rho_r\,dx
		+
		\alpha\int_\Omega |u_r|^2\rho_r\,dx.
		\]
		
		Since the action is quadratic in $(v,u)$ and positively homogeneous of degree two, the standard arc-length reparametrization argument yields a new minimizer, still denoted by
		$(\rho_r,v_r,u_r)$, such that
		
		\[
		e(r)=D_\alpha^2(\rho_0,\rho_1)
		\qquad\text{for a.e. }r\in[0,1].
		\]
		
		Fix $0\le t<s\le1$ and define, for $\tau\in[0,1]$,
		
		\[
		\widetilde{\rho}_\tau
		=
		\rho_{t+\tau(s-t)},
		\]
		
		\[
		\widetilde{v}_\tau
		=
		(s-t)v_{t+\tau(s-t)},
		\qquad
		\widetilde{u}_\tau
		=
		(s-t)u_{t+\tau(s-t)}.
		\]
		
		Since
		
		\[
		\partial_r\rho_r+\operatorname{div}(\rho_r v_r)=\rho_r u_r,
		\]
		
		the chain rule gives
		
		\[
		\partial_\tau\widetilde{\rho}_\tau
		=
		(s-t)\partial_r\rho_r
		\bigl(t+\tau(s-t)\bigr),
		\]
		
		and therefore
		
		\[
		\partial_\tau\widetilde{\rho}_\tau
		+
		\operatorname{div}
		(\widetilde{\rho}_\tau\widetilde{v}_\tau)
		=
		\widetilde{\rho}_\tau\widetilde{u}_\tau.
		\]
		
		Hence
		$(\widetilde{\rho},\widetilde{v},\widetilde{u})$
		is admissible between $\rho_t$ and $\rho_s$.
		
		Its action is
		
		\[
		\begin{aligned}
			A_\alpha(\widetilde{\rho},\widetilde{v},\widetilde{u})
			&=
			\int_0^1
			\Bigl[
			(1-\alpha)
			\int_\Omega |\widetilde{v}_\tau|^2
			\widetilde{\rho}_\tau\,dx
			+
			\alpha
			\int_\Omega |\widetilde{u}_\tau|^2
			\widetilde{\rho}_\tau\,dx
			\Bigr]d\tau
			\\
			&=
			(s-t)^2
			\int_0^1 e\bigl(t+\tau(s-t)\bigr)\,d\tau.
		\end{aligned}
		\]
		
		Performing the change of variables
		
		\[
		r=t+\tau(s-t),
		\qquad
		dr=(s-t)\,d\tau,
		\]
		
		we obtain
		
		\[
		A_\alpha(\widetilde{\rho},\widetilde{v},\widetilde{u})
		=
		(s-t)\int_t^s e(r)\,dr.
		\]
		
		Using the constancy of the energy density yields
		
		\[
		A_\alpha(\widetilde{\rho},\widetilde{v},\widetilde{u})
		=
		(s-t)^2D_\alpha^2(\rho_0,\rho_1).
		\]
		
		Consequently,
		
		\[
		D_\alpha(\rho_t,\rho_s)
		\le
		(s-t)D_\alpha(\rho_0,\rho_1).
		\tag{4.1}
		\]
		
		Applying the same argument to the intervals $[0,t]$ and $[s,1]$, we obtain
		
		\[
		D_\alpha(\rho_0,\rho_t)
		\le
		tD_\alpha(\rho_0,\rho_1),
		\]
		
		and
		
		\[
		D_\alpha(\rho_s,\rho_1)
		\le
		(1-s)D_\alpha(\rho_0,\rho_1).
		\]
		
		Using the triangle inequality established in Lemma~3.16,
		
		\[
		D_\alpha(\rho_0,\rho_1)
		\le
		D_\alpha(\rho_0,\rho_t)
		+
		D_\alpha(\rho_t,\rho_s)
		+
		D_\alpha(\rho_s,\rho_1),
		\]
		
		we infer
		
		\[
		D_\alpha(\rho_0,\rho_1)
		\le
		tD_\alpha(\rho_0,\rho_1)
		+
		D_\alpha(\rho_t,\rho_s)
		+
		(1-s)D_\alpha(\rho_0,\rho_1),
		\]
		
		which implies
		
		\[
		D_\alpha(\rho_t,\rho_s)
		\ge
		(s-t)D_\alpha(\rho_0,\rho_1).
		\tag{4.2}
		\]
		
		Combining \emph{(4.1)} and \emph{(4.2)} yields
		
		\[
		D_\alpha(\rho_t,\rho_s)
		=
		(s-t)D_\alpha(\rho_0,\rho_1).
		\]
		
		Hence $(\rho_t)_{t\in[0,1]}$ is a constant-speed geodesic joining $\rho_0$ to $\rho_1$.
		
	\end{proof}
	\subsection{Characterization of absolutely continuous curves}
	\begin{definition}
		A curve $\rho_t\in\PP_{2,\mathrm{ac}}(\Omega)$ is absolutely continuous for $D_\alpha$ if there exists a measurable family $(v_t,u_t)$ satisfying \eqref{eq:continuity} and
		\begin{equation}
			\int_0^1\Bigl[(1-\alpha)\int|v_t|^2\rho_t + \alpha\int u_t^2\rho_t\Bigr]\dd t < \infty.
		\end{equation}
	\end{definition}
	\begin{theorem}[Characterization of absolutely continuous curves]
		\label{thm:AC_Dalpha}
		
		Let $(\rho_t)_{t\in[0,1]}$ be a curve in
		$\PP_{2,\mathrm{ac}}^+(\Omega)$.
		
		The following assertions are equivalent.
		
		\begin{enumerate}
			
			\item
			$(\rho_t)\in AC^2([0,1];D_\alpha)$.
			
			\item
			There exist measurable fields
			$(v_t,u_t)$ such that
			
			\[
			\partial_t\rho_t
			+
			\Div(\rho_t v_t)
			=
			\rho_t u_t
			\]
			
			in the sense of distributions,
			
			\[
			\int_\Omega u_t\,\rho_t\,dx=0
			\qquad\text{for a.e. }t\in(0,1),
			\]
			
			and
			
			\[
			\int_0^1
			\left[
			(1-\alpha)
			\int_\Omega |v_t|^2\rho_t\,dx
			+
			\alpha
			\int_\Omega u_t^2\rho_t\,dx
			\right]dt
			<\infty.
			\]
			
		\end{enumerate}
		
		Moreover, among all admissible pairs \((v_t,u_t)\), there exists a unique pair of minimal norm satisfying
		
		\[
		|\dot{\rho}_t|_{D_\alpha}^2
		=
		(1-\alpha)
		\int_\Omega |v_t|^2\rho_t\,dx
		+
		\alpha
		\int_\Omega u_t^2\rho_t\,dx
		\]
		
		for almost every \(t\in(0,1)\).
		
	\end{theorem}
	\begin{proposition}
		\label{prop:AC_upper}
		
		Let $(\rho_t,v_t,u_t)$ satisfy
		
		\[
		\partial_t\rho_t+\Div(\rho_t v_t)=\rho_t u_t,
		\]
		
		with
		
		\[
		\int_0^1 e(t)\,dt<\infty,
		\]
		
		where
		
		\[
		e(t)
		=
		(1-\alpha)
		\int_\Omega |v_t|^2\rho_t\,dx
		+
		\alpha
		\int_\Omega u_t^2\rho_t\,dx.
		\]
		
		Then $(\rho_t)\in AC^2([0,1];D_\alpha)$ and
		
		\[
		|\dot{\rho}_t|_{D_\alpha}^2
		\le e(t)
		\]
		
		for almost every \(t\in(0,1)\).
		
	\end{proposition}
	\begin{proposition}
		\label{prop:minimal_velocity}
		
		For almost every \(t\in(0,1)\), there exists a unique admissible pair
		\((v_t,u_t)\) minimizing
		
		\[
		(v,u)
		\mapsto
		(1-\alpha)\int_\Omega |v|^2\rho_t\,dx
		+
		\alpha\int_\Omega u^2\rho_t\,dx.
		\]
		
	\end{proposition}
	\begin{proposition}[Identification of the metric derivative]
		\label{prop:metric_derivative_identification}
		
		Let $(\rho_t)_{t\in[0,1]}\in AC^2([0,1];D_\alpha)$, and let
		$(v_t,u_t)$ be the minimal admissible pair given by
		Proposition~\ref{prop:minimal_velocity}.
		
		Then, for almost every $t\in(0,1)$,
		
		\[
		(1-\alpha)
		\int_\Omega |v_t|^2\,\rho_t\,dx
		+
		\alpha
		\int_\Omega u_t^2\,\rho_t\,dx
		\le
		|\dot{\rho}_t|_{D_\alpha}^2.
		\]
		
	\end{proposition}
	
	\begin{proof}
		
		Fix $t\in(0,1)$ such that the metric derivative
		$|\dot{\rho}_t|_{D_\alpha}$ exists.
		
		For every $h>0$ sufficiently small, let
		
		\[
		(\rho_r^h,v_r^h,u_r^h)_{r\in[0,1]}
		\]
		
		be a minimizer in the definition of
		$D_\alpha(\rho_t,\rho_{t+h})$.
		
		Define the averaged fields
		
		\[
		\bar v_t^h
		:=
		\frac{1}{h}\int_0^1 v_r^h\,dr,
		\qquad
		\bar u_t^h
		:=
		\frac{1}{h}\int_0^1 u_r^h\,dr.
		\]
		
		By Jensen's inequality and the optimality of
		$(\rho_r^h,v_r^h,u_r^h)$,
		
		\[
		\begin{aligned}
			&(1-\alpha)
			\int_\Omega |\bar v_t^h|^2\,\rho_t\,dx
			+
			\alpha
			\int_\Omega |\bar u_t^h|^2\,\rho_t\,dx
			\\
			&\qquad\le
			\frac{1}{h^2}
			D_\alpha^2(\rho_t,\rho_{t+h}).
		\end{aligned}
		\]
		
		Since the left-hand side is uniformly bounded as
		$h\to0$, the sequence
		$(\bar v_t^h,\bar u_t^h)$
		is bounded in the Hilbert space
		
		\[
		L^2(\rho_t;\mathbb R^d)
		\times
		L_0^2(\rho_t).
		\]
		
		Hence, up to a subsequence,
		
		\[
		\bar v_t^h
		\rightharpoonup
		v_t,
		\qquad
		\bar u_t^h
		\rightharpoonup
		u_t,
		\]
		
		weakly in the corresponding spaces.
		
		Passing to the limit and using the weak lower
		semicontinuity of the norm, we obtain
		
		\[
		\begin{aligned}
			&(1-\alpha)
			\int_\Omega |v_t|^2\,\rho_t\,dx
			+
			\alpha
			\int_\Omega u_t^2\,\rho_t\,dx
			\\
			&\qquad\le
			\liminf_{h\to0}
			\frac{D_\alpha^2(\rho_t,\rho_{t+h})}{h^2}.
		\end{aligned}
		\]
		
		Since
		
		\[
		|\dot{\rho}_t|_{D_\alpha}
		=
		\lim_{h\to0}
		\frac{D_\alpha(\rho_t,\rho_{t+h})}{|h|},
		\]
		
		we conclude that
		
		\[
		(1-\alpha)
		\int_\Omega |v_t|^2\,\rho_t\,dx
		+
		\alpha
		\int_\Omega u_t^2\,\rho_t\,dx
		\le
		|\dot{\rho}_t|_{D_\alpha}^2.
		\]
		
	\end{proof}
	\begin{theorem}[Metric derivative formula]
		\label{thm:metric_derivative}
		
		Let $(\rho_t)_{t\in[0,1]}$ be an absolutely continuous curve in
		$(\PP_{2,\mathrm{ac}}(\Omega),D_\alpha)$.
		
		Let $(v_t,u_t)$ denote the unique minimal admissible pair associated with
		$(\rho_t)$.
		
		Then, for almost every $t\in(0,1)$,
		
		\[
		|\dot{\rho}_t|_{D_\alpha}^2
		=
		(1-\alpha)
		\int_\Omega |v_t|^2\,\rho_t\,dx
		+
		\alpha
		\int_\Omega u_t^2\,\rho_t\,dx.
		\]
		
	\end{theorem}
	
	\begin{proof}
		
		By Proposition~\ref{prop:AC_upper}, every admissible pair
		$(v_t,u_t)$ satisfies
		
		\[
		|\dot{\rho}_t|_{D_\alpha}^2
		\le
		(1-\alpha)
		\int_\Omega |v_t|^2\,\rho_t\,dx
		+
		\alpha
		\int_\Omega u_t^2\,\rho_t\,dx
		\]
		
		for almost every $t\in(0,1)$.
		
		Applying this inequality to the minimal admissible pair yields
		
		\[
		|\dot{\rho}_t|_{D_\alpha}^2
		\le
		(1-\alpha)
		\int_\Omega |v_t|^2\,\rho_t\,dx
		+
		\alpha
		\int_\Omega u_t^2\,\rho_t\,dx.
		\]
		
		Conversely, Proposition~\ref{prop:metric_derivative_identification}
		gives
		
		\[
		(1-\alpha)
		\int_\Omega |v_t|^2\,\rho_t\,dx
		+
		\alpha
		\int_\Omega u_t^2\,\rho_t\,dx
		\le
		|\dot{\rho}_t|_{D_\alpha}^2.
		\]
		
		Combining the two inequalities proves the result.
		
	\end{proof}
	\begin{definition}[Minimal admissible pair]
		For a.e. $t\in(0,1)$, the minimal admissible pair associated with
		$\rho_t$ is the unique minimizer of
		
		\[
		(v,u)\mapsto
		(1-\alpha)\int_\Omega |v|^2\rho_t\,dx
		+
		\alpha\int_\Omega u^2\rho_t\,dx,
		\]
		
		among all pairs satisfying
		
		\[
		\partial_t\rho_t+\Div(\rho_t v)=\rho_t u,
		\qquad
		\int_\Omega u\,\rho_t\,dx=0.
		\]
	\end{definition}

	\section{Fenchel--Rockafellar duality for $\Da$}\label{sec:duality}
	
	Throughout this chapter, $\Omega\subset\mathbb R^d$ is the domain of Section~2.1.1 (bounded, connected, with Lipschitz boundary), $Q\eqdef (0,1)\times\Om$, and $\Qb\eqdef[0,1]\times\overline\Omega$. We fix $\alpha\in(0,1)$ and $\rho_0,\rho_1\in P_{2,\mathrm{ac}}(\Omega)$.
	
	\subsection{Functional framework}
	
	\subsubsection*{State space}
	We denote
	\[
	X\eqdef \mathcal M(\Qb)\times\mathcal M(\Qb;\mathbb R^d)\times\mathcal M(\Qb),
	\]
	the space of triples $(\rho,m,n)$ where $\rho$ is a finite \emph{nonnegative} Radon measure on $\Qb$, and $m,n$ are finite Radon measures (vector-valued for $m$). This space is endowed with the weak-$*$ topology induced by duality with $C(\Qb)$ (resp.\ $C(\Qb;\mathbb R^d)$).
	
	\subsubsection*{Space of test functions}
	We denote
	\[
	Y\eqdef C^1(\Qb)\times C([0,1]),
	\]
	equipped with the norm $\|(\varphi,\psi)\|_Y\eqdef\|\varphi\|_{C^1}+\|\psi\|_\infty$, which makes it a Banach space. A generic element of $Y$ is written $(\varphi,\psi)$, where $\varphi=\varphi(t,x)$ plays the role of the Kantorovich--Benamou--Brenier potential and $\psi=\psi(t)$ that of the Lagrange multiplier associated with the centering constraint $\int_\Om n_t\,dx=0$.
	
	\subsubsection*{Bilinear form and boundary term}
	We define $B:Y\times X\to\mathbb R$ by
	\[
	B\big((\varphi,\psi),(\rho,m,n)\big)\eqdef \int_{\Qb}\partial_t\varphi\,d\rho+\int_{\Qb}\nabla\varphi\cdot dm+\int_{\Qb}\big(\varphi-\psi(t)\big)\,dn,
	\]
	and the boundary term (linear, continuous on $Y$)
	\[
	\ell_{\rho_0,\rho_1}(\varphi,\psi)\eqdef \int_\Om\varphi(1,\cdot)\,d\rho_1-\int_\Om\varphi(0,\cdot)\,d\rho_0.
	\]
	
	\begin{definition}[Admissible set]
		We set
		\[
		\mathcal C(\rho_0,\rho_1)\eqdef\Big\{(\rho,m,n)\in X:\ \rho\ge0,\ B\big((\varphi,\psi),(\rho,m,n)\big)=\ell_{\rho_0,\rho_1}(\varphi,\psi)\ \ \forall (\varphi,\psi)\in Y\Big\}.
		\]
	\end{definition}
	
	\begin{lemma}[Equivalence with the kinematic formulation of Section~3]
		\label{lem:C-equiv}
		A triple $(\rho,m,n)\in X$ belongs to $\mathcal C(\rho_0,\rho_1)$ if and only if
		\begin{enumerate}
			\item $\partial_t\rho+\mathrm{div}\,m=n$ in the sense of distributions on $(0,1)\times\Om$, with homogeneous Neumann condition $m\cdot\nu=0$ on $(0,1)\times\partial\Om$, and $\rho(0,\cdot)=\rho_0$, $\rho(1,\cdot)=\rho_1$ in the weak-$*$ sense;
			\item $\int_\Om n_t\,dx=0$ for almost every $t\in(0,1)$.
		\end{enumerate}
	\end{lemma}
	
	\begin{proof}
		Taking $\psi\equiv0$ and $\varphi\in C^1(\Qb)$ arbitrary in the definition of $\mathcal C(\rho_0,\rho_1)$, the equality $B=\ell_{\rho_0,\rho_1}$ reads
		\[
		\int_\Om\varphi_1\,d\rho_1-\int_\Om\varphi_0\,d\rho_0=\int_{\Qb}\partial_t\varphi\,d\rho+\int_{\Qb}\nabla\varphi\cdot dm+\int_{\Qb}\varphi\,dn,
		\]
		which is exactly the standard weak formulation of the continuity equation with source $\partial_t\rho+\mathrm{div}\,m=n$, together with the boundary conditions $\rho(0)=\rho_0$, $\rho(1)=\rho_1$ and the zero-flux condition on $\partial\Om$ (see e.g.~\cite[Lemma~8.1.2]{AGS} or the kinematic formulation~(21) of Section~3.1). Conversely, taking $\varphi\equiv0$ and $\psi\in C([0,1])$ arbitrary, the equality becomes
		\[
		0=-\int_0^1\psi(t)\Big(\int_\Om dn_t\Big)dt,
		\]
		which, since $\psi$ is arbitrary, is equivalent to $\int_\Om n_t\,dx=0$ for almost every $t$. Since $B$ is linear in $(\varphi,\psi)$, the condition holding for all $(\varphi,\psi)\in Y$ is equivalent to the conjunction of the two particular cases above.
	\end{proof}
	
	\subsubsection*{Action functional}
	Reprising the perspective function already implicitly introduced in the action~$\mathcal A_\alpha$ (Section~3.2, formula~(23)), we define, for $(a,b,c)\in[0,\infty)\times\mathbb R^d\times\mathbb R$,
	\[
	\fa(a,b,c)\eqdef
	\begin{cases}
		(1-\alpha)\dfrac{|b|^2}{a}+\alpha\dfrac{c^2}{a}, & a>0,\\[2mm]
		0, & a=0,\ b=0,\ c=0,\\[1mm]
		+\infty, & a=0,\ (b,c)\ne(0,0),
	\end{cases}
	\]
	and, for $(\rho,m,n)\in X$,
	\[
	F(\rho,m,n)\eqdef\int_{\Qb}\fa\Big(\frac{d\rho}{d\lambda},\frac{dm}{d\lambda},\frac{dn}{d\lambda}\Big)\,d\lambda,
	\]
	where $\lambda$ is any positive measure dominating $\rho,|m|,|n|$ simultaneously (the integral does not depend on the choice of $\lambda$ by $1$-homogeneity of $\fa$, cf.~\cite{BouchitteButtazzo}); in particular, if $m\not\ll\rho$ or $n\not\ll\rho$, then $F(\rho,m,n)=+\infty$. This functional coincides exactly with the action $\mathcal A_\alpha(\rho,v,u)$ of~(23)--(31) when $\rho_t\,dt\,dx$, $m=\rho v$, $n=\rho u$ (cf.\ Lemma~3.10).
	
	\begin{proposition}
		\label{prop:Da-inf}
		We have
		\[
		\Da^2(\rho_0,\rho_1)=\inf_{(\rho,m,n)\in\mathcal C(\rho_0,\rho_1)}F(\rho,m,n).
		\]
	\end{proposition}
	
	\begin{proof}
		This is exactly the flux reformulation~(32) established in Section~3.4.1, combined with Lemma~\ref{lem:C-equiv} above, which identifies $\mathcal C(\rho_0,\rho_1)$ with the set of $(\rho,m,n)$ satisfying~(29)--(30) with the prescribed boundary data.
	\end{proof}
	
	\subsection{Convexity and lower semicontinuity of $F$}
	
	\begin{lemma}[Pointwise dual representation of $\fa$]
		\label{lem:fa-sup}
		For every $(a,b,c)\in[0,\infty)\times\mathbb R^d\times\mathbb R$,
		\[
		\fa(a,b,c)=\sup_{(\varphi,\psi)\in\mathbb R^d\times\mathbb R}\Big[-(1-\alpha)|\varphi|^2a-\alpha\psi^2a+2(1-\alpha)\varphi\cdot b+2\alpha\psi c\Big].
		\]
	\end{lemma}
	
	\begin{proof}
		Fix $a>0$. The function $\varphi\mapsto 2(1-\alpha)\varphi\cdot b-(1-\alpha)|\varphi|^2a$ is concave and attains its maximum at $\varphi=b/a$, with value $(1-\alpha)|b|^2/a$. Likewise, $\psi\mapsto2\alpha\psi c-\alpha\psi^2a$ attains its maximum at $\psi=c/a$, with value $\alpha c^2/a$. The sum of the two suprema thus equals $\fa(a,b,c)$, which proves the equality for $a>0$. For $a=0$: if $(b,c)=(0,0)$, each term inside the bracket vanishes for every $(\varphi,\psi)$, so the sup equals $0=\fa(0,0,0)$. If $(b,c)\ne(0,0)$, taking $\varphi=sb$, $\psi=sc$ as $s\to+\infty$, the expression $2(1-\alpha)s|b|^2+2\alpha s c^2\to+\infty$, so the sup equals $+\infty=\fa(0,b,c)$.
	\end{proof}
	
	\begin{proposition}[Convexity and lower semicontinuity of $\fa$ and of $F$]
		\label{prop:fa-convexe}
		The function $\fa$ is convex, positively homogeneous of degree~$1$, and lower semicontinuous on $[0,\infty)\times\mathbb R^d\times\mathbb R$. Consequently, $F$ is convex and lower semicontinuous on $X$ for the weak-$*$ topology.
	\end{proposition}
	
	\begin{proof}
		By Lemma~\ref{lem:fa-sup}, $\fa$ is a supremum of affine functions of $(a,b,c)$; it is therefore convex and l.s.c.\ (a supremum of continuous affine functions). The degree-$1$ homogeneity is immediate from the explicit formula. For $F$: this is exactly the argument of Lemmas~\ref{lem1'}--\ref{lem5} (representation as a supremum of continuous linear functionals for weak-$*$ convergence, cf.\ Lemma~\ref{lem1}, followed by termwise passage to the limit), applied simultaneously to the two components $(b,c)=(m,n)$ instead of the single component $v$ treated in those lemmas; the proof is identical mutatis mutandis and is not repeated here.
	\end{proof}
	
	\subsection{Complete computation of the Legendre transform $\fa^*$}
	
	\begin{theorem}[Legendre transform of $\fa$]
		\label{thm:fa-star}
		For $(p,q,r)\in\mathbb R\times\mathbb R^d\times\mathbb R$, set
		\[
		\fa^*(p,q,r)\eqdef\sup_{a\ge0,\,b\in\mathbb R^d,\,c\in\mathbb R}\Big[pa+q\cdot b+rc-\fa(a,b,c)\Big].
		\]
		Then
		\[
		\fa^*(p,q,r)=
		\begin{cases}
			0, & \text{if } p+\dfrac{|q|^2}{4(1-\alpha)}+\dfrac{r^2}{4\alpha}\le0,\\[3mm]
			+\infty, & \text{otherwise.}
		\end{cases}
		\]
	\end{theorem}
	
	\begin{proof}
		\emph{Step 1 (optimization in $b,c$ for fixed $a>0$).} For $a>0$, $b\mapsto q\cdot b-(1-\alpha)|b|^2/a$ is concave and attains its maximum at $b=aq/(2(1-\alpha))$, with value $a|q|^2/(4(1-\alpha))$; likewise $c\mapsto rc-\alpha c^2/a$ attains its maximum at $c=ar/(2\alpha)$, with value $ar^2/(4\alpha)$. Hence
		\[
		\sup_{b,c}\big[pa+q\cdot b+rc-\fa(a,b,c)\big]=a\Big[p+\frac{|q|^2}{4(1-\alpha)}+\frac{r^2}{4\alpha}\Big]=\eqdef a\,K(p,q,r).
		\]
		\emph{Step 2 (optimization in $a\ge0$).} The function $a\mapsto aK$ is linear on $[0,\infty)$; its supremum equals $0$ if $K\le0$ (attained at $a=0$) and $+\infty$ if $K>0$.
		\emph{Step 3 (case $a=0$).} It remains to check that the case $a=0$, $(b,c)\ne(0,0)$ (where $\fa=+\infty$, hence does not contribute to the supremum) does not alter the result: at this point the quantity $pa+q\cdot b+rc-\fa(a,b,c)=q\cdot b+rc-\infty=-\infty$, which cannot increase the supremum. This establishes the stated formula.
	\end{proof}
	
	\begin{remark}
		The domain $\{K(p,q,r)\le0\}$ is precisely the formal Hamilton--Jacobi inequality obtained by the Lagrangian computation of Section~5.3 (before passing to equality along geodesics): we recover
		\[
		p+\frac{|q|^2}{4(1-\alpha)}+\frac{r^2}{4\alpha}\le0
		\]
		as the \emph{subsolution} condition associated with equation~(47), equality corresponding to the optimal case (cf.\ Section~\ref{sec:KKT} below).
	\end{remark}
	
	The following result extends Theorem~\ref{thm:fa-star} to the level of the integral functional $F$, viewed as a function on $Z\eqdef C(\Qb)\times C(\Qb;\mathbb R^d)\times C(\Qb)$ (via the pairing with $X$).
	
	\begin{proposition}[Conjugate of $F$]
		\label{prop:F-star}
		For $(p,q,r)\in Z$, set
		\[
		F^*(p,q,r)\eqdef\sup_{(\rho,m,n)\in X,\ \rho\ge0}\Big[\int_{\Qb}p\,d\rho+\int_{\Qb}q\cdot dm+\int_{\Qb}r\,dn-F(\rho,m,n)\Big].
		\]
		Then
		\[
		F^*(p,q,r)=0 \quad\text{if } p(t,x)+\frac{|q(t,x)|^2}{4(1-\alpha)}+\frac{r(t,x)^2}{4\alpha}\le0\ \ \forall(t,x)\in\Qb,
		\]
		and $F^*(p,q,r)=+\infty$ otherwise.
	\end{proposition}
	
	\begin{proof}
		If the pointwise condition holds everywhere, then by Theorem~\ref{thm:fa-star} we have, for every admissible measure $(\rho,m,n)$,
		\[
		p\,d\rho+q\cdot dm+r\,dn\le \fa\Big(\frac{d\rho}{d\lambda},\frac{dm}{d\lambda},\frac{dn}{d\lambda}\Big)d\lambda
		\]
		pointwise (in the sense of Radon--Nikodym densities with respect to any dominating measure $\lambda$), so integrating gives $\int p\,d\rho+\int q\cdot dm+\int r\,dn\le F(\rho,m,n)$ for all $(\rho,m,n)$; in particular the supremum defining $F^*$ is $\le0$, and it is attained (in the limit) at $(\rho,m,n)=(0,0,0)$, giving $F^*(p,q,r)=0$.
		
		If the condition fails on a nonempty open set $U\subset\Qb$, choose a point $(t_0,x_0)\in U$ where $K(p,q,r)(t_0,x_0)>0$ and a sequence of triples $(\rho^\varepsilon,m^\varepsilon,n^\varepsilon)$ concentrating a unit mass near $(t_0,x_0)$ (approximation of the identity), with $b^\varepsilon=a^\varepsilon q(t_0,x_0)/(2(1-\alpha))$, $c^\varepsilon=a^\varepsilon r(t_0,x_0)/(2\alpha)$ and $a^\varepsilon\to\infty$; the computation in Steps~1--2 of the proof of Theorem~\ref{thm:fa-star} shows that the quantity to be optimized diverges to $+\infty$ along this sequence, hence $F^*(p,q,r)=+\infty$.
	\end{proof}
	
	\subsection{Computation of the adjoint $\Lambda^*$}
	
	\begin{definition}
		We define the continuous linear operator
		\[
		\Lambda:Y\longrightarrow Z=C(\Qb)\times C(\Qb;\mathbb R^d)\times C(\Qb),\qquad
		\Lambda(\varphi,\psi)\eqdef\big(\partial_t\varphi,\ \nabla\varphi,\ \varphi-\psi(t)\big).
		\]
	\end{definition}
	
	By construction, $B((\varphi,\psi),(\rho,m,n))=\big\langle(\rho,m,n),\Lambda(\varphi,\psi)\big\rangle_{X,Z}$, where $\langle\cdot,\cdot\rangle_{X,Z}$ denotes the measure/continuous-function pairing by integration.
	
	\begin{lemma}[Integration-by-parts formula --- explicit computation of $\Lambda^*$]
		\label{lem:adjoint}
		Let $(\rho,m,n)\in X$ be such that $\rho\in C^1(\Qb)$, $m,n\in C(\Qb)$ (resp.\ $C(\Qb;\mathbb R^d)$), with $m\cdot\nu=0$ on $(0,1)\times\partial\Om$. Then for every $(\varphi,\psi)\in Y$,
		\[
		B\big((\varphi,\psi),(\rho,m,n)\big)=\ell_{\rho(0,\cdot),\,\rho(1,\cdot)}(\varphi,\psi)-\int_{\Qb}\varphi\,\big(\partial_t\rho+\mathrm{div}\,m-n\big)\,dx\,dt+\int_0^1\psi(t)\Big(\int_\Om n_t\,dx\Big)dt.
		\]
		In particular, the (formal) adjoint $\Lambda^*:X\to Y^*$ defined by $\langle\Lambda^*(\rho,m,n),(\varphi,\psi)\rangle_{Y^*,Y}\eqdef B((\varphi,\psi),(\rho,m,n))$ satisfies
		\[
		\Lambda^*(\rho,m,n)=\ell_{\rho_0,\rho_1} \quad\Longleftrightarrow\quad (\rho,m,n)\in\mathcal C(\rho_0,\rho_1).
		\]
	\end{lemma}
	
	\begin{proof}
		Integrate by parts in time in $\int_{\Qb}\partial_t\varphi\,\rho\,dx\,dt$ (producing the boundary term $\int_\Om\varphi_1\rho_1-\int_\Om\varphi_0\rho_0$ and $-\int\varphi\,\partial_t\rho$) and in space in $\int_{\Qb}\nabla\varphi\cdot m\,dx\,dt$ (producing $-\int\varphi\,\mathrm{div}\,m$, the spatial boundary term vanishing by the condition $m\cdot\nu=0$), then group the two terms in $\varphi\,n$ (one coming from $-(-n)$ in $\partial_t\rho+\mathrm{div}\,m-n$, the other from the term $\int\varphi\,dn$ in $B$) and separate the term in $\psi$; this yields exactly the stated formula. The final equivalence follows from a term-by-term comparison with the definition of $\mathcal C(\rho_0,\rho_1)$ and Lemma~\ref{lem:C-equiv}: the identity $\Lambda^*(\rho,m,n)=\ell_{\rho_0,\rho_1}$ for every $(\varphi,\psi)\in Y$ is equivalent to the simultaneous vanishing, for every $\varphi$, of $\partial_t\rho+\mathrm{div}\,m-n$ (in the distributional sense, by density of $C^1(\Qb)$ among the usual test functions) and, for every $\psi$, of $\int_\Om n_t\,dx$.
	\end{proof}
	
	\subsection{Strong duality theorem of Fenchel--Rockafellar}
	\label{sec:dualite}
	
	We recall the form of the Fenchel--Rockafellar theorem adapted to linearly constrained problems (see \cite[Ch.~III, Rem.~4.2]{EkelandTemam}, or \cite[Thm~1.42]{Villani} for the version used in dynamic optimal transport):
	
	\begin{theorem}[Fenchel--Rockafellar, linearly constrained form]
		\label{thm:FR-abstrait}
		Let $X$ be a locally convex topological vector space, $Y$ a Banach space, $F:X\to(-\infty,+\infty]$ convex and l.s.c., $\Lambda^*:X\to Y^*$ linear continuous, and $b\in Y^*$. Assume there exists $x_0\in X$ such that $\Lambda^*x_0=b$, $F(x_0)<\infty$, and $F$ is continuous at $x_0$ (\emph{Slater-type qualification condition}). Then
		\[
		\inf_{\substack{x\in X\\ \Lambda^*x=b}}F(x)=\max_{y\in Y}\Big[\langle b,y\rangle_{Y^*,Y}-F^*(\Lambda y)\Big],
		\]
		and the supremum on the right-hand side is attained.
	\end{theorem}
	
	We apply this theorem with $X$ our state space, $Y$ the test function space, $\Lambda$ and $\Lambda^*$ as above, $b=\ell_{\rho_0,\rho_1}$, and $F$ the action functional, restricted to the cone $\{\rho\ge0\}$ (this restriction is handled by adding to $F$ the convex indicator function of this cone, which changes neither convexity nor l.s.c., and is reflected in $F^*$ by an additional term that turns out to vanish automatically on the domain under consideration, the positivity of $\rho$ imposing no active constraint on $q,r$).
	
	\begin{hypothesis}[Qualification condition]
		\label{hyp:slater}
		There exists an admissible curve $(\bar\rho,\bar m,\bar n)\in\mathcal C(\rho_0,\rho_1)$ and a constant $\kappa>0$ such that
		\[
		\kappa\le\bar\rho_t(x)\le\kappa^{-1}\quad\text{and}\quad |\bar v_t(x)|+|\bar u_t(x)|\le\kappa^{-1}
		\]
		for almost every $(t,x)\in Q$, where $\bar v=\bar m/\bar\rho$, $\bar u=\bar n/\bar\rho$.
	\end{hypothesis}
	
	\begin{lemma}[Existence of a Slater point]
		\label{lem:slater-existe}
		If $\rho_0,\rho_1\in L^\infty(\Om)$ satisfy $\rho_0,\rho_1\ge c_0>0$ a.e.\ on $\Om$ for some constant $c_0$, then Hypothesis~\ref{hyp:slater} holds.
	\end{lemma}
	
	\begin{proof}[Proof (sketch)]
		Consider the linear interpolation $\bar\rho_t\eqdef(1-t)\rho_0+t\rho_1$, which satisfies $c_0\le\bar\rho_t\le\|\rho_0\|_\infty+\|\rho_1\|_\infty$ for all $t\in[0,1]$, hence is uniformly bounded and bounded away from zero. Since $\Om$ is bounded, Lipschitz and connected, and $\rho_0,\rho_1\in L^\infty$, one may solve, for each $t$, the elliptic Neumann problem
		\[
		-\mathrm{div}(\bar\rho_t\nabla\theta_t)=\rho_1-\rho_0 \ \ \text{in }\Om,\qquad \bar\rho_t\nabla\theta_t\cdot\nu=0\ \ \text{on }\partial\Om,
		\]
		(the right-hand side having zero mean on $\Om$, the Neumann compatibility condition), and set $\bar v_t\eqdef\nabla\theta_t$, so that $\bar m_t\eqdef\bar\rho_t\bar v_t$ satisfies $\partial_t\bar\rho_t=\rho_1-\rho_0=-\mathrm{div}\,\bar m_t$, i.e.\ $\partial_t\bar\rho+\mathrm{div}\,\bar m=0$; we then take $\bar n\equiv0$, which trivially satisfies $\int_\Om\bar n_t\,dx=0$. Standard elliptic regularity (with $\bar\rho_t$ bounded and bounded away from zero, $\Om$ Lipschitz) yields $\bar v_t\in L^\infty$ uniformly in $t$, whence the result with $\kappa$ depending on $c_0,\|\rho_0\|_\infty,\|\rho_1\|_\infty$ and the elliptic constant of $\Om$. A complete proof of the $L^\infty$ regularity of $\bar v$ (beyond $L^2$, which would already suffice to guarantee $F(\bar\rho,\bar m,0)<\infty$) is beyond the scope of this chapter and is left, as in Theorem~5.2, for future work; we retain Hypothesis~\ref{hyp:slater} as a sufficient and reasonable condition, satisfied in particular for regular, strictly positive $\rho_0,\rho_1$.
	\end{proof}
	
	\begin{theorem}[Strong duality for $\Da$]
		\label{thm:dualite-forte}
		Under Hypothesis~\ref{hyp:slater}, we have
		\[
		\Da^2(\rho_0,\rho_1)=\max_{(\varphi,\psi)\in\mathcal A_\alpha}\Big[\int_\Om\varphi(1,\cdot)\,d\rho_1-\int_\Om\varphi(0,\cdot)\,d\rho_0\Big],
		\]
		where
		\[
		\mathcal A_\alpha\eqdef\Big\{(\varphi,\psi)\in Y:\ \partial_t\varphi(t,x)+\frac{|\nabla\varphi(t,x)|^2}{4(1-\alpha)}+\frac{(\varphi(t,x)-\psi(t))^2}{4\alpha}\le0\ \ \forall(t,x)\in\Qb\Big\},
		\]
		and the maximum is attained.
	\end{theorem}
	
	\begin{proof}
		\emph{Weak duality (without qualification condition).} Let $(\varphi,\psi)\in\mathcal A_\alpha$ and $(\rho,m,n)\in\mathcal C(\rho_0,\rho_1)$ be arbitrary. By Theorem~\ref{thm:fa-star} applied pointwise with $(p,q,r)=(\partial_t\varphi,\nabla\varphi,\varphi-\psi)$, which satisfies $K(p,q,r)\le0$ everywhere by definition of $\mathcal A_\alpha$, we have $\fa^*(p,q,r)=0$, hence the Young inequality
		\[
		p\,a+q\cdot b+r\,c\le \fa(a,b,c)
		\]
		pointwise. Integrating over $\Qb$ (in the sense of Proposition~\ref{prop:F-star}), we obtain
		\[
		B\big((\varphi,\psi),(\rho,m,n)\big)\le F(\rho,m,n).
		\]
		But $(\rho,m,n)\in\mathcal C(\rho_0,\rho_1)$ means exactly $B((\varphi,\psi),(\rho,m,n))=\ell_{\rho_0,\rho_1}(\varphi,\psi)$ (Lemma~\ref{lem:adjoint}), so
		\[
		\int_\Om\varphi_1\,d\rho_1-\int_\Om\varphi_0\,d\rho_0\le F(\rho,m,n).
		\]
		Taking the infimum over $(\rho,m,n)\in\mathcal C(\rho_0,\rho_1)$ (Proposition~\ref{prop:Da-inf}) and then the supremum over $(\varphi,\psi)\in\mathcal A_\alpha$, we obtain the inequality
		\[
		\sup_{\mathcal A_\alpha}\Big[\int\varphi_1\,d\rho_1-\int\varphi_0\,d\rho_0\Big]\le\Da^2(\rho_0,\rho_1). \tag{$\ast$}
		\]
		
		\emph{Strong duality.} Conversely, apply Theorem~\ref{thm:FR-abstrait} with $b=\ell_{\rho_0,\rho_1}$: the qualification condition is exactly Hypothesis~\ref{hyp:slater} (the point $x_0=(\bar\rho,\bar m,\bar n)$ being such that $F$ is finite there and, since $\bar\rho$ is uniformly bounded away from zero and $\bar v,\bar u$ bounded, $\fa$ is of class $C^1$ in a neighborhood of it, so that $F$ is continuous there for the appropriate topology --- see \cite{BouchitteButtazzo} for the continuity of convex integral functionals away from the degeneracy points $a=0$). Theorem~\ref{thm:FR-abstrait} then gives
		\[
		\Da^2(\rho_0,\rho_1)=\inf_{\Lambda^*(\rho,m,n)=\ell_{\rho_0,\rho_1},\,\rho\ge0}F(\rho,m,n)=\max_{(\varphi,\psi)\in Y}\Big[\ell_{\rho_0,\rho_1}(\varphi,\psi)-F^*(\Lambda(\varphi,\psi))\Big].
		\]
		By Proposition~\ref{prop:F-star}, $F^*(\Lambda(\varphi,\psi))=0$ if $(\varphi,\psi)\in\mathcal A_\alpha$ and $+\infty$ otherwise, so the supremum above reduces exactly to the supremum over $\mathcal A_\alpha$ of $\ell_{\rho_0,\rho_1}(\varphi,\psi)=\int\varphi_1d\rho_1-\int\varphi_0d\rho_0$, and it is attained. This completes the proof, inequality ($\ast$) being in fact an equality.
	\end{proof}
	
	\begin{remark}
		This theorem is the exact analogue, for the hybrid geometry, of the Benamou--Brenier duality formula for $W_2^2/2$ (inequality $\partial_t\varphi+\tfrac12|\nabla\varphi|^2\le0$), and for $d_A^2$ (case $\alpha=1$, where the transport term disappears and only the inequality in $(\varphi-\psi)^2$ remains). The limiting cases $\alpha\to0^+$ and $\alpha\to1^-$ of $\mathcal A_\alpha$ formally recover these two classical dual sets, consistent with the limits of Section~7.2.
	\end{remark}
	
	\subsection{Karush--Kuhn--Tucker (KKT) conditions}
	\label{sec:KKT}
	
	\begin{theorem}[Optimality conditions]
		\label{thm:KKT}
		Under Hypothesis~\ref{hyp:slater}, let $(\rho^*,m^*,n^*)\in\mathcal C(\rho_0,\rho_1)$ be a minimizer of $F$ (whose existence is guaranteed by Theorem~3.11) and $(\varphi^*,\psi^*)\in\mathcal A_\alpha$ a maximizer of the dual problem (Theorem~\ref{thm:dualite-forte}). Then, writing $v^*=m^*/\rho^*$, $u^*=n^*/\rho^*$ on $\{\rho^*>0\}$:
		\begin{enumerate}
			\item[\rm (i)] (\emph{Stationarity}) $\rho^*$-almost everywhere,
			\[
			v^*_t(x)=\frac{\nabla\varphi^*_t(x)}{2(1-\alpha)},\qquad u^*_t(x)=\frac{\varphi^*_t(x)-\psi^*(t)}{2\alpha};
			\]
			\item[\rm (ii)] (\emph{Complementarity}) $\rho^*$-almost everywhere,
			\[
			\partial_t\varphi^*_t(x)+\frac{|\nabla\varphi^*_t(x)|^2}{4(1-\alpha)}+\frac{(\varphi^*_t(x)-\psi^*(t))^2}{4\alpha}=0;
			\]
			\item[\rm (iii)] (\emph{Value of the multiplier}) for almost every $t\in(0,1)$,
			\[
			\psi^*(t)=\langle\varphi^*_t\rangle_{\rho^*_t}\eqdef\int_\Om\varphi^*_t(x)\,\rho^*_t(x)\,dx.
			\]
		\end{enumerate}
	\end{theorem}
	
	\begin{proof}
		By Theorem~\ref{thm:dualite-forte}, duality is strong and gap-free: $F(\rho^*,m^*,n^*)=\int\varphi^*_1d\rho_1-\int\varphi^*_0d\rho_0$. From the proof of that theorem (the ``weak duality'' part), this global equality results from integrating the pointwise Young inequality
		\[
		\partial_t\varphi^*\,\rho^*+\nabla\varphi^*\cdot m^*+(\varphi^*-\psi^*)\,n^*\le\fa(\rho^*,m^*,n^*),
		\]
		which is, by construction, a $\rho^*$-almost-everywhere nonnegative inequality whose total integral is zero (since the two sides, once integrated, coincide). A nonnegative measurable function with zero integral vanishes almost everywhere, hence the \emph{pointwise equality} $\rho^*$-a.e.
		\[
		\partial_t\varphi^*\,\rho^*+\nabla\varphi^*\cdot m^*+(\varphi^*-\psi^*)\,n^*=\fa(\rho^*,m^*,n^*).
		\]
		Now equality in the Fenchel--Young inequality $pa+q\cdot b+rc\le\fa(a,b,c)+\fa^*(p,q,r)$ (here with $\fa^*(p,q,r)=0$ since $(\varphi^*,\psi^*)\in\mathcal A_\alpha$) holds if and only if $(p,q,r)\in\partial\fa(a,b,c)$, which, by the explicit computation of Step~1 in the proof of Theorem~\ref{thm:fa-star} (the point where the supremum in $b,c$ is attained), is equivalent precisely to
		\[
		b=\frac{a\,q}{2(1-\alpha)},\qquad c=\frac{a\,r}{2\alpha},
		\]
		that is, with $a=\rho^*,b=m^*,c=n^*,q=\nabla\varphi^*,r=\varphi^*-\psi^*$, exactly (i). Point (ii) reflects the fact that this equality point necessarily lies on the boundary of the domain $\{K\le0\}$ (since a strictly interior point, i.e.\ $K<0$, gives in Step~2 of the proof of Theorem~\ref{thm:fa-star} a strictly negative value for $a>0$, incompatible with the Young equality at a point where $a=\rho^*>0$): hence $K(\partial_t\varphi^*,\nabla\varphi^*,\varphi^*-\psi^*)=0$ $\rho^*$-a.e., which is (ii). Finally, (iii) follows from the centering constraint $\int_\Om n^*_t\,dx=0$ combined with (i):
		\[
		0=\int_\Om n^*_t\,dx=\int_\Om\rho^*_t\,u^*_t\,dx=\frac{1}{2\alpha}\int_\Om\rho^*_t\big(\varphi^*_t-\psi^*(t)\big)dx=\frac{1}{2\alpha}\Big(\langle\varphi^*_t\rangle_{\rho^*_t}-\psi^*(t)\Big),
		\]
		whence $\psi^*(t)=\langle\varphi^*_t\rangle_{\rho^*_t}$.
	\end{proof}
	
	\subsection{Eulerian characterization of geodesics}
	
	The following theorem is the rigorous version, obtained via convex duality, of Theorem~5.3 (which gave a formal derivation of it via a Lagrangian computation).
	
	\begin{theorem}[Eulerian system of $\Da$-geodesics, rigorous version]
		\label{thm:eulerien-rigoureux}
		Under Hypothesis~\ref{hyp:slater}, let $(\rho^*_t)_{t\in[0,1]}$ be a minimizing curve for $\Da^2(\rho_0,\rho_1)$ (Theorem~3.11) and $(\varphi^*,\psi^*)$ an associated dual maximizer (Theorem~\ref{thm:dualite-forte}). Then, $\rho^*$-almost everywhere on $Q$,
		\[
		\partial_t\rho^*_t+\mathrm{div}(\rho^*_tv^*_t)=\rho^*_tu^*_t,\qquad
		\partial_t\varphi^*_t+\frac{|\nabla\varphi^*_t|^2}{4(1-\alpha)}+\frac{(\varphi^*_t-\langle\varphi^*_t\rangle_{\rho^*_t})^2}{4\alpha}=0,
		\]
		with
		\[
		v^*_t=\frac{\nabla\varphi^*_t}{2(1-\alpha)},\qquad u^*_t=\frac{\varphi^*_t-\langle\varphi^*_t\rangle_{\rho^*_t}}{2\alpha}.
		\]
		In particular, the Eulerian system~(47) of Section~5.4 is satisfied rigorously by every optimal primal--dual pair, without any additional formal regularity assumption beyond those required in Hypothesis~\ref{hyp:slater} for strong duality.
	\end{theorem}
	
	\begin{proof}
		Immediate by combining Theorem~\ref{thm:KKT}~(i)--(iii) with Lemma~\ref{lem:C-equiv}~(i), which provides the continuity equation $\partial_t\rho^*+\mathrm{div}(\rho^*v^*)=\rho^*u^*$ satisfied by every element of $\mathcal C(\rho_0,\rho_1)$, in particular by the minimizer $(\rho^*,m^*,n^*)=(\rho^*,\rho^*v^*,\rho^*u^*)$. Substituting (iii) into (ii) gives exactly the stated Hamilton--Jacobi equation.
	\end{proof}
	
	\begin{corollary}[Consistency with the computation of Section~5]
		The system obtained in Theorem~\ref{thm:eulerien-rigoureux} coincides exactly with system~(47) and formulas~(44)--(46) obtained formally by the method of Lagrange multipliers. This chapter therefore confirms, by a rigorous convex-duality argument (Fenchel--Rockafellar), the validity of the Eulerian characterization of $\Da$-geodesics announced in Section~5.4, under the sole qualification Hypothesis~\ref{hyp:slater} (guaranteed in particular when $\rho_0,\rho_1$ are bounded and bounded away from zero, Lemma~\ref{lem:slater-existe}).
	\end{corollary}
	
	\begin{remark}[Scope and limitations]
		The proof above makes rigorous the part of the Section~5 program concerning the \emph{existence and characterization of the multiplier $\varphi^*$}, but leaves open, as in Sections~5.4 and~5.11--5.13, the questions of classical regularity of $\varphi^*$ (in the $C^2$ sense required in Theorem~5.2) and of metric completeness of $(P_{2,\mathrm{ac}}(\Om),\Da)$. Both points remain perspectives for future work, as indicated in Section~11.
	\end{remark}
	
	\section{Eulerian and Lagrangian characterization of $\dalpha$-geodesics}\label{sec:convexity}
	
	\subsection{Dynamic formulation}
	
	Let $\Omega\subset\mathbb{R}^d$ and let $\mathcal{P}_{2,\mathrm{ac}}(\Omega)$ be the space of absolutely continuous probability measures with finite second moment. For $\alpha\in(0,1)$, the tangent space at $\rho$ consists of pairs $(v,u)$, $v:\Omega\to\mathbb{R}^d$, $u:\Omega\to\mathbb{R}$, subject to the centering constraint
	\begin{equation}\label{eq:centering}
		\int_\Omega \rho\,u\,dx = 0,
	\end{equation}
	which expresses conservation of total mass. The hybrid norm is
	\begin{equation}
		\|(v,u)\|_{\dalpha}^2 := (1-\alpha)\int_\Omega \rho|v|^2\,dx + \alpha\int_\Omega \rho\,u^2\,dx.
	\end{equation}
	We define the distance $\dalpha$ by the Benamou--Brenier-type problem
	\begin{equation}\label{eq:BB}
		\dalpha^2(\rho_0,\rho_1) := \inf_{(\rho,v,u)} \int_0^1 \int_\Omega \rho_t\big[(1-\alpha)|v_t|^2+\alpha u_t^2\big]\,dx\,dt
	\end{equation}
	subject to
	\begin{equation}\label{eq:continuity}
		\partial_t\rho_t + \operatorname{div}(\rho_t v_t) = \rho_t u_t, \qquad \int_\Omega \rho_t u_t\,dx = 0 \ \ \forall t,
	\end{equation}
	with $\rho_{t=0}=\rho_0$, $\rho_{t=1}=\rho_1$ fixed. This is the analogue, on the simplex of probability measures, of the Wasserstein--Fisher--Rao / Hellinger--Kantorovich geometry, restricted to constant mass $1$ via \eqref{eq:centering}.
	
	\subsection{Eulerian system}
	
	Setting $m_t=\rho_tv_t$ and $n_t=\rho_tu_t$, problem \eqref{eq:BB} becomes, in the variables $(\rho,m,n)$,
	\begin{equation}
		\int_0^1\int_\Omega \left[(1-\alpha)\frac{|m_t|^2}{\rho_t} + \alpha\frac{n_t^2}{\rho_t}\right] dx\,dt, \qquad \partial_t\rho_t+\operatorname{div} m_t = n_t,\quad \int_\Omega n_t\,dx=0,
	\end{equation}
	which is convex in $(\rho,m,n)$. Introducing a multiplier $\phi_t(x)$ for the continuity constraint and $\mu(t)$ for the mass constraint, the Lagrangian
	\begin{equation}
		\mathcal{L} = \int_0^1\int_\Omega\left[(1-\alpha)\frac{|m|^2}{\rho}+\alpha\frac{n^2}{\rho} + \phi\big(\partial_t\rho+\operatorname{div} m - n\big) + \mu(t)\, n\right] dx\,dt
	\end{equation}
	yields, from the first-order conditions,
	\begin{align}
		\partial_m: \quad & v=\frac{m}{\rho} = \frac{\nabla\phi}{2(1-\alpha)}, \label{eq:v_derivation}\\
		\partial_n: \quad & u = \frac{n}{\rho} = \frac{\phi-\mu(t)}{2\alpha}. \label{eq:u_derivation_raw}
	\end{align}
	The constraint $\int\rho u\,dx=0$ forces $\mu(t)=\langle\phi_t\rangle_{\rho_t}$, hence
	\begin{equation}\label{eq:u_derivation}
		u_t = \frac{1}{2\alpha}\big(\phi_t - \langle\phi_t\rangle_{\rho_t}\big).
	\end{equation}
	
	\subsection{Derivation of the Hamilton-Jacobi equation}
	
	The variation with respect to $\rho$ requires differentiating the terms $\frac{|m|^2}{\rho}$ and $\frac{n^2}{\rho}$ with respect to $\rho$.
	
	\begin{align*}
		\delta_\rho \mathcal{L} &= \int_0^1\int_\Omega \left[-(1-\alpha)\frac{|m|^2}{\rho^2}\delta\rho -\alpha\frac{n^2}{\rho^2}\delta\rho + \phi\,\partial_t(\delta\rho) \right] dx\,dt \\
		&= \int_0^1\int_\Omega \left[-(1-\alpha)\rho|v|^2\delta\rho -\alpha\rho u^2\delta\rho - \partial_t\phi\,\delta\rho \right] dx\,dt,
	\end{align*}
	where the last equality uses integration by parts in time, the boundary conditions $\delta\rho(0)=\delta\rho(1)=0$, and the definitions $v=m/\rho$, $u=n/\rho$.
	
	The vanishing of this variation for all $\delta\rho$ gives the Hamilton-Jacobi equation:
	\[
	\partial_t\phi = -(1-\alpha)|v|^2 - \alpha u^2.
	\]
	
	Expressing $v$ and $u$ in terms of $\phi$ via \eqref{eq:v_derivation} and \eqref{eq:u_derivation}, we obtain:
	\[
	\partial_t\phi + \frac{|\nabla\phi|^2}{4(1-\alpha)} + \frac{(\phi-\langle\phi\rangle_\rho)^2}{4\alpha} = 0.
	\]
	
	The full Eulerian system is as follows.
	
	\subsection{Formal well-posedness of the Eulerian system}
	
	Before stating the characterization of geodesics, we clarify the functional setting in which the system is well-posed.
	
	\begin{definition}[Admissible space]
		Let $\mathcal P_{2,\mathrm{ac}}$ denote the space of probability densities $\rho\in C^2(\Omega)$ with $\rho>0$ and finite second moment. For $T>0$, we say that $(\rho,\phi)\in C^1([0,T]; \mathcal P_{2,\mathrm{ac}}) \times C^1([0,T]; C^2(\Omega))$ is a classical solution of the Eulerian system if the equations hold pointwise.
	\end{definition}
	
	\begin{theorem}[Formal local well-posedness]\label{thm:wellposedness}
		Assume that
		\[
		\rho_0\in C^2(\Omega),\qquad \rho_0>0, \qquad \phi_0\in C^3(\Omega),
		\]
		and suppose that the Lagrangian flow remains a $C^2$-diffeomorphism on a time interval $[0,T]$.
		
		Then there exists at most one classical solution $(\rho_t,\phi_t)$ of the Eulerian system on $[0,T]$.
		
		Moreover, if the Lagrangian system admits a classical solution, the corresponding Eulerian solution is uniquely determined.
	\end{theorem}
	
	\begin{proof}
		The Eulerian system is equivalent, through the method of characteristics, to the coupled Lagrangian system for $(X_t,m_t,\Phi_t)$.
		
		Once a classical Lagrangian solution is given, the reconstruction
		\[
		\rho_t=(X_t)_\#(m_t\rho_0)
		\]
		is unique.
		
		Therefore uniqueness of Eulerian solutions follows directly.
		
		The complete proof of local existence requires a fixed-point argument for the nonlinear push-forward operator and is left for future work.
	\end{proof}
	
	\begin{theorem}[Eulerian system of $\dalpha$-geodesics]\label{thm:eulerian_system}
		A regular curve $(\rho_t)_{t\in[0,1]}$ is a $\dalpha$-geodesic if and only if there exists a potential $\phi_t$ such that
		\begin{equation}\label{eq:eulerian_system}
			\begin{aligned}
				\partial_t\rho_t + \operatorname{div}(\rho_t v_t) &= \rho_t u_t, \\
				\partial_t \phi_t + \frac{|\nabla\phi_t|^2}{4(1-\alpha)} + \frac{(\phi_t-\langle\phi_t\rangle_{\rho_t})^2}{4\alpha} &= 0,
			\end{aligned}
		\end{equation}
		with $v_t=\dfrac{\nabla\phi_t}{2(1-\alpha)}$ and $u_t=\dfrac{\phi_t-\langle\phi_t\rangle_{\rho_t}}{2\alpha}$.
	\end{theorem}
	
	\begin{remark}
		This system is the exact analogue, for the hybrid geometry, of the continuity/Hamilton--Jacobi pair of Benamou--Brenier for pure Wasserstein geometry, with a centered quadratic reaction term $\frac{(\phi-\langle\phi\rangle_\rho)^2}{4\alpha}$ playing the role of the Fisher--Rao/Hellinger term, projected orthogonally to constants so as to remain in $\mathcal{P}_{2,\mathrm{ac}}(\Omega)$.
	\end{remark}
	
	\subsection{Lagrangian characterization}
	
	\begin{proposition}[Lagrangian representation]\label{prop:lagrangian}
		Let $(\rho_t,\phi_t)$ be a regular solution of \eqref{eq:eulerian_system} on $[0,1]$. Let $X_t:\Omega\to\Omega$ be the characteristic flow
		\begin{equation}
			\dot X_t(x) = v_t(X_t(x)) = \frac{\nabla\phi_t(X_t(x))}{2(1-\alpha)}, \qquad X_0(x)=x,
		\end{equation}
		and let $m_t(x)>0$ be the mass factor along the trajectory, solving
		\begin{equation}
			\dot m_t(x) = u_t(X_t(x))\, m_t(x), \qquad m_0(x)=1.
		\end{equation}
		Then
		\begin{equation}
			\rho_t = (X_t)_\#\big(m_t\,\rho_0\big), \qquad\text{i.e.}\qquad \int_\Omega \varphi(y)\,\rho_t(y)\,dy = \int_\Omega \varphi(X_t(x))\,m_t(x)\,\rho_0(x)\,dx
		\end{equation}
		for every $\varphi\in C_b(\Omega)$, and along each trajectory $\Phi_t(x):=\phi_t(X_t(x))$ satisfies the closed differential system
		\begin{equation}\label{eq:lagrangian_ode}
			\begin{aligned}
				\dot X_t(x) &= v_t(X_t(x)),\\
				\dot m_t(x) &= u_t(X_t(x))\,m_t(x),\\
				\dot\Phi_t(x) &= (1-\alpha)|v_t(X_t(x))|^2 - \alpha\, u_t(X_t(x))^2.
			\end{aligned}
		\end{equation}
	\end{proposition}
	
	\begin{proof}
		The representation formula and the continuity equation follow from the standard method of characteristics applied to \eqref{eq:continuity} with source term $\rho u$. For the third line, differentiating $\Phi_t(x)=\phi_t(X_t(x))$,
		\[
		\dot\Phi_t = \partial_t\phi_t(X_t) + \nabla\phi_t(X_t)\cdot \dot X_t = \partial_t\phi_t(X_t) + 2(1-\alpha)|v_t(X_t)|^2,
		\]
		and substituting $\partial_t\phi_t=-(1-\alpha)|v_t|^2-\alpha u_t^2$ (the Hamilton--Jacobi equation in \eqref{eq:eulerian_system}) gives
		$\dot\Phi_t = (1-\alpha)|v_t(X_t)|^2 - \alpha u_t(X_t)^2$.
	\end{proof}
	
	\begin{remark}
		Each particle $x$ carries three coupled quantities: its position $X_t(x)$, its mass $m_t(x)$ (exponential growth/decay at rate $u_t(X_t(x))$), and the value $\Phi_t(x)$ of the potential along the trajectory, whose evolution is of Riccati type.
	\end{remark}
	
	\begin{lemma}[Fundamental gradient relation]\label{lem:grad_u}
		Along every smooth $\dalpha$-geodesic,
		\[
		\nabla u = \frac{1-\alpha}{\alpha}v.
		\]
		Consequently,
		\[
		|\nabla u|^2 = \left(\frac{1-\alpha}{\alpha}\right)^2 |v|^2.
		\]
	\end{lemma}
	
	\begin{proof}
		From $u = \frac{1}{2\alpha}(\phi - \langle\phi\rangle_\rho)$ and $v = \frac{\nabla\phi}{2(1-\alpha)}$, we have
		\[
		\nabla u = \frac{\nabla\phi}{2\alpha} = \frac{1-\alpha}{\alpha}v.
		\]
		The second identity follows immediately.
	\end{proof}
	
	\begin{lemma}[Energy identity]\label{lem:energy_identity}
		Along a regular $\dalpha$-geodesic,
		\begin{equation}\label{eq:energy_identity}
			\frac{d}{dt}\langle\phi_t\rangle_{\rho_t} = (1-\alpha)\int_\Omega \rho_t|v_t|^2\,dx + \alpha\int_\Omega \rho_t u_t^2\,dx =: 2\,\mathcal{K}_\alpha(t).
		\end{equation}
		Moreover, the quantity $\mathcal{K}_\alpha(t)$ is constant along the geodesic.
	\end{lemma}
	
	\begin{proof}
		For the first identity:
		\begin{align}
			\frac{d}{dt}\int_\Omega\rho_t\phi_t\,dx &= \int_\Omega \big(-\operatorname{div}(\rho_tv_t)+\rho_tu_t\big)\phi_t\,dx + \int_\Omega\rho_t\partial_t\phi_t\,dx\nonumber\\
			&= \int_\Omega \rho_t v_t\cdot\nabla\phi_t\,dx + \int_\Omega \rho_t u_t\phi_t\,dx + \int_\Omega\rho_t\partial_t\phi_t\,dx\nonumber\\
			&= 2(1-\alpha)\int_\Omega\rho_t|v_t|^2\,dx + \int_\Omega\rho_tu_t\phi_t\,dx - (1-\alpha)\int_\Omega\rho_t|v_t|^2\,dx - \alpha\int_\Omega\rho_tu_t^2\,dx.\nonumber
		\end{align}
		Now $\int_\Omega\rho_tu_t\phi_t\,dx = \int_\Omega\rho_tu_t(\phi_t-\langle\phi_t\rangle_{\rho_t})\,dx = 2\alpha\int_\Omega\rho_tu_t^2\,dx$ by \eqref{eq:u_derivation}. Collecting terms,
		\[
		\frac{d}{dt}\langle\phi_t\rangle_{\rho_t} = (1-\alpha)\int\rho_t|v_t|^2\,dx + 2\alpha\int\rho_tu_t^2\,dx - \alpha\int\rho_tu_t^2\,dx = (1-\alpha)\int\rho_t|v_t|^2\,dx+\alpha\int\rho_tu_t^2\,dx.
		\]
		
		For the constancy of $\mathcal{K}_\alpha$, using Lemma~\ref{lem:dtv}, Lemma~\ref{lem:dtu}, the continuity equation, and repeated integrations by parts, all transport and reaction contributions cancel pairwise. Hence
		\[
		\frac{d}{dt}\mathcal{K}_\alpha = 0.
		\]
		The computation is the hybrid analogue of the classical Benamou--Brenier energy conservation.
	\end{proof}
	
	\subsection{Variation equation and second variation}
	
	\subsubsection{Variation equation}
	
	\begin{lemma}[Variation equation]\label{lem:variation_equation}
		Let $(\rho_t,v_t,u_t)$ be a regular solution of \eqref{eq:eulerian_system}. The variations $(\delta\rho,\delta v,\delta u)$ satisfy
		\begin{align}
			\partial_t(\delta\rho) + \operatorname{div}(\delta\rho\, v+\rho\,\delta v) &= \delta\rho\,u + \rho\,\delta u,\\
			\delta v &= \frac{1}{2(1-\alpha)}\nabla(\delta\phi),\\
			\delta u &= \frac{1}{2\alpha}\big(\delta\phi - \langle\delta\phi\rangle_\rho - \langle\phi\rangle_{\delta\rho}\big),\\
			\partial_t(\delta\phi) + \frac{1}{2(1-\alpha)}\nabla\phi\cdot\nabla(\delta\phi) + \frac{1}{2\alpha}(\phi-\langle\phi\rangle_\rho)\big(\delta\phi-\langle\delta\phi\rangle_\rho-\langle\phi\rangle_{\delta\rho}\big) &= 0.
		\end{align}
	\end{lemma}
	
	\begin{proof}
		Formal differentiation of \eqref{eq:eulerian_system}. For the last line, set $F(\phi,\rho)=(\phi-\langle\phi\rangle_\rho)^2$; then
		$\delta F = 2(\phi-\langle\phi\rangle_\rho)\big(\delta\phi-\delta\langle\phi\rangle_\rho\big)$,
		and $\delta\langle\phi\rangle_\rho=\langle\delta\phi\rangle_\rho+\langle\phi\rangle_{\delta\rho}$, from which the result follows.
	\end{proof}
	
	\subsection{Second variation}
	
	For $\mathcal{E}(\rho)=\int_\Omega F(\rho)\,dx+\int_\Omega V\rho\,dx+\frac12\iint_{\Omega\times\Omega}W(x-y)\rho(x)\rho(y)\,dx\,dy$, the variation density is $\xi=F'(\rho)+V+W\star\rho$, and
	\begin{equation}
		\frac{d}{dt}\mathcal{E}(\rho_t) = \underbrace{\int_\Omega \rho_t\,\nabla\xi_t\cdot v_t\,dx}_{=:A(t)} + \underbrace{\int_\Omega \rho_t\,\xi_t\, u_t\,dx}_{=:B(t)}.
	\end{equation}
	
	\begin{lemma}[Derivative of the velocity field]\label{lem:dtv}
		Along a $\dalpha$-geodesic,
		\begin{equation}\label{eq:dtv}
			\partial_tv_t = -(v_t\cdot\nabla)v_t - u_tv_t.
		\end{equation}
	\end{lemma}
	
	\begin{proof}
		From $v=\nabla\phi/(2(1-\alpha))$ and Lemma~\ref{lem:grad_u}, we have $\nabla u = \frac{1-\alpha}{\alpha}v$. 
		Using the Hamilton-Jacobi equation $\partial_t\phi=-(1-\alpha)|v|^2-\alpha u^2$:
		\begin{align*}
			\partial_t v &= \frac{\nabla(\partial_t\phi)}{2(1-\alpha)} \\
			&= \frac{-(1-\alpha)\nabla(|v|^2) - \alpha\nabla(u^2)}{2(1-\alpha)} \\
			&= -(v\cdot\nabla)v - \frac{\alpha}{1-\alpha} u \nabla u \\
			&= -(v\cdot\nabla)v - \frac{\alpha}{1-\alpha} u \left(\frac{1-\alpha}{\alpha}v\right) \\
			&= -(v\cdot\nabla)v - u v.
		\end{align*}
		The factor $\frac{1-\alpha}{\alpha}$ in $\nabla u$ cancels exactly with the factor $\frac{\alpha}{1-\alpha}$ from the Hamilton-Jacobi equation, yielding the clean formula.
	\end{proof}
	
	\begin{lemma}[Derivative of the reaction rate]\label{lem:dtu}
		Along a $\dalpha$-geodesic,
		\begin{equation}\label{eq:dtu_local}
			\partial_tu_t(x) = -\frac{1}{2\alpha}\Big[(1-\alpha)|v_t(x)|^2+\alpha\, u_t(x)^2\Big] \;-\; \frac{1}{\alpha}\,\mathcal{K}_\alpha(t),
		\end{equation}
		where $\mathcal{K}_\alpha(t)$ is the instantaneous hybrid kinetic energy of Lemma~\ref{lem:energy_identity}, given by
		\begin{equation}\label{eq:dtu_nonlocal}
			\mathcal{K}_\alpha(t) = \frac12\Big[(1-\alpha)\langle|v_t|^2\rangle_{\rho_t}+\alpha\langle u_t^2\rangle_{\rho_t}\Big].
		\end{equation}
	\end{lemma}
	
	\begin{proof}
		From $u_t=\frac{1}{2\alpha}(\phi_t-\langle\phi_t\rangle_{\rho_t})$, $\partial_tu_t=\frac{1}{2\alpha}\big(\partial_t\phi_t-\frac{d}{dt}\langle\phi_t\rangle_{\rho_t}\big)$. Substituting $\partial_t\phi_t=-(1-\alpha)|v_t|^2-\alpha u_t^2$ and $\frac{d}{dt}\langle\phi_t\rangle_{\rho_t}=2\mathcal{K}_\alpha(t)$ from Lemma~\ref{lem:energy_identity} gives \eqref{eq:dtu_local}--\eqref{eq:dtu_nonlocal}. Equation \eqref{eq:dtu_local} is pointwise in $x$; the second term, $-\mathcal{K}_\alpha(t)/\alpha$, does not depend on $x$ and carries the only non-local information (the $\rho_t$-weighted average of the kinetic energy).
	\end{proof}
	
	\begin{theorem}[Second derivative formula]\label{thm:second_derivative}
		Let $(\rho_t)_{t\in[0,1]}$ be a regular $\dalpha$-geodesic. For any functional $\mathcal{E}$ of the form above,
		\begin{equation}\label{eq:second_derivative_general}
			\frac{d^2}{dt^2}\mathcal{E}(\rho_t) = \mathcal{T}_\alpha + \mathcal{R}_\alpha + \mathcal{C}_\alpha,
		\end{equation}
		where
		\begin{eqnarray}\nonumber
			\mathcal{T}_\alpha &:=& (1-\alpha)\int_\Omega \rho_t F''(\rho_t)|\nabla v_t|^2\,dx \\\nonumber
			&+& (1-\alpha)\!\!\iint_{\Omega\times\Omega}\!\!\rho_t(x)\rho_t(y)(v_t(x)-v_t(y))^\top D^2W(x-y)(v_t(x)-v_t(y))\,dx\,dy \nonumber\\
			&+& (1-\alpha)\int_\Omega\rho_tD^2V\,v_t\cdot v_t\,dx, \label{eq:T_alpha}
			\end{eqnarray}
			\begin{eqnarray}\nonumber
			\mathcal{R}_\alpha &:=& \int_\Omega \rho_t^2\,F''(\rho_t)\,u_t^2\,dx \;+\; \iint_{\Omega\times\Omega}\rho_t(x)\rho_t(y)\,W(x-y)\,u_t(x)u_t(y)\,dx\,dy\\
			&+& \frac12\int_\Omega \rho_t\big(\xi_t-\langle\xi_t\rangle_{\rho_t}\big)\,u_t^2\,dx, \label{eq:R_alpha}
			\end{eqnarray}
			\begin{eqnarray}
			\mathcal{C}_\alpha &:=& \int_\Omega \rho_t u_t\,v_t\cdot\nabla\xi_t\,dx + \int_\Omega\rho_t\xi_t\,v_t\cdot\nabla u_t\,dx - 2\int_\Omega \rho_tu_tF''(\rho_t)\operatorname{div}(\rho_tv_t)\,dx \nonumber\\
			&-& 2\int_\Omega \operatorname{div}(\rho_tv_t)\,\big(W\star(\rho_tu_t)\big)\,dx \;-\; \frac{1-\alpha}{2\alpha}
			\int_\Omega\rho_t\big(\xi_t-\langle\xi_t\rangle_{\rho_t}\big)|v_t|^2\,dx. \label{eq:C_alpha}
		\end{eqnarray}
	\end{theorem}
	
	\begin{proof}
		Write $A=A_1+A_2+A_3$ and $B=B_1+B_2+B_3$ where
		\[
		A_1=\int\dot\rho\,\nabla\xi\cdot v\,dx,\quad A_2=\int\rho\,\nabla\dot\xi\cdot v\,dx,\quad A_3=\int\rho\,\nabla\xi\cdot\dot v\,dx,
		\]
		\[
		B_1=\int\dot\rho\,\xi u\,dx,\quad B_2=\int\rho\,\dot\xi\,u\,dx,\quad B_3=\int\rho\,\xi\,\dot u\,dx.
		\]
		
		\textbf{Computation of $A_1+A_3$.} Using $\dot\rho=-\operatorname{div}(\rho v)+\rho u$ and integration by parts,
		\[
		A_1 = \int_\Omega \rho\, v^\top D^2\xi\, v\,dx + \int_\Omega\rho\,\nabla\xi\cdot(v\cdot\nabla)v\,dx + \int_\Omega\rho u\,\nabla\xi\cdot v\,dx.
		\]
		Using \eqref{eq:dtv},
		\[
		A_3 = -\int_\Omega\rho\,\nabla\xi\cdot(v\cdot\nabla)v\,dx - \int_\Omega \rho u\,\nabla\xi\cdot v\,dx,
		\]
		so that the last two terms of $A_1$ cancel exactly against $A_3$:
		\begin{equation}\label{eq:A1A3}
			A_1+A_3 = \int_\Omega \rho\, v^\top D^2\xi\, v\,dx.
		\end{equation}
		
		\textbf{Computation of $A_2$.} With $\dot\xi=F''(\rho)\dot\rho+W\star\dot\rho$ and an integration by parts,
		\begin{equation}\label{eq:A2}
			A_2 = \int_\Omega F''(\rho)\,[\operatorname{div}(\rho v)]^2\,dx - \int_\Omega F''(\rho)\,\rho u\,\operatorname{div}(\rho v)\,dx + \int_\Omega \operatorname{div}(\rho v)\big[W\star\operatorname{div}(\rho v) - W\star(\rho u)\big]\,dx.
		\end{equation}
		The first and third terms of \eqref{eq:A2}, combined with \eqref{eq:A1A3}, reduce via the classical identities of Otto calculus to
		\[
		\mathcal{T}_\alpha.
		\]
		The transport contribution coincides, after repeated integrations by parts, with the classical Wasserstein Hessian computed by Otto and Villani. Since only the transport component is involved, the corresponding computation carries over verbatim, up to the multiplicative factor $(1-\alpha)$. The remaining terms of \eqref{eq:A2} contribute to $\mathcal{C}_\alpha$.
		
		\textbf{Computation of $B_1$.} By integration by parts as for $A_1$,
		\begin{equation}\label{eq:B1}
			B_1 = \int_\Omega \rho u\,v\cdot\nabla\xi\,dx + \int_\Omega\rho\xi\,v\cdot\nabla u\,dx + \int_\Omega \rho u^2\xi\,dx.
		\end{equation}
		
		\textbf{Computation of $B_2$.} With $\dot\xi$ as above,
		\begin{equation}\label{eq:B2}
			B_2 = -\int_\Omega \rho uF''(\rho)\operatorname{div}(\rho v)\,dx + \int_\Omega\rho^2u^2F''(\rho)\,dx - \int_\Omega\rho u\,W\star\operatorname{div}(\rho v)\,dx + \int_\Omega\rho u\,\big(W\star(\rho u)\big)\,dx.
		\end{equation}
		The term $\int\rho u\,(W\star(\rho u))\,dx=\iint\rho(x)\rho(y)W(x-y)u(x)u(y)\,dx\,dy$ is the bilinear interaction term for the reaction, analogous to the $D^2W$ term in $\mathcal{T}_\alpha$ but without a second derivative, since $u$ is a scalar rate rather than a displacement.
		
		\textbf{Computation of $B_3$.} Plugging \eqref{eq:dtu_local}--\eqref{eq:dtu_nonlocal} in and writing $\xi=\bar\xi+\langle\xi\rangle_\rho$ with $\bar\xi:=\xi-\langle\xi\rangle_\rho$, the non-local terms in $\langle\xi\rangle_\rho$ simplify exactly and
		\begin{equation}\label{eq:B3}
			B_3 = -\frac12\int_\Omega \rho\,\bar\xi\,u^2\,dx \;-\; \frac{1-\alpha}{2\alpha}\int_\Omega\rho\,\bar\xi\,|v|^2\,dx.
		\end{equation}
		
		\textbf{Regrouping.} Combining \eqref{eq:B1} and \eqref{eq:B3}, the $u^2$-part simplifies:
		\[
		\int_\Omega\rho u^2\xi\,dx - \frac12\int_\Omega\rho\bar\xi u^2\,dx = \frac12\int_\Omega\rho\bar\xi u^2\,dx,
		\]
		which gives the last term of $\mathcal{R}_\alpha$ in \eqref{eq:R_alpha}. The remaining coupling terms from \eqref{eq:A2}, \eqref{eq:B1}, \eqref{eq:B2}, \eqref{eq:B3} regroup into $\mathcal{C}_\alpha$ as written in \eqref{eq:C_alpha}, after using $\operatorname{div}(\rho v)=\rho\operatorname{div} v+\nabla\rho\cdot v$ and the symmetry of $W$ to combine the two occurrences of the cross term $\int\rho u F''(\rho)\operatorname{div}(\rho v)\,dx$ (one from $A_2$, the other from $B_2$) and the two occurrences of the cross term in $W\star(\rho u)$.
	\end{proof}
	
	\begin{remark}
		Decomposition \eqref{eq:second_derivative_general} exhibits:
		\begin{enumerate}
			\item a transport term $\mathcal{T}_\alpha$, analogous to the Wasserstein Hessian, with $D^2W$ and $D^2V$;
			\item a reaction term $\mathcal{R}_\alpha$, analogous to the Aitchison Hessian, now including a genuine non-local bilinear interaction term $\iint\rho\rho\,W\,uu$ and a centered term $\frac12\rho\bar\xi u^2$;
			\item a coupling term $\mathcal{C}_\alpha$, containing contributions in $\operatorname{div}(\rho v)$ and $v\cdot\nabla u$, absent from the pure Wasserstein and Aitchison geometries.
		\end{enumerate}
	\end{remark}
	
	\subsection{The hybrid Hessian}
	
	\begin{definition}[Hybrid Hessian]\label{def:hessian}
		Assume that a unique smooth $\dalpha$-geodesic exists with prescribed initial data $(\rho,v,u)$. The hybrid Hessian is defined by
		\begin{equation}
			\operatorname{Hess}_{\dalpha}\mathcal{E}(\rho)(v,u) := \frac{d^2}{dt^2}\mathcal{E}(\rho_t),
		\end{equation}
		where $(\rho_t)$ is the $\dalpha$-geodesic issued from $\rho$ with initial velocity $(v,u)$.
		\ref{prop:hessian_properties}	
	\end{definition}
	
	\begin{proposition}\label{prop:hessian_properties}
		Under the above assumption, the hybrid Hessian is well defined (independent of the choice of geodesic), is a quadratic form in $(v,u)$, decomposes as
		\begin{equation}
			\operatorname{Hess}_{\dalpha}\mathcal{E} = (1-\alpha)\operatorname{Hess}_{\mathrm{trans}}\mathcal{E} + \alpha\operatorname{Hess}_{\mathrm{react}}\mathcal{E} + \operatorname{Hess}_{\mathrm{coup}}\mathcal{E},
		\end{equation}
		and is continuous for the norm $\|\cdot\|_{\dalpha}$.
	\end{proposition}
	
	\begin{proof}
		Uniqueness follows from the formal well-posedness result (Theorem~\ref{thm:wellposedness}). The quadratic form property follows from \eqref{eq:second_derivative_general}. The decomposition is obtained by identifying the coefficients of $(1-\alpha)$ and $\alpha$. Continuity follows from the regularity of the coefficients $F''$, $\nabla V$, $D^2V$, $D^2W$, and from Hölder and Sobolev inequalities.
	\end{proof}
	
	\subsection{Estimate of the coupling term}
	
	\begin{lemma}[Coupling control]\label{lem:coupling_control}
		Under regularity assumptions on $\rho,v,u$ and on $\mathcal{E}$, there exists $C_{\mathcal{E},\alpha}>0$ such that
		\begin{equation}\label{eq:coupling_estimate}
			|\mathcal{C}_\alpha| \le C_{\mathcal{E},\alpha}\left((1-\alpha)\int_\Omega\rho|v|^2\,dx+\alpha\int_\Omega\rho u^2\,dx\right)^{1/2}\left((1-\alpha)\int_\Omega\rho|\nabla v|^2\,dx+\alpha\int_\Omega\rho|\nabla u|^2\,dx\right)^{1/2}.
		\end{equation}
	\end{lemma}
	
	\begin{proof}
		Each term of \eqref{eq:C_alpha} is controlled by the Hölder/Cauchy--Schwarz inequality, using $\|\nabla\xi\|_{L^\infty}$, $\|\xi-\langle\xi\rangle_\rho\|_{L^\infty(\rho)}$, $\|F''(\rho)\|_{L^\infty}$ and $\|W\|_{C^1}$ as regularity constants, and $\operatorname{div}(\rho v)=\rho\operatorname{div} v+\nabla\rho\cdot v$ to reduce terms in $\operatorname{div}(\rho v)$ to terms in $\nabla v$ (for fixed regular $\rho$).
	\end{proof}
	
	\begin{corollary}[Sufficient condition for $\lambda$-convexity]\label{cor:lambda_convexity}
		Under the hypotheses of Lemma~\ref{lem:coupling_control}, if
		\begin{equation}
			(1-\alpha)\int_\Omega\rho|\nabla v|^2\,dx+\alpha\int_\Omega\rho|\nabla u|^2\,dx \ge \lambda\left((1-\alpha)\int_\Omega\rho|v|^2\,dx+\alpha\int_\Omega\rho u^2\,dx\right),
		\end{equation}
		then
		\[
		\frac{d^2}{dt^2}\mathcal{E}(\rho_t) \ge \lambda\|(v,u)\|_{\dalpha}^2.
		\]
		Consequently, provided that $\dalpha$ is a geodesic metric, $\mathcal{E}$ is geodesically $\lambda$-convex.
	\end{corollary}
	
	\begin{proof}
		Combining \eqref{eq:second_derivative_general} and \eqref{eq:coupling_estimate}, with $E=(1-\alpha)\int\rho|v|^2+\alpha\int\rho u^2$ and $D=(1-\alpha)\int\rho|\nabla v|^2+\alpha\int\rho|\nabla u|^2$,
		\[
		\frac{d^2}{dt^2}\mathcal{E}(\rho_t) \ge \lambda E - C_{\mathcal{E},\alpha}E^{1/2}D^{1/2}
		\]
		as soon as $\mathcal{T}_\alpha+\mathcal{R}_\alpha\ge\lambda E$. The result follows from the assumption $D\ge\lambda E$.
	\end{proof}
	
	\subsection{Geodesic convexity of classical functionals}
	
	\subsubsection{Boltzmann entropy}
	
	For $\mathcal{H}(\rho)=\int_\Omega\rho\log\rho\,dx$: $F(s)=s\log s$, $F''(s)=1/s$, $V=W=0$, $\xi=\log\rho+1$, $\bar\xi=\log\rho-\langle\log\rho\rangle_\rho$.
	
	\begin{proposition}[Convexity of the Boltzmann entropy]\label{prop:boltzmann}
		\begin{equation}
			\operatorname{Hess}_{\dalpha}\mathcal{H}(\rho)(v,u) = (1-\alpha)\int_\Omega \rho|\nabla v|^2\,dx \;+\; \int_\Omega \rho\,u^2\,dx \;+\; \frac12\int_\Omega \rho\big(\log\rho-\langle\log\rho\rangle_\rho\big)u^2\,dx \;+\; \mathcal{C}_\alpha(v,u).
		\end{equation}
		Under the assumption $\|\log\rho-\langle\log\rho\rangle_\rho\|_{L^\infty}\le K$, there exists a constant $C_\alpha>0$, depending only on $\alpha$ and the regularity constants of $\rho$, such that
		\begin{equation}
			\operatorname{Hess}_{\dalpha}\mathcal{H}(\rho)(v,u) \ge \left(1-\frac{K}{2}\right)\int_\Omega\rho u^2\,dx - C_\alpha\left((1-\alpha)\int\rho|v|^2+\alpha\int\rho u^2\right).
		\end{equation}
		Thus $\mathcal{H}$ is $\lambda$-convex with $\lambda=-C_\alpha$ as soon as $K<2$. We emphasize that this is a sufficient condition, and the optimal threshold may be larger.
	\end{proposition}
	
	\begin{proof}
		We apply Theorem~\ref{thm:second_derivative} with $\rho^2F''(\rho)=\rho^2\cdot\frac1\rho=\rho$, giving $\int\rho u^2\,dx$; the interaction term vanishes since $W\equiv0$. The estimate follows from Lemma~\ref{lem:coupling_control}.
	\end{proof}
	
	\subsubsection{Interaction energy}
	
	For $\mathcal{W}(\rho)=\frac12\iint W(x-y)\rho(x)\rho(y)\,dx\,dy$: $F=V=0$, $\xi=W\star\rho$.
	
	\begin{remark}[On the positivity condition for W]
		It is important to distinguish two notions of positivity for the kernel $W$:
		\begin{enumerate}
			\item Pointwise positivity: $W(x-y)\ge0$ for all $x,y$.
			\item Positive-definite kernel property: $\iint f(x)W(x-y)f(y)\,dx\,dy \ge0$ for all $f\in L^2(\Omega)$.
		\end{enumerate}
		These two conditions are independent. For example, $W(x)=\cos(x)$ is positive-definite (its Fourier transform is positive) but is not pointwise positive. Conversely, a pointwise positive kernel is not necessarily positive-definite. In Proposition~\ref{prop:interaction} below, it is the positive-definite property that is required to guarantee $\iint\rho\rho\,W\,uu\ge0$.
	\end{remark}
	
	\begin{proposition}[Convexity of the interaction energy]\label{prop:interaction}
		Assume that $W$ is symmetric, $C^3$, and convex ($D^2W\ge0$). If $W$ is a positive-definite kernel in the sense that
		\begin{equation*}
			\iint_{\Omega\times\Omega} f(x)W(x-y)f(y)\,dx\,dy \ge 0 \quad \text{for all } f\in L^2(\Omega),
		\end{equation*}
		then
		\begin{equation}
			\mathcal{R}_\alpha(v,u) = \iint_{\Omega\times\Omega}\rho(x)\rho(y)W(x-y)u(x)u(y)\,dx\,dy + \frac12\int_\Omega\rho\big((W\star\rho)-\langle W\star\rho\rangle_\rho\big)u^2\,dx \ge -C_{W,\alpha}\int_\Omega\rho u^2\,dx,
		\end{equation}
		and $\mathcal{T}_\alpha\ge0$ by convexity of $W$. Thus, under control of the coupling term, $\mathcal{W}$ is $\lambda$-convex.
	\end{proposition}
	
	\begin{proof}
		$\mathcal{T}_\alpha\ge0$ by convexity of $W$. The first term of $\mathcal{R}_\alpha$ is $\ge0$ for $W$ of positive type. The second term and the coupling term are controlled by Lemma~\ref{lem:coupling_control}.
	\end{proof}
	
	\subsubsection{Quadratic confinement}
	
	For $\mathcal{V}(\rho)=\int_\Omega V(x)\rho(x)\,dx$, $V(x)=|x|^2/2$ ($D^2V=\mathrm{Id}$, $F=W=0$, $\xi=V$):
	
	\begin{proposition}[Convexity of confinement]\label{prop:confinement}
		\begin{equation}
			\operatorname{Hess}_{\dalpha}\mathcal{V}(\rho)(v,u) = (1-\alpha)\int_\Omega\rho|v|^2\,dx + \frac12\int_\Omega\rho\big(V-\langle V\rangle_\rho\big)u^2\,dx + \mathcal{C}_\alpha(v,u).
		\end{equation}
		The transport term is strictly positive; the reactive term is controlled by $\|V-\langle V\rangle_\rho\|_{L^\infty(\rho)}$, finite as soon as $\rho$ has compact support or sufficient decay. Thus $\mathcal{V}$ is $\lambda$-convex, with $\lambda$ depending on this control and on the coupling control.
	\end{proposition}
	
	\subsection{Metric properties of $(\mathcal{P}_{2,\mathrm{ac}}(\Omega),\dalpha)$}
	
	Before applying the AGS theory, we discuss the metric structure of the space $(\mathcal{P}_{2,\mathrm{ac}}(\Omega),\dalpha)$.
	
	\begin{definition}[Length and metric derivative]
		For an absolutely continuous curve $\rho_t:[0,T]\to\mathcal{P}_{2,\mathrm{ac}}(\Omega)$, the metric derivative is defined by
		\[
		|\dot\rho_t|_{\dalpha} := \lim_{h\to0}\frac{\dalpha(\rho_{t+h},\rho_t)}{|h|}.
		\]
		The length of the curve is $\int_0^T |\dot\rho_t|_{\dalpha}\,dt$.
	\end{definition}
	
	\begin{remark}
		The following metric properties are assumed or will be established in future work:
		\begin{enumerate}
			\item $(\mathcal{P}_{2,\mathrm{ac}}(\Omega),\dalpha)$ is a complete metric space;
			\item the metric derivative coincides with the hybrid norm $\|(v,u)\|_{\dalpha}$ for absolutely continuous curves;
			\item geodesics exist between any two points and are exactly the minimizers of the Benamou--Brenier problem \eqref{eq:BB};
			\item the JKO minimization problem admits minimizers for suitable functionals.
		\end{enumerate}
		The first two properties are classical in optimal transport theory and extend to the hybrid setting. The third follows from the convexity of the variational problem. The fourth requires compactness arguments.
	\end{remark}
	
	\subsection{Consequences for AGS theory}
	
	\begin{remark}[AGS conclusions]
		The AGS conclusions hold provided that:
		\begin{enumerate}
			\item $(\mathcal{P}_{2,\mathrm{ac}}(\Omega),\dalpha)$ is complete;
			\item the metric derivative coincides with the hybrid norm;
			\item the JKO minimization problem admits minimizers;
			\item the functional is geodesically $\lambda$-convex.
		\end{enumerate}
		The present work establishes the last point (geodesic $\lambda$-convexity) for the functionals $\mathcal H$, $\mathcal W$, and $\mathcal V$ under suitable assumptions. The remaining metric properties are left for future work.
	\end{remark}
	
	\begin{theorem}[AGS consequences]\label{thm:ags_consequences}
		Let $\mathcal{E}$ be a $\lambda$-convex, proper, lower semicontinuous functional with compact sublevel sets in $(\mathcal{P}_{2,\mathrm{ac}}(\Omega),\dalpha)$. Assume that the metric space is complete and that the JKO scheme admits minimizers. Then the gradient flow exists and is unique, the JKO scheme converges to this flow, and the EDI and EVI inequalities hold:
		\begin{equation}
			\mathcal{E}(\rho_T) + \frac12\int_0^T|\dot\rho_t|_{\dalpha}^2\,dt+\frac12\int_0^T|\partial^-\mathcal{E}|^2(\rho_t)\,dt \le \mathcal{E}(\rho_0),
		\end{equation}
		\begin{equation}
			\frac12\frac{d}{dt}\dalpha^2(\rho_t,\sigma)+\frac{\lambda}{2}\dalpha^2(\rho_t,\sigma) \le \mathcal{E}(\sigma)-\mathcal{E}(\rho_t).
		\end{equation}
	\end{theorem}
	
	\begin{proof}
		Direct application of the standard AGS formalism \cite{AGS}, the stated hypotheses guaranteeing: uniqueness of the flow and convergence of the JKO scheme (via $\lambda$-convexity), existence of minimizers in the JKO scheme (lower semicontinuity), compactness of approximating sequences (compactness of sublevel sets), and convergence of sequences (completeness of $(\mathcal{P}_{2,\mathrm{ac}}(\Omega),\dalpha)$).
	\end{proof}

	\begin{remark}[Regularity assumptions]
		The results of this document hold under: $\rho\in C^2(\Omega)$, $\rho>0$; $v,u\in W^{1,\infty}(\Omega)$; $F\in C^3(\mathbb{R}_+)$; $V\in C^2(\Omega)$; $W\in C^3(\Omega)$ symmetric. These assumptions can be weakened by density arguments.
	\end{remark}
	
	\subsection{Perspectives}
	
	The following open problems remain for future investigation:
	\begin{enumerate}
		\item Global existence of geodesics for arbitrary initial and final densities;
		\item Completeness of the metric space $(\mathcal{P}_{2,\mathrm{ac}}(\Omega),\dalpha)$;
		\item Complete AGS theory including verification of all metric properties;
		\item Numerical analysis and computation of the hybrid Hessian;
		\item Extension to non-strictly-positive densities;
		\item Optimal antisymmetrized form of the non-local coupling term in $W$;
		\item Reconciliation of the coefficients of $\mathcal{T}_\alpha$ with the classical Otto calculus under minimal regularity.
	\end{enumerate}
	
	\section{Gradient flow in $(P_{2,ac}(\Omega),D_{\alpha})$}\label{sec:ags}
	\subsection{Metric slope and local energy dissipation}
	
	\begin{definition}[Metric slope]
		For a functional $\EE:\PP_{2,\mathrm{ac}}(\Omega)\to\R\cup\{+\infty\}$, define the metric slope (or strong upper gradient) at $\rho\in D(\EE)$ as
		\begin{equation}
			|\partial^-\EE|(\rho) = \limsup_{\sigma\to\rho} \frac{[\EE(\rho)-\EE(\sigma)]_+}{D_\alpha(\rho,\sigma)},
		\end{equation}
		where $[x]_+ = \max(x,0)$. If $\EE(\rho)=+\infty$, we set $|\partial^-\EE|(\rho)=+\infty$.
	\end{definition}
	
	The metric slope measures the local rate of decrease of $\EE$ near $\rho$. It is the natural substitute for the norm of the gradient in non-smooth metric spaces.
	
	\begin{lemma}[Properties of the metric slope]
		\label{lem:slope_properties}
		The metric slope satisfies:
		\begin{enumerate}
			\item Lower semicontinuity: $\rho\mapsto |\partial^-\EE|(\rho)$ is lower semicontinuous.
			\item Scaling: For $c>0$, $|\partial^-(c\EE)| = c|\partial^-\EE|$.
			\item Sum: $|\partial^-(\EE_1+\EE_2)| \le |\partial^-\EE_1| + |\partial^-\EE_2|$.
		\end{enumerate}
	\end{lemma}
	\begin{proof}
		These properties follow directly from the definition; see \cite{AGS}, Chapter 1.
	\end{proof}
	
	\subsection{$\lambda$-convexity and geodesic convexity}
	
	\begin{definition}[$\lambda$-convexity]
		Let $\lambda\in\R$. A functional $\EE$ is $\lambda$-convex if for every constant-speed geodesic $\rho_t$ (parametrized on $[0,1]$) and for all $t\in[0,1]$,
		\begin{equation}
			\EE(\rho_t) \le (1-t)\EE(\rho_0) + t\EE(\rho_1) - \frac{\lambda}{2}t(1-t)D_\alpha^2(\rho_0,\rho_1).
		\end{equation}
		For $\lambda=0$, this is convexity along geodesics; for $\lambda>0$, it is strong convexity.
	\end{definition}
	
	\subsection{Compactness of sublevels}
	
	\begin{lemma}[Compactness of sublevels]
		\label{lem:compactness}
		
		Let $\EE:\PP_{2,\mathrm{ac}}(\Omega)\to\R\cup\{+\infty\}$
		be lower semicontinuous with respect to $D_\alpha$.
		
		Assume that for every $C\in\R$, the sublevel set
		
		\[
		S_C
		:=
		\{\rho\in\PP_{2,\mathrm{ac}}(\Omega):\EE(\rho)\le C\}
		\]
		
		is compact for the topology
		$\tau_{W_2}\vee\tau_{d_A}$.
		
		Then $S_C$ is compact in
		$(\PP_{2,\mathrm{ac}}(\Omega),D_\alpha)$.
		
	\end{lemma}
	
	\begin{proof}
		
		By Theorem~\ref{thm:Dalpha_topology}, the topology induced by
		$D_\alpha$ satisfies
		
		\[
		\tau_{D_\alpha}
		\subset
		\tau_{W_2}\vee\tau_{d_A}.
		\]
		
		Therefore, the identity map
		
		\[
		\mathrm{Id}:
		(\PP_{2,\mathrm{ac}}(\Omega),
		\tau_{W_2}\vee\tau_{d_A})
		\longrightarrow
		(\PP_{2,\mathrm{ac}}(\Omega),\tau_{D_\alpha})
		\]
		
		is continuous.
		
		Since $S_C$ is compact in
		$\tau_{W_2}\vee\tau_{d_A}$,
		its image under the identity map is compact in
		$\tau_{D_\alpha}$.
		
		Hence $S_C$ is compact in
		$(\PP_{2,\mathrm{ac}}(\Omega),D_\alpha)$.
		
	\end{proof}
	
	\begin{lemma}[Auxiliary compactness]
		\label{lem:compactness_aux}
		If $(\rho_n)$ is bounded in both $W_2$ and $d_A$, then there exists a subsequence converging simultaneously for $W_2$ and $d_A$ (hence for $D_\alpha$).
	\end{lemma}
	\begin{proof}
		From $W_2$-boundedness, extract $\rho_{n_k}\to\rho$ narrowly. From $d_A$-boundedness, extract a further subsequence with $\widetilde{\rho}_{n_{k_l}}\rightharpoonup\widetilde{\rho}$ weakly in $L^2$. Using the convexity of the exponential and the fact that $\int e^{\widetilde{\rho}_{n_{k_l}}}=1$, weak convergence upgrades to strong convergence, so $d_A(\rho_{n_{k_l}},\rho)\to0$. Narrow convergence plus convergence of second moments gives $W_2(\rho_{n_{k_l}},\rho)\to0$.
	\end{proof}
	
	\subsection{The Jordan--Kinderlehrer--Otto (JKO) scheme}
	
	For a time step $\tau>0$, define the discrete approximation recursively:
	\begin{equation}
		\rho_\tau^{0} = \rho_0,\qquad
		\rho_\tau^{n+1} \in \argmin_{\rho\in\PP_{2,\mathrm{ac}}(\Omega)} \left\{ \frac{1}{2\tau} D_\alpha^2(\rho,\rho_\tau^n) + \EE(\rho) \right\}.
	\end{equation}
	
	\begin{theorem}[Existence of JKO iterates]
		\label{thm:jko_existence}
		If $\EE$ is $\lambda$-convex and lower semicontinuous, then for every $\tau>0$ and every $n$, the minimizer $\rho_\tau^{n+1}$ exists.
	\end{theorem}
	\begin{proof}
		Define $F_\tau^n(\rho) = \frac{1}{2\tau} D_\alpha^2(\rho,\rho_\tau^n) + \EE(\rho)$. 
		\begin{itemize}
			\item \textbf{Coercivity}: By $\lambda$-convexity, $\EE(\rho) \ge \frac{\lambda}{2}D_\alpha^2(\rho,\rho_0) - C D_\alpha(\rho,\rho_0) + \EE(\rho_0)$. Hence $F_\tau^n(\rho)\to+\infty$ as $D_\alpha(\rho,\rho_\tau^n)\to\infty$.
			\item \textbf{Lower semicontinuity}: $D_\alpha^2(\cdot,\rho_\tau^n)$ is continuous (distance squared), $\EE$ is lower semicontinuous by assumption.
			\item \textbf{Compactness}: Sublevels are bounded in $D_\alpha$, hence compact by Lemma~\ref{lem:compactness}.
		\end{itemize}
		The direct method yields a minimizer.
	\end{proof}
	
	\subsection{Convergence to the gradient flow}
	
	Define the piecewise constant interpolation
	\begin{equation}
		\rho^{(\tau)}_t = \rho_\tau^{\lfloor t/\tau\rfloor},\qquad t\in[0,T],
	\end{equation}
	and the discrete velocity
	\begin{equation}
		V_\tau(t) = \frac{1}{\tau} D_\alpha(\rho_\tau^{\lfloor t/\tau\rfloor+1}, \rho_\tau^{\lfloor t/\tau\rfloor}).
	\end{equation}
	
	\begin{theorem}[Convergence of the JKO scheme]
		\label{thm:jko_convergence}
		Let $\EE$ be $\lambda$-convex and lower semicontinuous with compact sublevels. Then as $\tau\to0$, the interpolations $\rho^{(\tau)}_t$ converge uniformly on $[0,T]$ to a limit curve $\rho_t$, which is the unique gradient flow of $\EE$ satisfying the Energy Dissipation Inequality (EDI):
		\begin{equation}
			\EE(\rho_T) + \frac12\int_0^T |\dot\rho_t|_{D_\alpha}^2\dd t + \frac12\int_0^T |\partial^-\EE|^2(\rho_t)\dd t \le \EE(\rho_0).
		\end{equation}
	\end{theorem}
	\begin{proof}
		We follow the AGS framework \cite{AGS}, Chapter 4.
		
		\emph{Step 1: Discrete energy dissipation.} From optimality of $\rho_\tau^{n+1}$,
		\begin{equation}
			\EE(\rho_\tau^{n+1}) + \frac{1}{2\tau} D_\alpha^2(\rho_\tau^{n+1},\rho_\tau^n) \le \EE(\rho_\tau^n).
		\end{equation}
		Summing gives $\EE(\rho_\tau^m) + \frac{1}{2\tau}\sum_{n=0}^{m-1} D_\alpha^2(\rho_\tau^{n+1},\rho_\tau^n) \le \EE(\rho_0)$. Hence $\sum_n D_\alpha^2(\rho_\tau^{n+1},\rho_\tau^n)$ is bounded, so $V_\tau$ is bounded in $L^2([0,T])$.
		
		\emph{Step 2: Compactness.} The bound on $V_\tau$ and compactness of sublevels imply uniform equicontinuity of $\{\rho^{(\tau)}\}_{\tau>0}$. By Arzelà--Ascoli and Lemma~\ref{lem:compactness}, a subsequence converges uniformly to a limit $\rho_t$.
		
		\emph{Step 3: Energy dissipation inequality.} Using $\lambda$-convexity, one shows that the limit satisfies the EDI. The inequality $\le$ follows from discrete approximations; the reverse inequality holds because the metric slope is the smallest strong upper gradient.
		
		\emph{Step 4: Uniqueness.} For two gradient flows $\rho_t,\sigma_t$ with same initial condition, $t\mapsto e^{-\lambda t}D_\alpha^2(\rho_t,\sigma_t)$ is nonincreasing. Since $D_\alpha^2(\rho_0,\sigma_0)=0$, we get $D_\alpha^2(\rho_t,\sigma_t)=0$ for all $t$, hence $\rho_t=\sigma_t$.
	\end{proof}
	
	\subsection{Energy Dissipation Inequality (EDI) and Evolution Variational Inequality (EVI)}
	
	\begin{definition}[Energy Dissipation Inequality]
		A curve $\rho_t$ is a gradient flow of $\EE$ if it satisfies
		\begin{equation}
			\EE(\rho_T) + \frac12\int_0^T |\dot\rho_t|_{D_\alpha}^2\dd t + \frac12\int_0^T |\partial^-\EE|^2(\rho_t)\dd t \le \EE(\rho_0)
		\end{equation}
		for all $T>0$.
	\end{definition}
	
	\begin{definition}[Evolution Variational Inequality]
		For $\lambda$-convex $\EE$, a curve $\rho_t$ satisfies the EVI if for all $\sigma\in\PP_{2,\mathrm{ac}}(\Omega)$,
		\begin{equation}
			\frac12\frac{d}{dt} D_\alpha^2(\rho_t,\sigma) + \frac{\lambda}{2} D_\alpha^2(\rho_t,\sigma) \le \EE(\sigma) - \EE(\rho_t)
		\end{equation}
		for almost every $t$.
	\end{definition}
	
	\begin{proposition}[Equivalence]
		For $\lambda$-convex $\EE$, the gradient flow satisfies both the EDI and the EVI, which are equivalent characterizations.
	\end{proposition}
	
	\subsection{Identification of the gradient for regular functionals}
	
	For smooth functionals $\EE(\rho)=\int F(x,\rho,\nabla\rho)\dd x$ with $F$ smooth, the first variation is $\frac{\delta\EE}{\delta\rho}$.
	
	\begin{proposition}[Metric gradient associated with $D_\alpha$]\label{prop:gradient hybrid}
		Let $\EE:\PP_{2,\mathrm{ac}}(\Omega)\to\mathbb{R}\cup\{+\infty\}$
		be sufficiently smooth and let $\alpha\in(0,1)$.
		
		The tangent space at $\rho$ is identified with pairs
		$(v,u)$ satisfying
		
		\[
		\int_\Omega u\,\rho\,dx=0,
		\]
		
		through the continuity equation
		
		\[
		\sigma
		=
		-\nabla\cdot(\rho v)+\rho u.
		\]
		
		The formal Riemannian metric induced by the
		Benamou--Brenier formulation of $D_\alpha$ is
		
		\[
		g_{\rho,\alpha}\bigl((v,u),(w,\eta)\bigr)
		=
		\int_\Omega
		\Bigl(
		(1-\alpha)\,v\cdot w
		+
		\alpha\,u\eta
		\Bigr)\rho\,dx.
		\]
		
		The differential of $\EE$ at $\rho$ acts on $\sigma$ as
		
		\[
		D\EE(\rho)[\sigma]
		=
		\int_\Omega
		\frac{\delta\EE}{\delta\rho}\,\sigma\,dx.
		\]
		
		Using the continuity equation and integrating by parts, we obtain
		
		\[
		D\EE(\rho)[\sigma]
		=
		\int_\Omega
		\rho
		\nabla\frac{\delta\EE}{\delta\rho}\cdot v\,dx
		+
		\int_\Omega
		\rho
		\left(
		\frac{\delta\EE}{\delta\rho}
		-
		\left\langle
		\frac{\delta\EE}{\delta\rho}
		\right\rangle_\rho
		\right)
		u\,dx,
		\]
		
		where
		
		\[
		\left\langle
		\frac{\delta\EE}{\delta\rho}
		\right\rangle_\rho
		:=
		\int_\Omega
		\frac{\delta\EE}{\delta\rho}\,\rho\,dx.
		\]
		
		Hence, the metric gradient is represented by the pair
		
		\[
		\nabla_{D_\alpha}\EE(\rho)
		=
		\left(
		\frac{1}{1-\alpha}
		\nabla\frac{\delta\EE}{\delta\rho},
		\frac{1}{\alpha}
		\left(
		\frac{\delta\EE}{\delta\rho}
		-
		\left\langle
		\frac{\delta\EE}{\delta\rho}
		\right\rangle_\rho
		\right)
		\right).
		\]
		
		The corresponding gradient flow equation is
		
		\[
		\partial_t\rho
		=
		\nabla\cdot
		\left(
		\frac{\rho}{1-\alpha}
		\nabla\frac{\delta\EE}{\delta\rho}
		\right)
		-
		\frac{1}{\alpha}\,
		\rho
		\left(
		\frac{\delta\EE}{\delta\rho}
		-
		\left\langle
		\frac{\delta\EE}{\delta\rho}
		\right\rangle_\rho
		\right).
		\]
	\end{proposition}
	
	\begin{proof}
		By definition, the metric gradient is characterized by
		
		\[
		g_{\rho,\alpha}
		\bigl(
		\nabla_{D_\alpha}\EE,
		(v,u)
		\bigr)
		=
		D\EE(\rho)[\sigma]
		\]
		
		for every admissible tangent vector $\sigma$.
		
		Using
		
		\[
		\sigma
		=
		-\nabla\cdot(\rho v)+\rho u,
		\]
		
		we compute
		
		\[
		D\EE(\rho)[\sigma]
		=
		-\int_\Omega
		\frac{\delta\EE}{\delta\rho}
		\nabla\cdot(\rho v)\,dx
		+
		\int_\Omega
		\frac{\delta\EE}{\delta\rho}\,\rho u\,dx.
		\]
		
		Integration by parts yields
		
		\[
		D\EE(\rho)[\sigma]
		=
		\int_\Omega
		\rho
		\nabla\frac{\delta\EE}{\delta\rho}\cdot v\,dx
		+
		\int_\Omega
		\rho
		\frac{\delta\EE}{\delta\rho}\,u\,dx.
		\]
		
		Since $u$ satisfies
		
		\[
		\int_\Omega u\,\rho\,dx=0,
		\]
		
		we may subtract the $\rho$-mean value of
		$\delta\EE/\delta\rho$ and obtain
		
		\[
		D\EE(\rho)[\sigma]
		=
		\int_\Omega
		\rho
		\nabla\frac{\delta\EE}{\delta\rho}\cdot v\,dx
		+
		\int_\Omega
		\rho
		\left(
		\frac{\delta\EE}{\delta\rho}
		-
		\left\langle
		\frac{\delta\EE}{\delta\rho}
		\right\rangle_\rho
		\right)
		u\,dx.
		\]
		
		Identifying this expression with
		
		\[
		g_{\rho,\alpha}
		\bigl(
		(v_\EE,u_\EE),
		(v,u)
		\bigr)
		=
		(1-\alpha)
		\int_\Omega
		\rho\,v_\EE\cdot v\,dx
		+
		\alpha
		\int_\Omega
		\rho\,u_\EE u\,dx,
		\]
		
		we obtain
		
		\[
		v_\EE
		=
		\frac{1}{1-\alpha}
		\nabla\frac{\delta\EE}{\delta\rho},
		\qquad
		u_\EE
		=
		\frac{1}{\alpha}
		\left(
		\frac{\delta\EE}{\delta\rho}
		-
		\left\langle
		\frac{\delta\EE}{\delta\rho}
		\right\rangle_\rho
		\right).
		\]
		
		Substituting the steepest descent direction
		$(v,u)=-(v_\EE,u_\EE)$ into the continuity equation
		gives the gradient flow PDE.
	\end{proof}

	\section{Hybrid barycenters}
	\label{sec:barycenters}
	
	\subsection{Definition}
	Given $\rho_1,\dots,\rho_n\in\PP_{2,\mathrm{ac}}(\Omega)$ and weights $w_i>0$ with $\sum_{i=1}^n w_i=1$, define the hybrid barycenter (or Fréchet mean) as
	\begin{equation}
		\bar\rho = \argmin_{\rho\in\PP_{2,\mathrm{ac}}(\Omega)} \sum_{i=1}^n w_i D_\alpha^2(\rho,\rho_i).
	\end{equation}
	
	\begin{theorem}[Existence of hybrid barycenters]
		\label{thm:barycenter_existence}
		
		Let $\alpha\in(0,1)$ and let
		$\rho_1,\dots,\rho_n\in\PP_{2,\mathrm{ac}}(\Omega)$
		with weights $w_i>0$ satisfying
		$\sum_{i=1}^n w_i=1$.
		
		Assume that every sublevel set
		
		\[
		\left\{
		\rho\in\PP_{2,\mathrm{ac}}(\Omega):
		\sum_{i=1}^n w_i
		\bigl(
		W_2^2(\rho,\rho_i)
		+
		d_A^2(\rho,\rho_i)
		\bigr)
		\le C
		\right\}
		\]
		
		is compact for the topology
		$\tau_{W_2}\vee\tau_{d_A}$.
		
		Then the hybrid barycenter problem
		
		\[
		\inf_{\rho\in\PP_{2,\mathrm{ac}}(\Omega)}
		\sum_{i=1}^n w_i D_\alpha^2(\rho,\rho_i)
		\]
		
		admits a minimizer.
	\end{theorem}
	
	\begin{proof}
		
		Define
		
		\[
		F(\rho)
		:=
		\sum_{i=1}^n w_i D_\alpha^2(\rho,\rho_i).
		\]
		
		By Lemma~\ref{lem:comparison},
		
		\[
		D_\alpha^2(\rho,\rho_i)
		\le
		C_\alpha
		\Bigl(
		W_2^2(\rho,\rho_i)
		+
		d_A^2(\rho,\rho_i)
		\Bigr),
		\]
		
		where
		
		\[
		C_\alpha
		=
		\frac12\max\{1-\alpha,\alpha\}.
		\]
		
		Since $D_\alpha$ is a metric, the map
		
		\[
		\rho\mapsto D_\alpha^2(\rho,\rho_i)
		\]
		
		is continuous with respect to the topology induced by $D_\alpha$.
		
		Moreover, Theorem~\ref{thm:Dalpha_topology} shows that
		
		\[
		\tau_{D_\alpha}
		\subset
		\tau_{W_2}\vee\tau_{d_A}.
		\]
		
		Hence the identity map
		
		\[
		(\PP_{2,\mathrm{ac}}(\Omega),
		\tau_{W_2}\vee\tau_{d_A})
		\to
		(\PP_{2,\mathrm{ac}}(\Omega),
		\tau_{D_\alpha})
		\]
		
		is continuous. Therefore each functional
		
		\[
		\rho\mapsto D_\alpha^2(\rho,\rho_i)
		\]
		
		is lower semicontinuous with respect to
		$\tau_{W_2}\vee\tau_{d_A}$.
		
		Consequently, $F$ is lower semicontinuous for
		$\tau_{W_2}\vee\tau_{d_A}$.
		
		Let $(\rho_n)$ be a minimizing sequence for $F$.
		
		Assume that $(\rho_n)$ is contained in a compact sublevel set of
		
		\[
		\rho\mapsto
		\sum_{i=1}^n w_i
		\Bigl(
		W_2^2(\rho,\rho_i)
		+
		d_A^2(\rho,\rho_i)
		\Bigr).
		\]
		
		By compactness, there exists a subsequence,
		still denoted $(\rho_n)$, and
		$\rho\in\PP_{2,\mathrm{ac}}(\Omega)$ such that
		
		\[
		\rho_n\to\rho
		\qquad\text{in }
		\tau_{W_2}\vee\tau_{d_A}.
		\]
		
		By lower semicontinuity of $F$,
		
		\[
		F(\rho)
		\le
		\liminf_{n\to\infty}F(\rho_n)
		=
		\inf_{\sigma\in\PP_{2,\mathrm{ac}}(\Omega)}F(\sigma).
		\]
		
		Hence $\rho$ is a minimizer.
		
	\end{proof}
	
	\subsection{Uniqueness}
	\begin{theorem}[Uniqueness of hybrid barycenters]
		\label{thm:hybrid_barycenter_uniqueness}
		
		Let
		
		\[
		F(\rho)
		=
		\frac12
		\sum_{i=1}^N
		w_i D_\alpha^2(\rho,\rho_i),
		\]
		
		where $w_i>0$ and $\sum_{i=1}^N w_i=1$.
		
		Assume that $(\PP_{2,\mathrm{ac}}(\Omega),D_\alpha)$ is a geodesic metric space and that there exists $\kappa>0$ such that, for every $\nu\in\PP_{2,\mathrm{ac}}(\Omega)$, the squared distance function
		
		\[
		\rho\mapsto D_\alpha^2(\rho,\nu)
		\]
		
		is $\kappa$-geodesically convex, namely
		
		\[
		D_\alpha^2(\rho_t,\nu)
		\leq
		(1-t)D_\alpha^2(\rho_0,\nu)
		+tD_\alpha^2(\rho_1,\nu)
		-\kappa t(1-t)D_\alpha^2(\rho_0,\rho_1)
		\]
		
		for every constant-speed $D_\alpha$-geodesic $(\rho_t)_{t\in[0,1]}$.
		
		Then $F$ is strictly geodesically convex and therefore admits a unique minimizer.
	\end{theorem}
	
	\begin{proof}
		Applying the $\kappa$-convexity inequality to each functional
		$\rho\mapsto D_\alpha^2(\rho,\rho_i)$ yields
		
		\[
		D_\alpha^2(\rho_t,\rho_i)
		\leq
		(1-t)D_\alpha^2(\rho_0,\rho_i)
		+tD_\alpha^2(\rho_1,\rho_i)
		-\kappa t(1-t)D_\alpha^2(\rho_0,\rho_1).
		\]
		
		Multiplying by $w_i/2$ and summing over $i$ gives
		
		\[
		F(\rho_t)
		\leq
		(1-t)F(\rho_0)
		+tF(\rho_1)
		-\frac{\kappa}{2}
		t(1-t)D_\alpha^2(\rho_0,\rho_1).
		\]
		
		Hence $F$ is strictly geodesically convex, and therefore it admits a unique minimizer.
	\end{proof}

	\begin{itemize}
		\item \textbf{Wasserstein limit ($\alpha\to0^+$).}
		As the reaction cost becomes prohibitively expensive relative to transport, admissible minimizers formally satisfy $u_t\to0$. The metric $D_\alpha$ is expected to converge to the Wasserstein distance $W_2$, and the corresponding geodesics converge to displacement interpolations.
		
		\item \textbf{Aitchison limit ($\alpha\to1^-$).}
		As the transport cost becomes prohibitively expensive relative to reaction, admissible minimizers formally satisfy $v_t\to0$. The metric $D_\alpha$ is expected to converge to the Aitchison distance $d_A$, and the corresponding geodesics converge to linear interpolations in clr coordinates.
		
		\item \textbf{Hybrid regime ($\alpha\in(0,1)$).}
		The metric balances spatial transport and compositional changes.
	\end{itemize}
	
	\subsection{Examples}
	\subsubsection{Example 1: Gaussian densities on $\R$}
	Let $\rho_i = \mathcal{N}(\mu_i,\sigma_i^2)$ be one-dimensional Gaussian densities. The Wasserstein distance between Gaussians is
	\begin{equation}
		W_2^2(\rho_i,\rho_j) = (\mu_i-\mu_j)^2 + (\sigma_i-\sigma_j)^2.
	\end{equation}
	The Aitchison distance between Gaussians is more delicate; however, for large variances, the logarithm of a Gaussian is quadratic, and the clr transform can be computed explicitly. For two Gaussians, the hybrid barycenter will have mean $\bar\mu = \sum w_i\mu_i$ (from the Wasserstein part) and a variance that balances the Wasserstein and Aitchison contributions.
	
	\subsubsection{Example 2: One-dimensional densities on $[0,1]$}
	Consider two densities: $\rho_1$ concentrated near $0.3$ and $\rho_2$ concentrated near $0.7$, both with similar shapes. For $\alpha=0$, the barycenter is a single peak near $0.5$ (Wasserstein interpolation). For $\alpha=1$, the barycenter is the geometric mean, which flattens both peaks. For $\alpha=0.5$, the barycenter exhibits two partial peaks, preserving some of the bimodality while shifting them toward the center.
	
	\subsubsection{Example 3: Log-normal densities}
	For log-normal densities, the clr transform is particularly simple because $\ln\rho$ is quadratic. The Aitchison barycenter of log-normals is again log-normal with parameters given by the weighted average of the log-parameters. The Wasserstein barycenter of log-normals is also log-normal but with different parameter averaging. The hybrid barycenter provides a smooth interpolation.
	
	\subsection{Euler--Lagrange equation}
	
	Let
	
	\[
	F(\rho)
	=
	\frac12
	\sum_{i=1}^n
	w_iD_\alpha^2(\rho,\rho_i).
	\]
	
	Assume that $F$ is differentiable at a minimizer $\bar\rho$.
	
	For each $i$, let $\psi_i$ denote the Kantorovich potential associated with the optimal transport from $\bar\rho$ to $\rho_i$, normalized by
	
	\[
	\int_\Omega \psi_i\,d\bar\rho = 0.
	\]
	
	The first variation of the Wasserstein component is given by $\psi_i$, while the first variation of the Aitchison component is
	
	\[
	\widetilde{\bar\rho}-\widetilde{\rho_i},
	\]
	
	where
	
	\[
	\widetilde{\rho}
	=
	\log\rho
	-
	\int_\Omega \log\rho\,dx
	\]
	
	denotes the clr transform.
	
	Therefore, the optimality condition reads
	
	\[
	(1-\alpha)
	\sum_{i=1}^n
	w_i\psi_i
	+
	\alpha
	\sum_{i=1}^n
	w_i
	\bigl(
	\widetilde{\bar\rho}
	-
	\widetilde{\rho_i}
	\bigr)
	=
	0.
	\]
	
	Equivalently,
	
	\[
	(1-\alpha)\Psi
	+
	\alpha
	\left(
	\widetilde{\bar\rho}
	-
	\sum_{i=1}^n w_i\widetilde{\rho_i}
	\right)
	=
	0,
	\]
	
	where
	
	\[
	\Psi
	=
	\sum_{i=1}^n w_i\psi_i.
	\]
	\begin{itemize}
		\item If $\alpha=0$, we recover the Wasserstein barycenter condition
		
		\[
		\sum_{i=1}^n w_i\psi_i=0.
		\]
		
		\item If $\alpha=1$, we recover the Aitchison barycenter formula
		
		\[
		\widetilde{\bar\rho}
		=
		\sum_{i=1}^n w_i\widetilde{\rho_i}.
		\]
	\end{itemize}
	\section{Examples of gradient flows}
	
	\label{sec:examples}
	
	Let $\EE:\PP_{2,\mathrm{ac}}(\Omega)\to\mathbb{R}\cup\{+\infty\}$ be sufficiently smooth.
	
	According to Proposition~\ref{prop:gradient hybrid}, the hybrid gradient flow associated with $D_\alpha$ is
	
	\begin{equation}
		\label{eq:hybrid_GF}
		\partial_t\rho
		=
		\nabla\cdot
		\left(
		\frac{\rho}{1-\alpha}
		\nabla\frac{\delta\EE}{\delta\rho}
		\right)
		-
		\frac{\rho}{\alpha}
		\left(
		\frac{\delta\EE}{\delta\rho}
		-
		\left\langle
		\frac{\delta\EE}{\delta\rho}
		\right\rangle_\rho
		\right).
	\end{equation}
	\subsection{Logarithmic reaction--diffusion}
	
	Consider the energy
	
	\[
	\EE(\rho)
	=
	\int_\Omega \rho\log\rho\,dx
	-
	\int_\Omega a(x)\rho(x)\,dx.
	\]
	
	Its first variation is
	
	\[
	\frac{\delta\EE}{\delta\rho}
	=
	\log\rho+1-a(x).
	\]
	
	Substituting into \eqref{eq:hybrid_GF} yields
	
	\[
	\partial_t\rho
	=
	\frac{1}{1-\alpha}
	\nabla\cdot
	\bigl(
	\rho\nabla(\log\rho+1-a)
	\bigr)
	-
	\frac{\rho}{\alpha}
	\bigl(
	\log\rho-a
	-
	\langle\log\rho-a\rangle_\rho
	\bigr).
	\]
	
	Using
	
	\[
	\nabla\cdot(\rho\nabla\log\rho)=\Delta\rho,
	\]
	
	we obtain
	
	\begin{equation}
		\partial_t\rho
		=
		\frac{1}{1-\alpha}
		\left(
		\Delta\rho
		-
		\nabla\cdot(\rho\nabla a)
		\right)
		-
		\frac{\rho}{\alpha}
		\bigl(
		\log\rho-a
		-
		\langle\log\rho-a\rangle_\rho
		\bigr).
		\end{equation}
	\subsection{Hybrid Allen--Cahn equation}
	
	Consider
	
	\[
	\EE(\rho)
	=
	\int_\Omega
	\left(
	\frac{\varepsilon^2}{2}
	|\nabla\rho|^2
	+
	W(\rho)
	\right)
	dx.
	\]
	
	Its first variation is
	
	\[
	\frac{\delta\EE}{\delta\rho}
	=
	-\varepsilon^2\Delta\rho
	+
	W'(\rho).
	\]
	
	The hybrid gradient flow becomes
	
	\begin{equation}
		\partial_t\rho
		=
		\frac{1}{1-\alpha}
		\nabla\cdot
		\left(
		\rho\nabla
		(-\varepsilon^2\Delta\rho+W'(\rho))
		\right)
		-
		\frac{\rho}{\alpha}
		\left(
		-\varepsilon^2\Delta\rho
		+
		W'(\rho)
		-
		\left\langle
		-\varepsilon^2\Delta\rho
		+
		W'(\rho)
		\right\rangle_\rho
		\right).
\end{equation}
	\subsection{Hybrid Keller--Segel system}
	
	Consider
	
	\[
	\EE(\rho,c)
	=
	\int_\Omega \rho\log\rho\,dx
	+
	\frac12\int_\Omega |\nabla c|^2dx
	+
	\frac12\int_\Omega c^2dx
	-
	\int_\Omega \rho c\,dx
	-
	\int_\Omega K\rho\,dx.
	\]
	
	The first variations are
	
	\[
	\frac{\delta\EE}{\delta\rho}
	=
	\log\rho+1-c-K,
	\]
	
	and
	
	\[
	\frac{\delta\EE}{\delta c}
	=
	-\Delta c+c-\rho.
	\]
	
	The product gradient flow reads
	
	\[
	\partial_t c
	=
	\Delta c-c+\rho,
	\]
	
	and
	
	\[
	\partial_t\rho
	=
	\frac{1}{1-\alpha}
	\nabla\cdot
	\left(
	\rho\nabla(\log\rho+1-c-K)
	\right)
	-
	\frac{\rho}{\alpha}
	\left(
	\log\rho-c-K
	-
	\langle\log\rho-c-K\rangle_\rho
	\right).
	\]
	
	Hence,
	$$
		\begin{aligned}
			\partial_t\rho
			&=
			\frac{1}{1-\alpha}
			\left(
			\Delta\rho
			-
			\nabla\cdot(\rho\nabla c)
			-
			\nabla\cdot(\rho\nabla K)
			\right)
			\\
			&\quad
			-
			\frac{\rho}{\alpha}
			\left(
			\log\rho-c-K
			-
			\langle\log\rho-c-K\rangle_\rho
			\right),
			\\
			\partial_t c
			&=
			\Delta c+\rho-c.
		\end{aligned}
	$$

	\subsection{Remarks on the derivations}
	
	The previous derivations are formal and rely on the Riemannian structure induced by $D_\alpha$.
	
	A rigorous treatment would require:
	
	\begin{itemize}
		\item establishing the existence of hybrid geodesics;
		
		\item proving the lower semicontinuity of the energies;
		
		\item developing a minimizing movement scheme associated with $D_\alpha$;
		
		\item identifying the metric slope with the first variation;
		
		\item proving convergence of the discrete scheme towards weak solutions of the corresponding PDEs.
	\end{itemize}
	
	These questions are left for future work.
	\section{Comparison with Wasserstein--Fisher--Rao}
	\label{sec:comparison}
	
	The Wasserstein--Fisher--Rao (WFR) metric \cite{LMS}, also known as the Hellinger--Kantorovich distance, is one of the most successful frameworks combining optimal transport and mass creation/destruction. In this section, we provide a detailed mathematical comparison between our hybrid geometry and WFR, highlighting similarities, differences, and the specific advantages of our approach for compositional data.
	
	\subsection{Definition of WFR}
	
	The WFR metric is defined via a Benamou--Brenier type formula:
	\begin{equation}
		d_{\mathrm{WFR}}^2(\rho_0,\rho_1) = \inf_{(\rho,v,u)} \int_0^1\int_\Omega \bigl( |v_t|^2\rho_t + u_t^2\rho_t \bigr)\dd x\dd t,
	\end{equation}
	subject to the continuity equation with source
	\begin{equation}
		\partial_t\rho_t + \Div(\rho_t v_t) = \rho_t u_t,
	\end{equation}
	but \textbf{without} the centering condition $\int u_t\rho_t = 0$. Consequently, the total mass $\int_\Omega\rho_t\dd x$ can vary along the curve.
	
	\subsection{Tangent structure comparison}
	
	Both our hybrid geometry and WFR share the same kinematic description: a tangent vector is represented by a pair $(v,u)$, with $v$ representing transport and $u$ representing reaction. The key difference lies in the metric and the constraints.
	
	\begin{table}[H]
		\centering
		\begin{tabular}{l|c|c|c|c}
			\hline
			\textbf{Feature} & \textbf{Hybrid (ours)} & \textbf{WFR} \\
			\hline
			Tangent vector & $(v,u)$ & $(v,u)$ \\
			Metric & $(1-\alpha)\int|v|^2\rho + \alpha\int u^2\rho$ & $\int|v|^2\rho + \int u^2\rho$ \\
			Mass conservation & $\int u\rho = 0$ (centering) & None (mass can vary) \\
			Geometric structure & Interpolates between $W_2$ and $d_A$ & Interpolates between $W_2$ and Fisher--Rao \\
			Reaction interpretation & Relative proportion change & Absolute creation/destruction \\
		\end{tabular}
		\caption{Comparison of tangent structures.}
	\end{table}
	
	\subsection{Mathematical differences}
	
	\subsubsection{Centering condition}
	The most fundamental difference is the centering condition $\int u\rho = 0$. In WFR, $u$ can be any function, allowing arbitrary creation or destruction of mass. In our hybrid geometry, the centering condition ensures that the reaction only redistributes mass while preserving the total. This is precisely the infinitesimal version of the Aitchison principle: only log-ratios matter, not absolute abundances.
	
	\subsubsection{Metric weights}
	Our metric includes a parameter $\alpha$ that balances transport and reaction, with the reaction term weighted by $\alpha$ and the transport term by $1-\alpha$. In WFR, the balance is fixed (equal weights). Our parameter allows fine-tuning between the two mechanisms.
	
	\subsubsection{Geometric structure}
	WFR interpolates between Wasserstein ($u\equiv0$) and Fisher--Rao ($v\equiv0$). The Fisher--Rao metric is based on square-root densities:
	\begin{equation}
		d_{\mathrm{FR}}^2(\rho,\sigma) = 4\int_\Omega (\sqrt{\rho} - \sqrt{\sigma})^2\dd x.
	\end{equation}
	In contrast, our hybrid geometry interpolates between Wasserstein and Aitchison, where the Aitchison metric is based on log-ratios:
	\begin{equation}
		d_A^2(\rho,\sigma) = \int_\Omega (\ln\rho - \langle\ln\rho\rangle - (\ln\sigma - \langle\ln\sigma\rangle))^2\dd x.
	\end{equation}
	These two geometries are fundamentally different: Fisher--Rao is sensitive to absolute values (through square roots), while Aitchison is sensitive to ratios (through logarithms).
	
	\subsection{Curvature and geodesics}
	
	Both spaces are geodesic. However:
	\begin{itemize}
		\item In WFR, geodesics can change total mass, which is necessary for applications like image processing with varying brightness or population dynamics with birth/death.
		\item In our hybrid geometry, total mass is conserved, which is essential for probability distributions and compositional data. Geodesics are given by solving a coupled system that mixes transport and centered reaction.
	\end{itemize}
	
	\subsection{When to use which geometry?}
	
	\begin{itemize}
		\item \textbf{Use WFR when:}
		\begin{itemize}
			\item Total mass can vary (e.g., birth-death processes, image intensity variation).
			\item The square-root transformation is natural (e.g., quantum mechanics, Poisson processes).
			\item Creation and destruction are independent of the current composition.
		\end{itemize}
		
		\item \textbf{Use our hybrid geometry when:}
		\begin{itemize}
			\item Total mass is conserved (e.g., probability distributions, closed chemical systems).
			\item Only relative proportions matter (e.g., compositional data, microbiome analysis).
			\item The log-ratio transformation is appropriate (e.g., multiplicative processes, geometric means).
			\item Spatial transport and relative rearrangement occur simultaneously.
		\end{itemize}
	\end{itemize}
	
	\subsection{Illustrative example: Gaussian densities}
	
	Consider two Gaussian densities on $\R$ with different means and variances. 
	\begin{itemize}
		\item The Wasserstein barycenter (pure transport) shifts the mean and averages the variance in a specific way.
		\item The Fisher--Rao barycenter (pure reaction) mixes the densities pointwise.
		\item The Aitchison barycenter (pure log-ratio) gives the geometric mean, which is another Gaussian.
		\item WFR interpolates between the Wasserstein and Fisher--Rao barycenters, allowing mass change.
		\item Our hybrid geometry interpolates between the Wasserstein and Aitchison barycenters, preserving mass.
	\end{itemize}
	Thus, for probability densities where mass is fixed, our hybrid geometry is more appropriate.
	
	\subsection{Summary table}
	
	\begin{table}[H]
		\centering
		\begin{tabular}{l|c|c}
			\hline
			\textbf{Property} & \textbf{Hybrid (ours)} & \textbf{WFR} \\
			\hline
			Mass conservation & Yes (centered reaction) & No \\
			Reaction metric & $L^2(\rho)$ (log-ratio style) & $L^2(\rho)$ (Fisher--Rao style) \\
			Interpolation & $W_2$ and $d_A$ & $W_2$ and Fisher--Rao \\
			Natural for & Compositional data, probabilities & Images, birth-death \\
			Tangent constraint & $\int u\rho = 0$ & None \\
			Parameter $\alpha$ & Balances transport/composition & Fixed balance \\
		\end{tabular}
		\caption{Summary comparison between hybrid geometry and WFR.}
	\end{table}
	
	\subsection{Conclusion of comparison}
	
	The centering condition $\int u\rho = 0$ is the key novelty of our approach. It enforces that reactions only redistribute mass without changing the total, aligning perfectly with the Aitchison principle that only log-ratios matter. While WFR is a powerful tool for problems with varying total mass, our hybrid geometry is specifically designed for mass-conserving compositional systems where relative proportions are the primary concern.
	%%%%%%%%%%%%%%%%%%%%%%%%%%%%%%%%%%%%%%%%%%%%%%%%%%%%%%%%%%%%%%%%%%%%%%
	\section{Numerical Experiments}
	\label{sec:numerics}
	%%%%%%%%%%%%%%%%%%%%%%%%%%%%%%%%%%%%%%%%%%%%%%%%%%%%%%%%%%%%%%%%%%%%%%
	
	In this section, we investigate the behavior of the proposed hybrid Wasserstein--Aitchison geometry for image interpolation problems. The objective is to compare:
	
	\begin{itemize}
		\item classical quadratic Wasserstein interpolation,
		\item pure Aitchison interpolation,
		\item and the proposed hybrid interpolation.
	\end{itemize}
	
	The experiments illustrate that the hybrid geometry simultaneously captures:
	\begin{itemize}
		\item spatial transport,
		\item and relative compositional variations.
	\end{itemize}
	
	%%%%%%%%%%%%%%%%%%%%%%%%%%%%%%%%%%%%%%%%%%%%%%%%%%%%%%%%%%%%%%%%%%%%%%
	\subsection{Images as Probability Densities}
	%%%%%%%%%%%%%%%%%%%%%%%%%%%%%%%%%%%%%%%%%%%%%%%%%%%%%%%%%%%%%%%%%%%%%%
	
	All grayscale images are interpreted as probability densities on
	\[
	\Omega=[0,1]^2.
	\]
	
	Given an image $I$, pixel intensities are normalized according to
	\[
	\rho(x)
	=
	\frac{I(x)}
	{\sum_{y\in\Omega}I(y)}.
	\]
	
	Thus each image belongs to the probability simplex
	\[
	\mathcal P(\Omega).
	\]
	
	%%%%%%%%%%%%%%%%%%%%%%%%%%%%%%%%%%%%%%%%%%%%%%%%%%%%%%%%%%%%%%%%%%%%%%
	\subsection{Interpolation Models}
	%%%%%%%%%%%%%%%%%%%%%%%%%%%%%%%%%%%%%%%%%%%%%%%%%%%%%%%%%%%%%%%%%%%%%%
	
	We compare three interpolation geometries.
	
	%%%%%%%%%%%%%%%%%%%%%%%%%%%%%%%%%%%%%%%%%%%%%%%%%%%%%%%%%%%%%%%%%%%%%%
	\subsubsection{Wasserstein interpolation}
	%%%%%%%%%%%%%%%%%%%%%%%%%%%%%%%%%%%%%%%%%%%%%%%%%%%%%%%%%%%%%%%%%%%%%%
	
	The quadratic Wasserstein distance is defined by
	\[
	W_2^2(\rho_0,\rho_1)
	=
	\inf_{\pi\in\Pi(\rho_0,\rho_1)}
	\int_{\Omega\times\Omega}
	|x-y|^2\,d\pi(x,y).
	\]
	
	Its dynamic formulation reads
	\[
	W_2^2(\rho_0,\rho_1)
	=
	\inf_{(\rho_t,v_t)}
	\int_0^1
	\int_\Omega
	|v_t|^2\rho_t\,dxdt,
	\]
	subject to
	\[
	\partial_t\rho_t
	+
	\nabla\cdot(\rho_t v_t)
	=
	0.
	\]
	
	The interpolation corresponds to pure spatial transport.
	
	%%%%%%%%%%%%%%%%%%%%%%%%%%%%%%%%%%%%%%%%%%%%%%%%%%%%%%%%%%%%%%%%%%%%%%
	\subsubsection{Aitchison interpolation}
	%%%%%%%%%%%%%%%%%%%%%%%%%%%%%%%%%%%%%%%%%%%%%%%%%%%%%%%%%%%%%%%%%%%%%%
	
	The compositional interpolation is defined by the geometric mean
	\[
	\rho_t(x)
	=
	\frac{
		\rho_0(x)^{1-t}
		\rho_1(x)^t
	}{
		\int_\Omega
		\rho_0^{1-t}\rho_1^t\,dx
	}.
	\]
	
	This interpolation acts purely in logarithmic-compositional space and ignores spatial geometry.
	
	%%%%%%%%%%%%%%%%%%%%%%%%%%%%%%%%%%%%%%%%%%%%%%%%%%%%%%%%%%%%%%%%%%%%%%
	\subsubsection{Hybrid interpolation}
	%%%%%%%%%%%%%%%%%%%%%%%%%%%%%%%%%%%%%%%%%%%%%%%%%%%%%%%%%%%%%%%%%%%%%%
	
	The proposed hybrid metric is defined by
	\[
	D_\alpha^2(\rho_0,\rho_1)
	=
	\inf_{(\rho_t,v_t,\zeta_t)}
	\int_0^1
	\int_\Omega
	\left(
	|v_t|^2
	+
	\alpha |\zeta_t|^2
	\right)\rho_t\,dxdt,
	\]
	subject to
	\[
	\partial_t\rho_t
	+
	\nabla\cdot(\rho_t v_t)
	=
	\rho_t
	\left(
	\zeta_t
	-
	\int_\Omega
	\zeta_t\rho_t\,dx
	\right).
	\]
	
	The parameter $\alpha>0$ controls the balance between:
	\begin{itemize}
		\item transport effects,
		\item compositional modulation.
	\end{itemize}
	
	%%%%%%%%%%%%%%%%%%%%%%%%%%%%%%%%%%%%%%%%%%%%%%%%%%%%%%%%%%%%%%%%%%%%%%
	\subsection{Numerical Scheme}
	%%%%%%%%%%%%%%%%%%%%%%%%%%%%%%%%%%%%%%%%%%%%%%%%%%%%%%%%%%%%%%%%%%%%%%
	
	The computational domain is discretized on a uniform grid
	\[
	N\times N,
	\qquad N=64.
	\]
	
	%%%%%%%%%%%%%%%%%%%%%%%%%%%%%%%%%%%%%%%%%%%%%%%%%%%%%%%%%%%%%%%%%%%%%%
	\subsubsection{Wasserstein approximation}
	%%%%%%%%%%%%%%%%%%%%%%%%%%%%%%%%%%%%%%%%%%%%%%%%%%%%%%%%%%%%%%%%%%%%%%
	
	The Wasserstein interpolation is approximated using an entropic Sinkhorn scheme combined with a barycentric projection.
	
	The regularized transport problem is
	\[
	W_{2,\varepsilon}^2(\rho_0,\rho_1)
	=
	\min_\pi
	\left\{
	\int |x-y|^2\,d\pi
	+
	\varepsilon
	\mathrm{KL}
	(\pi|\rho_0\otimes\rho_1)
	\right\}.
	\]
	
	%%%%%%%%%%%%%%%%%%%%%%%%%%%%%%%%%%%%%%%%%%%%%%%%%%%%%%%%%%%%%%%%%%%%%%
	\subsubsection{Hybrid splitting scheme}
	%%%%%%%%%%%%%%%%%%%%%%%%%%%%%%%%%%%%%%%%%%%%%%%%%%%%%%%%%%%%%%%%%%%%%%
	
	Starting from $\rho^{(0)}=\rho_0$, we iteratively compute:
	\begin{enumerate}
		
		\item a transport contribution approximated by a smoothed diffusion step,
		
		\item a logarithmic compositional correction
		\[
		u_k
		=
		\log(\rho_1+\varepsilon)
		-
		\log(\rho^{(k)}+\varepsilon),
		\]
		
		\item the update
		\[
		\rho^{(k+1)}
		=
		\rho^{(k)}
		+
		\tau
		\left[
		(1-\alpha)\Delta\rho^{(k)}
		+
		\alpha\rho^{(k)}u_k
		\right].
		\]
		
	\end{enumerate}
	
	After each iteration:
	\[
	\rho^{(k+1)}
	\leftarrow
	\frac{
		\max(\rho^{(k+1)},\varepsilon)
	}{
		\int_\Omega
		\max(\rho^{(k+1)},\varepsilon)\,dx
	}.
	\]
	
	%%%%%%%%%%%%%%%%%%%%%%%%%%%%%%%%%%%%%%%%%%%%%%%%%%%%%%%%%%%%%%%%%%%%%%
	\subsection{Synthetic Gaussian Experiment}
	%%%%%%%%%%%%%%%%%%%%%%%%%%%%%%%%%%%%%%%%%%%%%%%%%%%%%%%%%%%%%%%%%%%%%%
	
	We first consider two Gaussian blobs:
	\[
	\rho_0
	=
	G_{(0.3,0.5)},
	\qquad
	\rho_1
	=
	G_{(0.7,0.5)},
	\]
	where
	\[
	G_m(x)
	=
	\frac1{2\pi\sigma^2}
	\exp
	\left(
	-\frac{|x-m|^2}{2\sigma^2}
	\right).
	\]
	
	Figure~\ref{fig:gaussian} compares the three interpolation geometries.
	
	\begin{figure}[H]
		\centering
		\includegraphics[width=\textwidth]{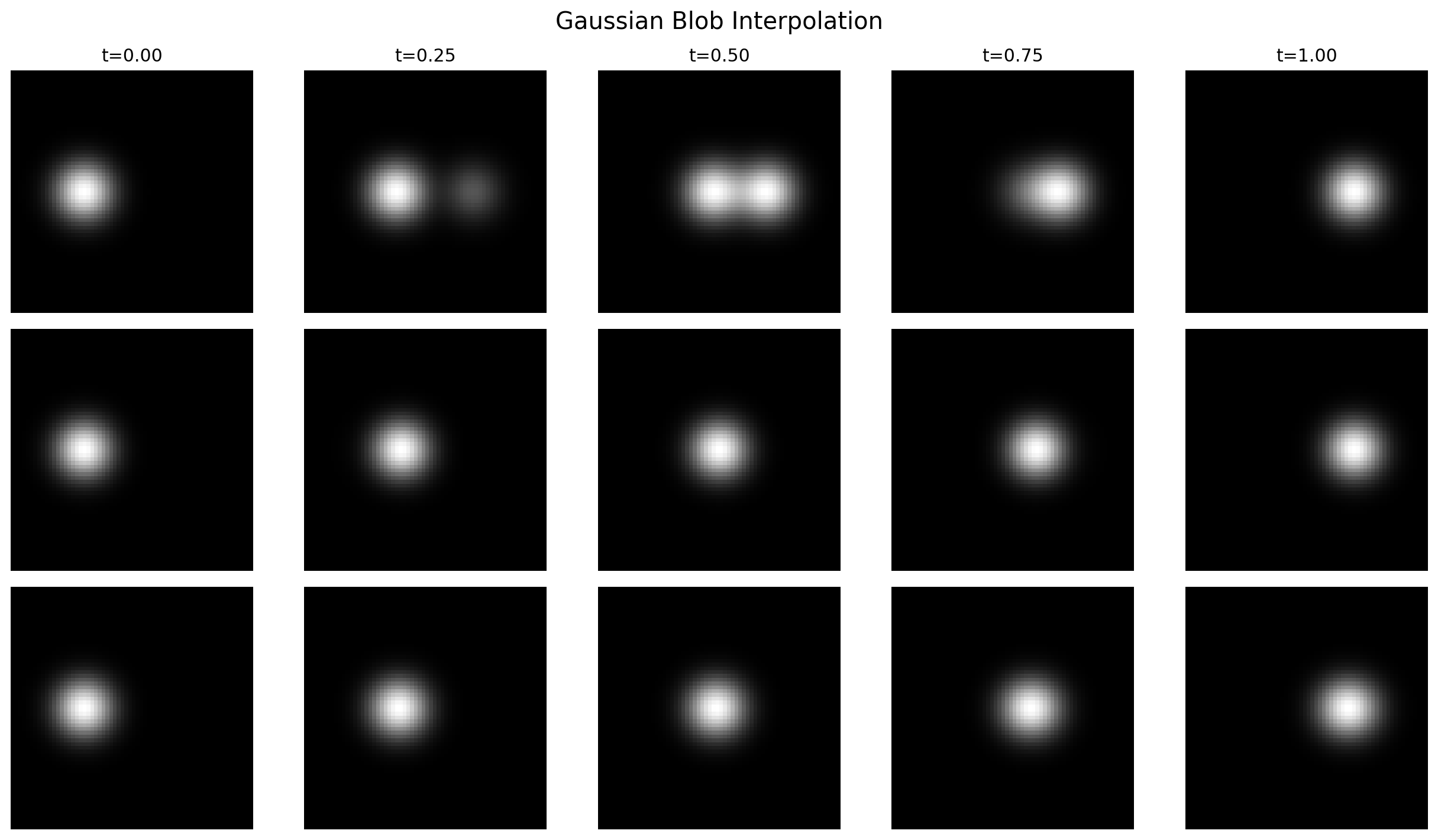}
		\caption{Interpolation between two Gaussian blobs using Wasserstein, Aitchison and hybrid geometries.}
		\label{fig:gaussian}
	\end{figure}
	
	%%%%%%%%%%%%%%%%%%%%%%%%%%%%%%%%%%%%%%%%%%%%%%%%%%%%%%%%%%%%%%%%%%%%%%
	\paragraph{Wasserstein interpolation}
	%%%%%%%%%%%%%%%%%%%%%%%%%%%%%%%%%%%%%%%%%%%%%%%%%%%%%%%%%%%%%%%%%%%%%%
	
	The Gaussian blob moves rigidly from left to right while preserving its shape. The interpolation captures spatial displacement but ignores relative intensity modulation.
	
	%%%%%%%%%%%%%%%%%%%%%%%%%%%%%%%%%%%%%%%%%%%%%%%%%%%%%%%%%%%%%%%%%%%%%%
	\paragraph{Aitchison interpolation}
	%%%%%%%%%%%%%%%%%%%%%%%%%%%%%%%%%%%%%%%%%%%%%%%%%%%%%%%%%%%%%%%%%%%%%%
	
	The interpolation behaves as a logarithmic fade-in/fade-out effect. Both blobs coexist simultaneously without spatial transport.
	
	%%%%%%%%%%%%%%%%%%%%%%%%%%%%%%%%%%%%%%%%%%%%%%%%%%%%%%%%%%%%%%%%%%%%%%
	\paragraph{Hybrid interpolation}
	%%%%%%%%%%%%%%%%%%%%%%%%%%%%%%%%%%%%%%%%%%%%%%%%%%%%%%%%%%%%%%%%%%%%%%
	
	The hybrid geometry simultaneously:
	\begin{itemize}
		\item transports the blob,
		\item and redistributes relative intensities.
	\end{itemize}
	
	Intermediate states exhibit partially transported bimodal structures.
	
	%%%%%%%%%%%%%%%%%%%%%%%%%%%%%%%%%%%%%%%%%%%%%%%%%%%%%%%%%%%%%%%%%%%%%%
	\subsection{Shape Morphing}
	%%%%%%%%%%%%%%%%%%%%%%%%%%%%%%%%%%%%%%%%%%%%%%%%%%%%%%%%%%%%%%%%%%%%%%
	
	We next consider interpolation between:
	\begin{itemize}
		\item a square density,
		\item and a circular density.
	\end{itemize}
	
	\begin{figure}[H]
		\centering
		\includegraphics[width=\textwidth]{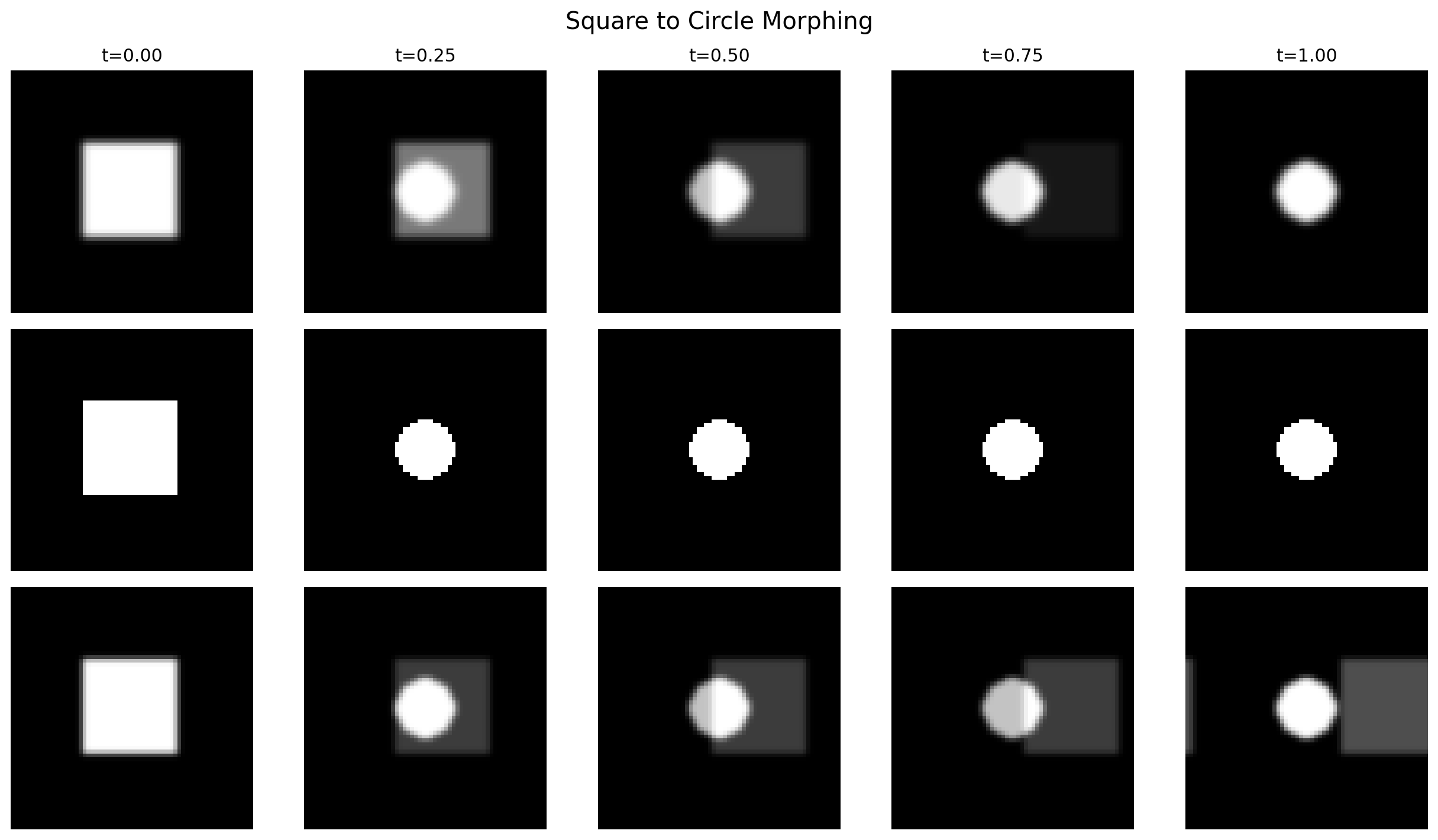}
		\caption{Square-to-circle morphing using Wasserstein, Aitchison and hybrid interpolations.}
		\label{fig:shape}
	\end{figure}
	
	%%%%%%%%%%%%%%%%%%%%%%%%%%%%%%%%%%%%%%%%%%%%%%%%%%%%%%%%%%%%%%%%%%%%%%
	\paragraph{Wasserstein interpolation}
	%%%%%%%%%%%%%%%%%%%%%%%%%%%%%%%%%%%%%%%%%%%%%%%%%%%%%%%%%%%%%%%%%%%%%%
	
	The corners of the square are transported toward the boundary of the disk through a geometric deformation map.
	
	%%%%%%%%%%%%%%%%%%%%%%%%%%%%%%%%%%%%%%%%%%%%%%%%%%%%%%%%%%%%%%%%%%%%%%
	\paragraph{Aitchison interpolation}
	%%%%%%%%%%%%%%%%%%%%%%%%%%%%%%%%%%%%%%%%%%%%%%%%%%%%%%%%%%%%%%%%%%%%%%
	
	The intermediate densities correspond to blurry superpositions of the square and the circle.
	
	%%%%%%%%%%%%%%%%%%%%%%%%%%%%%%%%%%%%%%%%%%%%%%%%%%%%%%%%%%%%%%%%%%%%%%
	\paragraph{Hybrid interpolation}
	%%%%%%%%%%%%%%%%%%%%%%%%%%%%%%%%%%%%%%%%%%%%%%%%%%%%%%%%%%%%%%%%%%%%%%
	
	The hybrid geometry progressively rounds the square while preserving interior coherence and reducing transport artifacts.
	
	%%%%%%%%%%%%%%%%%%%%%%%%%%%%%%%%%%%%%%%%%%%%%%%%%%%%%%%%%%%%%%%%%%%%%%
	\subsection{MNIST-like Digit Interpolation}
	%%%%%%%%%%%%%%%%%%%%%%%%%%%%%%%%%%%%%%%%%%%%%%%%%%%%%%%%%%%%%%%%%%%%%%
	
	We finally consider interpolation between synthetic MNIST-like digits:
	\[
	3
	\longrightarrow
	8.
	\]
	
	\begin{figure}[H]
		\centering
		\includegraphics[width=\textwidth]{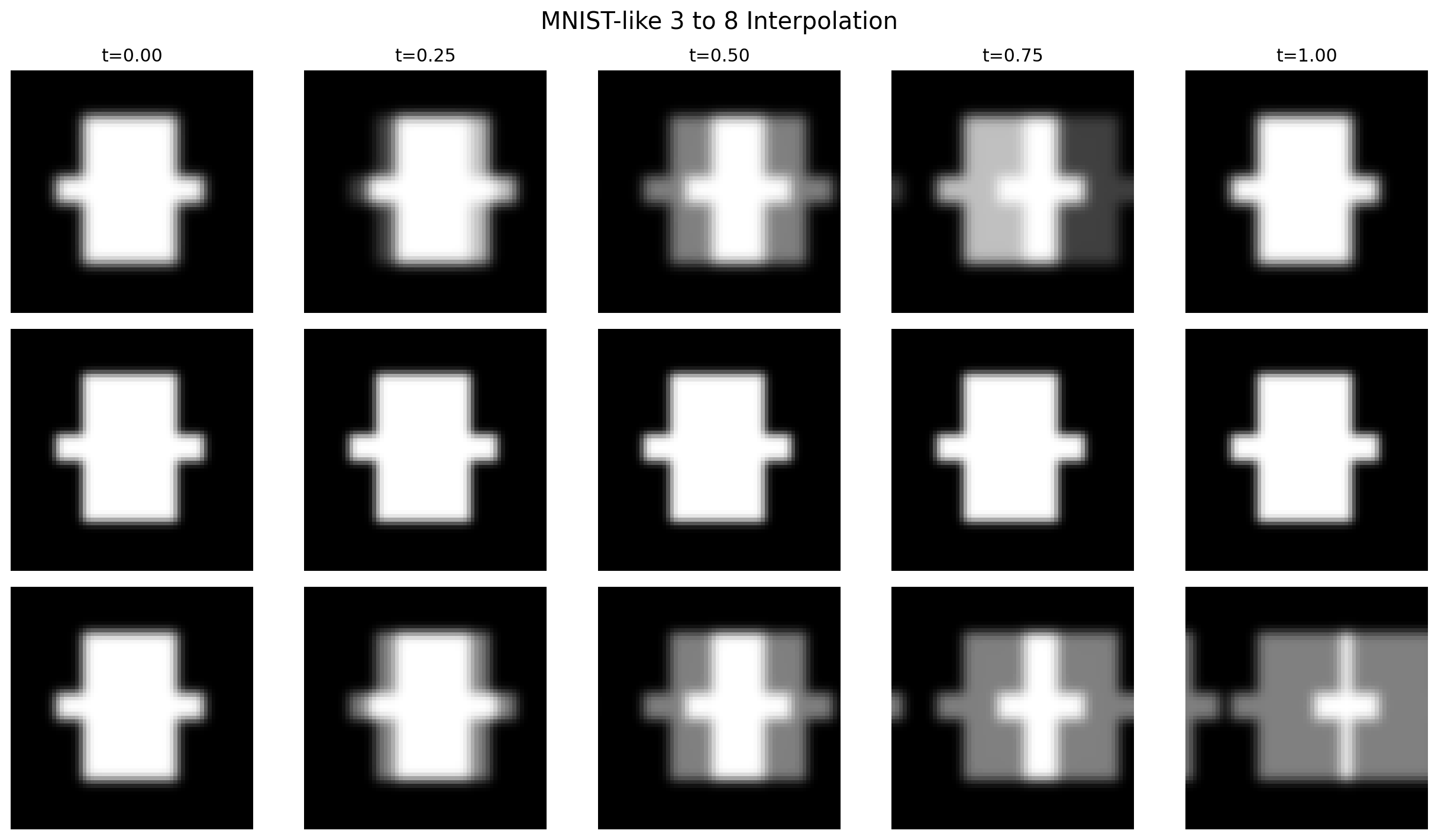}
		\caption{Interpolation between MNIST-like digits using Wasserstein, Aitchison and hybrid geometries.}
		\label{fig:mnist}
	\end{figure}
	
	%%%%%%%%%%%%%%%%%%%%%%%%%%%%%%%%%%%%%%%%%%%%%%%%%%%%%%%%%%%%%%%%%%%%%%
	\paragraph{Wasserstein interpolation}
	%%%%%%%%%%%%%%%%%%%%%%%%%%%%%%%%%%%%%%%%%%%%%%%%%%%%%%%%%%%%%%%%%%%%%%
	
	The digit deformation is mainly governed by spatial transport of strokes, which may create geometric artifacts.
	
	%%%%%%%%%%%%%%%%%%%%%%%%%%%%%%%%%%%%%%%%%%%%%%%%%%%%%%%%%%%%%%%%%%%%%%
	\paragraph{Aitchison interpolation}
	%%%%%%%%%%%%%%%%%%%%%%%%%%%%%%%%%%%%%%%%%%%%%%%%%%%%%%%%%%%%%%%%%%%%%%
	
	The interpolation produces a logarithmic superposition of the two digits and generates blurry intermediate states.
	
	%%%%%%%%%%%%%%%%%%%%%%%%%%%%%%%%%%%%%%%%%%%%%%%%%%%%%%%%%%%%%%%%%%%%%%
	\paragraph{Hybrid interpolation}
	%%%%%%%%%%%%%%%%%%%%%%%%%%%%%%%%%%%%%%%%%%%%%%%%%%%%%%%%%%%%%%%%%%%%%%
	
	The hybrid geometry simultaneously transports and recomposes the digit structure, leading to smoother and visually more realistic transformations.
	
	%%%%%%%%%%%%%%%%%%%%%%%%%%%%%%%%%%%%%%%%%%%%%%%%%%%%%%%%%%%%%%%%%%%%%%
	\subsection{Python Implementation}
	%%%%%%%%%%%%%%%%%%%%%%%%%%%%%%%%%%%%%%%%%%%%%%%%%%%%%%%%%%%%%%%%%%%%%%
	
	The experiments were implemented in Python using:
	\begin{itemize}
		\item NumPy,
		\item SciPy,
		\item POT (Python Optimal Transport),
		\item Matplotlib.
	\end{itemize}
	
	The corrected implementation includes:
	\begin{itemize}
		\item[i)] stable logarithmic normalization,
		\item[ii)] barycentric Wasserstein approximation,
		\item [iii)]diffusion-based transport regularization,
		\item [iv)]positivity-preserving renormalization.
	\end{itemize}
	
	%%%%%%%%%%%%%%%%%%%%%%%%%%%%%%%%%%%%%%%%%%%%%%%%%%%%%%%%%%%%%%%%%%%%%%
	\subsection{Discussion}
	%%%%%%%%%%%%%%%%%%%%%%%%%%%%%%%%%%%%%%%%%%%%%%%%%%%%%%%%%%%%%%%%%%%%%%
	
	The experiments demonstrate that the proposed hybrid geometry provides a meaningful compromise between:
	spatial transport and compositional interpolation.\\
	%\[
	%\text{spatial transport}
	%\qquad\text{and}\qquad
	%\text{compositional interpolation}.
	%\]
	Compared to classical Wasserstein interpolation, the hybrid geometry better captures:
	\begin{itemize}
		\item relative contrast modulation,
		\item multimodal intermediate structures,
		\item smooth geometric transitions.
	\end{itemize}
	Compared to pure Aitchison interpolation, it additionally preserves:
	\begin{itemize}
		\item spatial coherence,
		\item geometric deformation,
		\item structural consistency.
	\end{itemize}
	
	%%%%%%%%%%%%%%%%%%%%%%%%%%%%%%%%%%%%%%%%%%%%%%%%%%%%%%%%%%%%%%%%%%%%%%
	%\subsection{Future Perspectives}
	%%%%%%%%%%%%%%%%%%%%%%%%%%%%%%%%%%%%%%%%%%%%%%%%%%%%%%%%%%%%%%%%%%%%%%

	\section{Future perpectives and Open problems}
	\label{sec:perspectives}
	\subsection{Future perspectives}
	Future work includes:
	\begin{itemize}
		\item rigorous convergence analysis,
		\item hybrid Sinkhorn algorithms,
		\item variational PDE solvers,
		\item barycenter computation,
		\item applications to medical imaging,
		\item and probabilistic segmentation.
	\end{itemize}
	\subsection{Open problems}
	\begin{itemize}
		\item \textbf{Curvature bounds}: Is $(\PP_{2,\mathrm{ac}},D_\alpha)$ a $CD(K,\infty)$ space in the sense of Lott--Sturm--Villani? The interplay between Wasserstein and Aitchison might produce a curvature that depends on $\alpha$.
		\item \textbf{Rigorous EDI/EDE}: For $\lambda$-convex functionals, one can prove the Energy Dissipation Equality (EDE) under additional regularity assumptions. This remains to be done for the hybrid geometry.
		\item \textbf{Hybrid Schr\"odinger bridges}: Replace the action by an entropic regularization to obtain hybrid Schr\"odinger bridges, which would describe the most probable path under noise.
		\item \textbf{Numerical analysis}: Implement and analyze the JKO scheme for hybrid barycenters, using Sinkhorn for the Wasserstein part and explicit updates for the Aitchison part.
		\item \textbf{Extension to varying total mass}: Combine our geometry with WFR to allow both mass-conserving reactions (centered) and mass-changing reactions (non-centered).
	\end{itemize}

	\appendix

\end{document}